\documentclass[reqno,12pt]{amsart}
\usepackage{amsmath,amsfonts,amssymb,amsthm}
\allowdisplaybreaks
\usepackage{graphics}
\usepackage{epsfig}
\usepackage[usenames,dvipsnames]{color}
\usepackage{tikz}
\usepackage{mathtools}
\usepackage[colorlinks=true,linkcolor=black,menucolor=black,citecolor=black]{hyperref}
\usepackage[margin = .75in]{geometry}
\usepackage{cite}
\usepackage{enumerate}
\usepackage{soul}

\renewcommand{\d}{\ensuremath{\mathrm{d}}}
\newcommand{\R}{\ensuremath{\mathbb{R}}}

\newcommand{\N}{\ensuremath{\mathbb{N}}}

\newcommand{\eps}{\ensuremath{\varepsilon}}

\newcommand{\vertiii}[1]{{\left\lvert\kern-0.25ex\left\lvert\kern-0.25ex\left\lvert #1
    \right\rvert\kern-0.25ex\right\rvert\kern-0.25ex\right\rvert}}
\renewcommand{\d}{\ensuremath{{\rm d}}}

\xdefinecolor{lightblue}{RGB}{160,220,255}

\newtheorem{theorem}{Theorem}[section]

\newtheorem{proposition}[theorem]{Proposition}

\newtheorem{lemma}[theorem]{Lemma}

\numberwithin{equation}{section}
\makeatletter
\def\@setthanks{\def\thanks##1{\par##1}\thankses}
\makeatother

\begin{document}
%
\title[Strong solutions and stability for non-Newtonian thin fluid films]{Strong solutions and stability for a thin-film equation of shear-thinning fluids with contact line in partial wetting}
\thanks{%
\begin{tabular}{@{}p{0.875\textwidth}@{\hspace{1em}}c@{}}
The authors acknowledge the \emph{Lorentz Center} in Leiden for its hospitality during the workshop \emph{Nonlinear diffusion models: analytical \& numerical challenges} in January 2026, where fruitful discussions regarding this work took place. This publication is part of the project \emph{Codimension two free boundary problems} (with project number \emph{VI.Vidi.223.019} of the research program \emph{ENW Vidi}) which is financed by the \emph{Dutch Research Council} (\emph{NWO}). Both MVG and KN have received funding from this project during the preparation of this paper.
&
\raisebox{-.9\height}{\includegraphics[height=.85in]{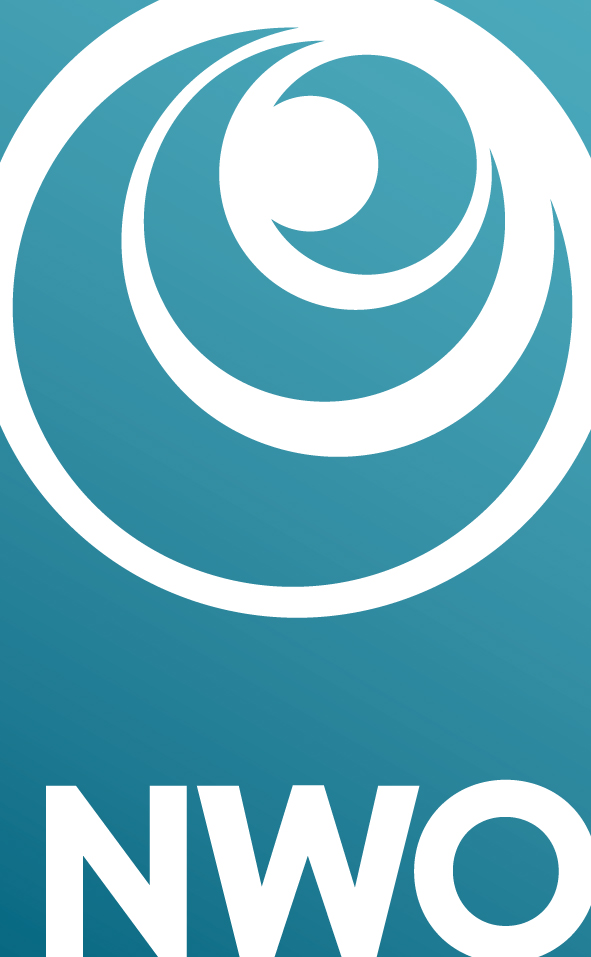}}
\end{tabular}%
}
\date{\today}
\author{Manuel~V.~Gnann}
\address[Manuel~V.~Gnann]{Delft Institute of Applied Mathematics, Faculty of Electrical Engineering, Mathematics and Computer Science, Delft University of Technology, Mekelweg 4, 2628CD Delft, Netherlands}
\email{M.V.Gnann@tudelft.nl}
\author{Christina~Lienstromberg}
\address[Christina~Lienstromberg]{Institute of Analysis, Dynamics and Modeling, University of Stuttgart, Pfaffenwaldring 57, 70569 stuttgart, Germany}
\email{christina.lienstromberg@mathematik.uni-stuttgart.de}
\author{Katerina~Nik}
\address[Katerina~Nik]{King Abdullah University of Science and Technology (KAUST), CEMSE Division, Thuwal 23955-
6900, Saudi Arabia}
\email{katerina.nik@kaust.edu.sa}
\begin{abstract}
We consider a power-law thin-film equation for strongly shear-thinning fluids. Weak solutions to this equation have been constructed more than twenty years ago by Ansini and Giacomelli. Here, we pass over to analyzing strong solutions with nonzero contact angle (partial-wetting regime), and place emphasis on studying the behavior of solutions near points where the film height vanishes (the contact-line region) by considering perturbations of a linear profile. The leading-order equation in von-Mises coordinates shows similarities with the evolution equation for the $p$-Laplace, though being of fourth order. Using a time discretization, we reduce the leading-order problem to finding a variational solution, and pass to the limit in the discretization scheme on suitably estimating higher-order nonlinear terms in conjunction with compactness arguments. This proves existence and asymptotic stability of strong solutions that are perturbations of the linear profile, and yields control on the contact-line velocity on carefully tracking singular terms in our estimates. While we believe that the transformed equation shows mathematical features the analysis of which stands on its own merit, it also physically corroborates shear thinning behavior as an alternative in resolving the no-slip paradox, as opposed to more standard approaches like introducing slip at the liquid-solid interface. 
\end{abstract}
\maketitle
\tableofcontents
\section{Introduction}
\subsection{Setting}
Consider the non-Newtonian thin-film equation
\begin{subequations}\label{free}
\begin{equation}\label{tfe}
h_t + \big(h^{\alpha+2} |h_{yyy}|^{\alpha-1} h_{yyy}\big)_y = 0 \quad \text{in } \{h > 0\},
\end{equation}
where the dependent variable $h = h(t,y)$ describes the height of an incompressible viscous thin film as a function of time $t \ge 0$ and lateral position (base point) $y \in \R$. The parameter $\alpha > 0$ is the so-called flow-behavior exponent which features the strain-dependent rheology of the fluid. More precisely, for $0 < \alpha < 1$, the corresponding fluid is shear-thickening, while it is Newtonian for $\alpha=1$ and shear-thinning for $\alpha > 1$. In this paper, we consider the strongly shear-thinning regime in which $\alpha > 2$. Equation~\eqref{tfe} can be derived through a lubrication approximation (small film heights $h$ at moderate Reynolds number) from the underlying (Navier-) Stokes system for a non-Newtonian power-law fluid with a no-slip condition at the substrate. For a formal derivation in case of the related Ellis law, see for instance \cite{WeidnerSchwartz1994}.

\bigskip

We solve \eqref{tfe} with the contact-angle condition
\begin{equation}\label{angle}
\lim_{\{h > 0\} \owns \tilde y \to y} |\partial_{\tilde y} h(t,\tilde y)| = 1 \quad \text{for } y \in \partial\{h > 0\},
\end{equation}
and the condition for the contact-line velocity
\begin{equation}\label{contact-vel}
\lim_{\{h > 0\} \owns \tilde y \to Y_0}
 h^{\alpha+1}(t,\tilde y) |\partial^3_{\tilde{y}} h(t,\tilde y)|^{\alpha -1} \partial^3_{\tilde{y}} h(t,\tilde y)=\frac{\d Y_0}{\d t}
\quad \text{for the contact lines } Y_0 \in \partial\{h > 0\},
\end{equation}
\end{subequations}
which implies no flux at the triple junction (the contact line), where liquid, air, and solid meet. A special solution to \eqref{free} is the linear profile
\begin{equation}\label{linear-profile}
h = y_+ \coloneqq \begin{cases} y & \text{ for } y > 0, \\ 0 & \text{ for } y \le 0. \end{cases}
\end{equation}
The present paper deals with perturbations of this linear profile \eqref{linear-profile}. Specifically, this enables us to precisely study the properties of \eqref{free} close to the contact line $\partial\{h > 0\}$, separating wetted from dry regions. Here, we prove that for $\alpha > 2$ strong solutions to \eqref{free} with moving free boundary exist, thus rigorously justifying shear-thinning behavior as an alternative to assuming slip at the liquid-solid interface in resolving the no-slip paradox. Indeed, for the Newtonian thin-film equation with no-slip condition at the substrate
\[
h_t + (h^3 h_{yyy})_y = 0 \quad \text{in } \{h > 0\},
\]
the contact line cannot move. See the first discussion of the no-slip paradox in \cite{HuhScriven1971} and the reviews \cite{deGennes1985,OronDavisBankoff1995,BonnEggers2009} for formal results from the physics literature. Its regularization on introducing (nonlinear) slip at the liquid-solid interface is given by
\begin{equation}\label{tfe-n}
h_t + (h^n h_{yyy})_y = 0 \quad \text{in } \{h > 0\},
\end{equation}
where $n \in [1,3)$. Then it is somewhat natural also from the mathematical perspective to ask for regularizations thereby incorporating the term $h_{yyy}$ in the flux. 
In this regard, from the physical perspective \eqref{tfe} turns out to be a natural regularization.
See \cite{WeidnerSchwartz1994,ansini_giacomelli_2002} for a discussion of the no-slip paradox in the case of Ellis' law.

\subsection{Transformation onto a fixed domain: von-Mises transform}
Next, we transform the free-boundary problem \eqref{free} onto a fixed domain. Therefore, we suppose $\{h > 0\} = (Y(t,0),\infty)$ and use the von-Mises transform (cf.~Figure~\ref{fig:mises})
\begin{equation}\label{mises}
h(t,Y(t,x)) = x.
\end{equation}
\begin{figure}[htp]
\centering
\begin{tikzpicture}[scale=1]
\path[fill=lightblue] (1.5,0) to [out=33.7,in=195] (9,3) -- (9,0) -- (1.5,0);

\draw [very thick,->] (-1.2,0) -- (10,0);
\draw [very thick,->] (0,-1.2) -- (0,5.2);

\draw[very thick,red,dashed] (1.5,0) to (9,5);
\draw[very thick,blue] (1.5,0) to [out=33.7,in=195] (9,3);

\draw [gray,dashed] (5.1,-.2) -- (5.1,2.4);
\draw [gray,dashed] (6.7,-.8) -- (6.7,2.4);
\draw [gray,dashed] (1.5,0) -- (1.5,-.8);
\draw [gray,dashed] (0,2.4) -- (6.7,2.4);

\draw [thick,blue,dashed,<->] (0,-.2) -- (5.1,-.2);
\draw [blue] (2.55,-.2) node[anchor=north] {$Y(t,x)$};

\draw [thick,red,dashed,<->] (1.5,-.8) -- (6.7,-.8);
\draw [red] (4.1,-.8) node[anchor=north] {$x$};

\draw (0,2.4) node[anchor=east] {$h(t,\color{blue}Y(t,x)\color{black}) = \color{red}x\color{black}$};

\draw (2,3) node[anchor=east] {gas};

\draw [blue] (8.5,1.5) node[anchor=east] {liquid};

\draw [thick,violet,->] (1.4,.8) -- (1.5,0);
\draw [violet] (1.4,.8) node[anchor=south] {triple junction};

\draw (0,4.9) node[anchor=east] {film height $h$};

\draw (8.7,0) node[anchor=north] {base point $x$ or $y$};

\end{tikzpicture}
\caption{Schematic plot of a liquid thin film and the von-Minses transform \eqref{mises}.}
\label{fig:mises}
\end{figure}
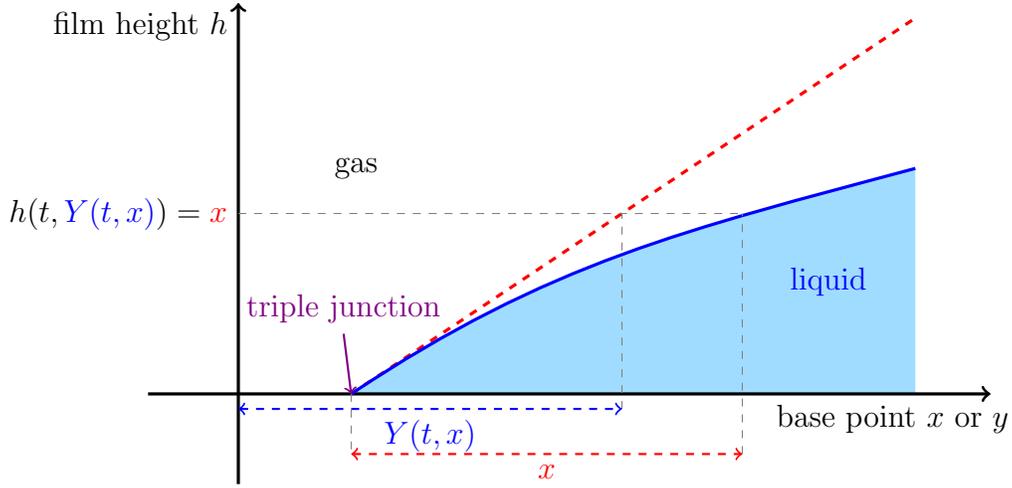
The transformation \eqref{mises} is well-defined provided $[0,\infty) \owns x \mapsto Y(t,x)$ is, for $t \ge 0$ fixed, a strictly increasing profile. We differentiate \eqref{mises} with respect to $t$ and $x$ and obtain
\begin{equation}\label{mises-der}
h_y(t,Y(t,x)) \, Y_x(t,x) = 1 \quad \text{and} \quad \partial_t h(t,Y(t,x)) + h_y(t,Y(t,x)) \, Y_t(t,x) = 0.
\end{equation}
Defining
\begin{equation}\label{f-mises}
F \coloneq Y_x^{-1}
\end{equation}
and using
\begin{equation}\label{trafo-der}
\partial_y = F \partial_x,
\end{equation}
we obtain with \eqref{mises} and \eqref{mises-der} that
\begin{equation}\label{hy-mises}
h_y(t,Y(t,x)) = F(t,x).
\end{equation}
Additionally using \eqref{tfe} in combination with \eqref{mises}, \eqref{mises-der}, \eqref{trafo-der}, and \eqref{hy-mises}, we find
\begin{equation}\label{y-t}
F \partial_t Y = F \partial_x \big( x^{\alpha+2} \, |F\partial_x(F\partial_xF)|^{\alpha-1} \, F\partial_x(F\partial_xF) \big) \quad \text{for } t,x > 0,
\end{equation}
so that with $F_t \stackrel{\eqref{f-mises}}{=} - F^2 Y_{xt}$ and noting that $\partial_x(F\partial_xF) = \frac 1 2 \partial_x^2 F^2$, we obtain
\begin{subequations}\label{problem-mises}
\begin{equation}\label{eq-mises}
F_t + 2^{-\alpha} \, F^2 \, \partial_x^2 \big( x^{\alpha+2} \, F^{\alpha} g_{\alpha}(\partial_x^2 F^2) \big) = 0 \quad \text{for } t,x > 0,
\end{equation}
where
\begin{align}
g_\alpha(v) &\coloneqq |v|^{\alpha-1} v. \label{def-galpha-intro}
\end{align}
With \eqref{angle} and \eqref{hy-mises}, the contact-angle condition translates into
\begin{equation}\label{angle-mises}
F = 1 \quad \text{at } x = 0.
\end{equation}
With \eqref{contact-vel} the condition for the contact-line velocity takes the form
\begin{equation}\label{contact-vel-mises}
\lim_{x\downarrow 0} \, x^{\alpha+1} g_\alpha(F\partial_x^2 F^2) = 2^\alpha \partial_t Y(t,0).
\end{equation}
\end{subequations}
\subsection{Leading-order equation}
Multiplying the evolution equation \eqref{eq-mises} by $F$, we have
\[
2^{\alpha-1} (F^2)_t + (F^2)^{\frac 3 2} \, \partial_x^2 \big( x^{\alpha+2} \, (F^2)^{\frac \alpha 2} g_{\alpha}(\partial_x^2 F^2) \big) = 0 \quad \text{for } t,x > 0,
\]
so that writing
\begin{equation}\label{def-u}
F^2 \eqqcolon 1 + u,
\end{equation}
and rescaling time according to
\begin{equation}\label{rescale-time}
t \longmapsto 2^{\alpha-1} t,
\end{equation}
we have
\begin{align*}
u_t + (1+u)^{\frac{3}{2}} \, \partial_x^2 \big( x^{\alpha+2} \, (1+u)^{\frac{\alpha}{2}}  g_{\alpha}(\partial_x^2 u) \big) = 0 \quad \text{for } t,x > 0.
\end{align*}
Therefore, equation~\eqref{eq-mises} attains the leading-order form
\begin{subequations}\label{problem-u}
\begin{equation}\label{eq-lead-mises}
u_t + \partial_x^2 \big(x^{\alpha+2} \, g_\alpha(\partial_x^2 u)\big) = f \quad \text{for } t,x > 0,
\end{equation}
where $f = N(u)$ contains all terms of order higher than $|u|^\alpha$:
\begin{align}
N(u) &\coloneqq (1-  (1+u)^{\frac{3}{2}}) \,  \partial_x^2 (x^{\alpha+2} \, g_{\alpha}(\partial_x^2 u))
+ (1+u)^{\frac{3}{2}}\, \partial_x^2\big( (1- (1+u)^{\frac{\alpha}{2}}) x^{\alpha+2} \, g_{\alpha}(\partial_x^2 u) \big), \label{def-nu}
\end{align}
From \eqref{angle-mises} and \eqref{def-u} we infer that the dependent variable $u$ additionally satisfies the boundary condition
\begin{equation}\label{bc-u}
u = 0 \quad \text{at } x = 0,
\end{equation}
%
and from \eqref{contact-vel-mises} we obtain that
\begin{equation}\label{bc-contact-vel}
\lim_{x\downarrow0} x^{\alpha+1} g_\alpha(\partial_x^2 u)
= 2 \partial_t Y(t,0).
\end{equation}
\end{subequations}
\subsection{Heuristics: leading-order estimates}
\label{sec:leading-order estimates}
On assuming that a sufficiently regular solution $u \colon (0,\infty)^2 \to \R$ to \eqref{problem-u} already exists, in this subsection we derive a-priori estimates for $u$.
To this end, we test \eqref{eq-lead-mises} with $u$ in space and, integrating by parts twice, we obtain
\begin{subequations}\label{est-lead-weak-max-1}
\begin{equation}\label{est-lead-weak-1}
\frac{1}{2} \frac{\d}{\d t} \int_0^\infty u^2 \, \d x + \int_0^\infty x^{\alpha+2} \, |\partial_x^2 u|^{\alpha+1} \, \d x = \int_0^\infty f \, u \, \d x.
\end{equation}
\bigskip

Next, we test \eqref{eq-lead-mises} with $\partial_x^2 (x^{\alpha+2} \, g_\alpha(\partial_x^2 u))$ and get with
\begin{align*}
\int_0^\infty u_t \, \partial_x^2 (x^{\alpha+2} \, g_\alpha(\partial_x^2 u)) \, \d x = \int_0^\infty x^{\alpha+2} \, (\partial_x^2 u_t) \, g_\alpha(\partial_x^2 u) \, \d x = \frac{1}{\alpha+1} \frac{\d}{\d t} \int_0^\infty x^{\alpha+2} \, |\partial_x^2 u|^{\alpha+1} \, \d x
\end{align*}
the identity
\[
\frac{1}{\alpha+1} \frac{\d}{\d t} \int_0^\infty x^{\alpha+2} |\partial_x^2 u|^{\alpha+1} \, \d x + \int_0^\infty \big|\partial_x^2 (x^{\alpha+2} \, g_\alpha(\partial_x^2 u))\big|^2 \, \d x = \int_0^\infty f \, \partial_x^2 (x^{\alpha+2} \, g_\alpha(\partial_x^2 u)) \, \d x,
\]
which by Young's inequality entails the estimate
\begin{equation}\label{est-lead-max-1}
\frac{2}{\alpha+1} \frac{\d}{\d t} \int_0^\infty x^{\alpha+2} |\partial_x^2 u|^{\alpha+1} \, \d x + \int_0^\infty \big|\partial_x^2 (x^{\alpha+2} \, g_\alpha(\partial_x^2 u))\big|^2 \, \d x \le \int_0^\infty f^2 \, \d x.
\end{equation}
\end{subequations}
In time-integrated form, the identity \eqref{est-lead-weak-1} entails
\begin{subequations}\label{est-lead-weak-max}
\begin{align}
\sup_{t \ge 0} \int_0^\infty u^2 \, \d x + 2 \int_0^\infty \int_0^\infty x^{\alpha+2} \, |\partial_x^2 u|^{\alpha+1} \, \d x \, \d t \le 2 \int_0^\infty u_0^2 \, \d x + 4 \int_0^\infty \int_0^\infty |f| \, |u| \, \d x \, \d t, \label{est-lead-weak-2}
\end{align}
where $u_0 \colon (0,\infty) \to \R$ is the initial data. Time integration of estimate~\eqref{est-lead-max-1} leads to
\begin{align}\nonumber
& \frac{2}{\alpha+1} \sup_{t \ge 0} \int_0^\infty x^{\alpha+2} |\partial_x^2 u|^{\alpha+1} \, \d x + \int_0^\infty \int_0^\infty \big|\partial_x^2 (x^{\alpha+2} \, g_\alpha(\partial_x^2 u))\big|^2 \, \d x \, \d t \\
& \quad \le \frac{4}{\alpha+1} \int_0^\infty x^{\alpha+2} |\partial_x^2 u_0|^{\alpha+1} \, \d x + 2 \int_0^\infty \int_0^\infty f^2 \, \d x \, \d t. \label{est-lead-max-2}
\end{align}
\end{subequations}
On noting that $\alpha > 1$, estimates~\eqref{est-lead-weak-max} suggest that $\sup_{t, x \ge 0} |u|$ is controlled by these inequalities in conjunction with an interpolation inequality. Indeed, for $p > 1$ we have, using the fundamental theorem of calculus and H\"older's inequality,
\begin{align*}
\sup_{x \ge 0} |u|^p &\le p \int_0^\infty |u|^{p-1} |\partial_x u| \, \d x = p \int_0^\infty \big(x^{- \frac{1}{\alpha+1}} |u|^{p-1}\big) \big(x^{\frac{1}{\alpha+1}} |\partial_x u|\big) \, \d x \\
&\le p \Big(\int_0^\infty x^{-\frac 1 \alpha} \, |u|^{\frac{(p-1) (\alpha+1)}{\alpha}} \, \d x\Big)^{\frac{\alpha}{\alpha+1}} \Big(\int_0^\infty x \, |\partial_x u|^{\alpha+1} \, \d x\Big)^{\frac{1}{\alpha+1}}.
\end{align*}
By Hardy's inequality, the second term on the right-hand side can be estimated as
\[
\int_0^\infty x \, |\partial_x u|^{\alpha+1} \, \d x \lesssim_\alpha \int_0^\infty x^{\alpha+2} \, |\partial_x^2 u|^{\alpha+1} \, \d x.
\]
Furthermore, for the first term we obtain by H\"older's inequality,
\begin{align*}
\int_0^\infty x^{-\frac 1 \alpha} \, |u|^{\frac{(p-1) (\alpha+1)}{\alpha}} \, \d x &= \int_0^\infty \big(x^{-\alpha} \, |u|^{\alpha+1}\big)^{\frac{1}{\alpha^2}} \, |u|^{\frac{(p-1) (\alpha+1)}{\alpha} - \frac{\alpha+1}{\alpha^2}} \, \d x \\
&\le \Big(\int_0^\infty x^{-\alpha} \, |u|^{\alpha+1} \, \d x\Big)^{\frac{1}{\alpha^2}} \Big(\int_0^\infty |u|^{\frac{p \alpha - \alpha - 1}{\alpha-1}} \, \d x\Big)^{\frac{\alpha^2-1}{\alpha^2}}.
\end{align*}
Applying Hardy's inequality twice yields
\[
\int_0^\infty x^{-\alpha} \, |u|^{\alpha+1} \, \d x \lesssim_\alpha \int_0^\infty x^{\alpha+2} \, |\partial_x^2 u|^{\alpha+1} \, \d x.
\]
Furthermore, choosing $p \coloneq \frac{3\alpha-1}{\alpha} = 3 - \frac 1 \alpha$ yields $\frac{p \alpha - \alpha - 1}{\alpha-1} = 2$, so that we obtain
\[
\sup_{t, x \ge 0} |u| \lesssim_\alpha \Big(\sup_{t \ge 0} \int_0^\infty u^2 \, \d x\Big)^{\frac{\alpha-1}{3\alpha-1}} \Big(\sup_{t \ge 0} \int_0^\infty x^{\alpha+2} \, |\partial_x^2 u|^{\alpha+1} \, \d x\Big)^{\frac{1}{3\alpha-1}}.
\]
We will generalize this estimate to a weighted version in \eqref{eq:weight xu} of Lemma~\ref{lem:weight xu} below.
Smallness of $\sup_{t, x \ge 0} |u|$ implies by \eqref{def-u} that $\sup_{t,x \ge 0} |F-1|$ is small, which by \eqref{f-mises} entails that $\sup_{t,x \ge 0} |Y_x-1|$ is small, and thus by \eqref{mises-der} that the von-Mises transform \eqref{mises} is well-defined provided $\sup_{t, x \ge 0} |u|$ is sufficiently small. Hence, we anticipate that with \eqref{est-lead-weak-max} we obtain existence of stable strong solutions $u$ of perturbations of $u = 0$, the latter corresponding to the linear profile \eqref{linear-profile} in the original variables.

\subsection{Heuristics on the kernel and singular expansion: the restriction $\alpha > 2$}
\label{sec:kernel}
We compute the solution space to
\[
\partial_x^2 \big(x^{\alpha+2} \, g_\alpha(\partial_x^2 u)\big) = 0 \quad \text{for } x > 0,
\]
where $u$ satisfies $u = 0$ at $x = 0$. Two integrations yield (with $a, b \in \R$ generic constants)
\[
x^{\alpha+2} \, g_\alpha(\partial_x^2 u) = a x + b \quad \Longleftrightarrow \quad g_\alpha(\partial_x^2 u) = a x^{-\alpha-1} + b x^{-\alpha-2}.
\]
A finite contact-line velocity according to \eqref{bc-contact-vel} implies $b = 0$, but in order for it to be non-trivial, we must allow for $a \ne 0$. Hence, if $a \ne 0$ we get with \eqref{def-galpha-intro} on re-defining $a$,
\[
\partial_x^2 u = a x^{-1-\frac 1 \alpha}.
\]
Integrating twice and using $u = 0$ at $x = 0$ yields with generic constants $a,b \in \R$,
\[
u =  a x^\beta + b x \quad \text{with } \beta \coloneqq \frac{\alpha-1}{\alpha}.
\]
Since the nonlinearity $f = N(u)$ (cf.~\eqref{def-nu}) in equation \eqref{eq-lead-mises} mixes $x^\beta$ and $x$, we then expect that general solutions to \eqref{problem-mises} will have an expansion of the form
\[
F^2 = 1 + u = 1 + \sum_{j + k \ge 1} a_{j+\beta k}(t) x^{j+\beta k} \quad \text{as } x \downarrow 0,
\]
with generic constants $a_{j+\beta k}(t)$, at least when $\beta$ is irrational (otherwise, resonances may appear, leading to logarithmic corrections in the expansion). We now use
\[
\partial_x^2 u = a x^{\beta-2} + b x^{2\beta-2} + o(x^{2\beta-2}) \quad \text{as } x \downarrow 0,
\]
with $a,b \in \R$ being generic constants again,
so that
\[
g_\alpha(\partial_x^2 u) = a x^{-\alpha-1} + b x^{\beta-\alpha-1} + o(x^{\beta-\alpha-1}) \quad \text{as } x \downarrow 0.
\]
Thus
\[
x^{\alpha+2} \, g_\alpha(\partial_x^2 u) = a x + b x^{\beta+1} + o(x^{\beta+1}) \quad \text{as } x \downarrow 0,
\]
i.e.,
\[
\partial_x^2 \big(x^{\alpha+2} \, g_\alpha(\partial_x^2 u)\big) = a x^{\beta-1} + o(x^{\beta-1}) \quad \text{as } x \downarrow 0.
\]
We thus deduce the leading-order terms for $0 < \eps \ll 1$,
\[
\int_0^\eps u^2 \, \d x \sim \int_0^\eps x^{2\beta} \, \d x < \infty, \quad \int_0^\eps x^{\alpha+2} \, |\partial_x^2 u|^{\alpha+1} \, \d x \sim \int_0^\eps x^{- \frac 1 \alpha} \, \d x < \infty \quad \text{since } \alpha > 1,
\]
and
\[
\int_0^\eps \big(\partial_x^2 (x^{\alpha+2} \, g_\alpha(\partial_x^2 u))\big)^2 \, \d x \sim \int_0^\eps x^{2\beta-2} \, \d x = \int_0^\eps x^{-\frac 2 \alpha} \, \d x < \infty,
\]
provided $\alpha > 2$. This explains why we will restrict our result to the strongly shear-thinning regime in which $\alpha > 2$. This constraint appears for instance in the proof of Lemma~\ref{lem:con-lem-sup-uxx} below.

\section{Main results and discussion}
\subsection{Main results}
Motivated by the considerations in \S\ref{sec:leading-order estimates}, we define the linear space for the initial datum as
\begin{subequations}\label{def-ualpha-space-norm}
\begin{equation}\label{def-ualpha}
U_\alpha\coloneq \big\{u \in L^2(0,\infty) \colon x^{\frac{\alpha+2}{\alpha+1}} \partial_x^2 u \in L^{\alpha+1}(0,\infty), \; u = 0 \text{ at } x = 0\big\}.
\end{equation}
The corresponding norm is given by
\begin{equation}
\|u\|_{U_\alpha} \coloneq \|u\|_{L^2(0,\infty)} + \big\|x^{\frac{\alpha+2}{\alpha+1}} \partial_x^2 u\big\|_{L^{\alpha+1}(0,\infty)}.
\end{equation}
\end{subequations}
By Lemma~\ref{lem-approx-ualpha} below the test functions $C_\mathrm{c}^\infty((0,\infty))$ are dense in $U_\alpha$ and by estimate~\eqref{eq:weight xu} of Lemma~\ref{lem:weight xu} the boundary condition $u = 0$ at $x = 0$ is well-defined. We have the following existence result:
\begin{theorem}[Existence and a-priori bounds]\label{th:ex}
Suppose $\alpha > 2$. For initial data $u_0$ such that $\|u_0\|_{U_\alpha}$ is sufficiently small, there exists a 
locally integrable $u \colon [0,\infty)^2 \to \R$ such that
the a-priori estimate
\begin{align}\nonumber
& \sup_{t \ge 0} \int_0^\infty \big(u^2 + x^{\alpha+2} |\partial_x^2 u|^{\alpha+1}\big) \, \d x + \int_0^\infty \int_0^\infty \big((\partial_t u)^2 + x^{\alpha+2} |\partial_x^2 u|^{\alpha+1} + (\partial_x^2 (x^{\alpha+2} g_\alpha(\partial_x^2 u)))^2\big) \, \d x \, \d t \nonumber\\
& \quad \le C \int_0^\infty \big(u_0^2 + x^{\alpha+2} |\partial_x^2 u_0|^{\alpha+1}\big) \, \d x \label{apriori-u}
\end{align}
is satisfied for some $C < \infty$, $u = 0$ at $x = 0$, the strong formulation
\begin{equation}\label{eq-strong}
\partial_t u + \partial_x^2\big(x^{\alpha+2} g_\alpha(\partial_x^2 u)\big) = N(u) \quad \text{almost everywhere in } (0,\infty)^2
\end{equation}
holds true, and $u(0,\cdot) = u_0$ is attained in $U_\alpha$.
This implies stability of the linear profile $u = 0$ (cf.~\eqref{linear-profile}, \eqref{mises}, \eqref{f-mises}, \eqref{def-u}). Additionally, we have H\"older continuity in time, that is, $u \in C^{\frac 1 2}([0,\infty);L^2(0,\infty))$, and continuous in $U_\alpha$, that is, $u \in C^0([0,\infty);U_\alpha)$. Furthermore, we have the a-priori estimate
\begin{align}\nonumber
& \Big(\sup_{t, x \ge 0} x^{\beta_1} |u|\Big)^{\frac{3 \alpha-1}{(\alpha-1) (\beta_1+1) + 1}} + \int_0^\infty \Big(\sup_{x \ge 0} x^{\beta_1'} |u - c \delta_1 x^{\frac{\alpha-1}{\alpha}}|\Big)^{\frac{3\alpha-1}{1-\alpha-(\alpha+1) \beta_1'}} \, \d t + \Big(\sup_{t, x \ge 0} x^{\beta_2} |\partial_x u|\Big)^{\frac{3 \alpha-1}{(\alpha-1) \beta_2 + 1}} \\
& + \int_0^\infty \Big(\sup_{x \ge 0} x^{\beta_2'} |\partial_x (u - c \delta_2 x^{\frac{\alpha-1}{\alpha}})|\Big)^{\frac{3\alpha-1}{2-(\alpha+1)\beta_2'}} \, \d t + \int_0^\infty \Big(\sup_{x \ge 0} x^{\beta_3} |\partial_x^2 (u- c \delta_3 x^{\frac{\alpha-1}{\alpha}})|\Big)^{\frac{3\alpha-1}{\alpha+3 - (\alpha+1)\beta_3}} \, \d t \nonumber \\
& \quad \le C \int_0^\infty \big(u_0^2 + x^{\alpha+2} |\partial_x^2 u_0|^{\alpha+1}\big) \, \d x, \label{embed-main}
\end{align}
for some $C < \infty$, where $c \in \R$ is such that $\partial_x w(0) = g_\alpha((1-\alpha) c/\alpha^2)$, and
\begin{align*}
\beta_1 \in \Big[\frac{1-\alpha}{\alpha+1},\frac{\alpha+2}{5\alpha+3}\Big), \quad \beta_1' \in \Big[\frac{1-2\alpha}{2\alpha},\frac{1-\alpha}{\alpha+1}\Big), \quad \beta_2 \in \Big[\frac{2}{\alpha+1},\frac{3(\alpha+2)}{5\alpha +3}\Big), \quad \beta_2' \in \Big[\frac{1}{2\alpha}, \frac{2}{\alpha+1}\Big), \\
\beta_3 \in \Big[\frac{1-2\alpha}{2\alpha},\frac{1-\alpha}{\alpha+1}\Big], \quad \delta_1 \coloneq \begin{cases} 1 & \text{if } \beta_1' < \frac{1-\alpha}{\alpha}, \\ 0 & \text{if } \beta_1' \ge \frac{1-\alpha}{\alpha}, \end{cases} \quad \delta_2 \coloneq \begin{cases} 1 & \text{if } \beta_2' < \frac 1 \alpha, \\ 0 & \text{if } \beta_2' \ge \frac 1 \alpha, \end{cases} \quad \delta_3 \coloneq \begin{cases} 1 & \text{if } \beta_3 < \frac{\alpha+1}{\alpha}, \\ 0 & \text{if }  \beta_3 \ge \frac{\alpha+1}{\alpha}. \end{cases}
\end{align*}
Lastly, we have asymptotic stability according to 
\begin{subequations}
\begin{align}\label{decay-main}
\int_0^\infty
x^{\alpha+2} |\partial_x^2 u|^{\alpha+1}
\, \d x &= o(t^{-1}), \\
\sup_{x \ge 0} x^{\beta_1} |u| &=
o(t^{\frac{2\beta_1-1}{3\alpha-1}}),
\label{decay-xbu}\\
\sup_{x \ge 0} x^{\beta_2} |\partial_x u| &= o(t^{\frac{2\beta_2-3}{3\alpha-1}}),
\label{decay-xbu'}
\end{align}
\end{subequations}
as $t \to \infty$.
\end{theorem}
Notice that in view of the rescaling of time according to \eqref{rescale-time}, the vertically-averaged horizontal velocity of the fluid film is given by
\[
V \coloneq 2^{\alpha-1} \partial_t Y \stackrel{\eqref{y-t}}{=} 2^{\alpha-1} h^{\alpha+1} |h_{yyy}|^{\alpha-1} h_{yyy},
\]
which is in line with viewing \eqref{tfe} as a continuity equation $\partial_t h + \partial_y (h V) = 0$ in $\{h > 0\}$. Using \eqref{mises}, \eqref{trafo-der}, \eqref{hy-mises}, and \eqref{def-u} leads to
\begin{equation}\label{velocity}
V = \frac 1 2 x^{\alpha+1} (F^2)^{\frac\alpha 2} g_\alpha(\partial_x^2 F^2) = \frac 1 2 x^{\alpha+1} (1+u)^{\frac\alpha 2} g_\alpha(\partial_x^2 u).
\end{equation}
We have the following a-priori bound for the velocity $V$:
\begin{theorem}[A-priori estimate for the velocity]\label{th-velocity}
In the situation of Theorem~\ref{th:ex}, there exists a constant $C < \infty$ such that
\begin{equation}\label{a-priori-contact-vel}
\|V\|_{L^2(0,\infty;BC^0([0,\infty)))}^2 = \int_0^\infty \sup_{x \ge 0} |V(t,x)|^2 \, \d t \le C \int_0^\infty \big(u_0^2 + x^{\alpha+2} |\partial_x^2 u_0|^{\alpha+1}\big) \, \d x.
\end{equation}
\end{theorem}
\subsection{Discussion}
Theorem~\ref{th-velocity} implies in particular that the contact-line velocity
\[
V_{|x = 0} = (\partial_t Y)(\cdot,0) \stackrel{\eqref{bc-contact-vel}}{=} \frac 1 2 \lim_{x\downarrow0} x^{\alpha+1} g_\alpha(\partial_x^2 u)
\]
is controlled in $L^2(0,\infty)$. Specifically, we note that a nonzero contact-line velocity requires that $\lim_{x \downarrow 0} x^{\frac{\alpha+1}{\alpha}} (\partial_x^2 u)(t,x) \ne 0$, which in view of \eqref{def-ualpha-space-norm} is an admissible choice for the initial data $u_0$. We have thus indeed found a solution concept, for which nonzero contact-line velocities are admissible, occur, and are controlled, and have thus rigorously resolved the no-slip paradox for the free-boundary problem \eqref{free}. 

\bigskip

This sets our work apart from previous results like \cite{AnsiniGiacomelli2004} by Ansini and Giacomelli, in which for the first time weak solutions and their qualitative properties are derived for the more general class of PDEs
\[
h_t + (h^n |h_{yyy}|^{\alpha-1} h_{yyy})_y = 0 \quad \text{in } \{h > 0\},
\]
where $n \ge 0$ and $\alpha \ge 1$ (which reduces to the no-slip case \eqref{tfe} on setting $n = \alpha+2$). There, it is proved that the speed of propagation is finite (cf.~\cite[Theorem~2]{AnsiniGiacomelli2004} with a propagation rate matching the self-similar scaling), and that $h^n |h_{yyy}|^{\alpha-1} h_{yyy} \in L^1(\{h > 0\})$ for the flux, which is insufficient for obtaining control on the contact-line velocity. The previous works \cite{King2001,King_2001} by King deal with formal asymptotics at the contact line and the analysis of special solutions (both in the shear-thickening and the shear-thinning case), so that our analysis can be viewed as a rigorous treatment of the contact-line asymptotics in \cite{King2001,King_2001}. In \cite{ansini_giacomelli_2002}, Ansini and Giacomelli use formal asymptotics near the contact line to show that the related Ellis thin-film equation
\begin{equation*}
    h_t + \bigl(h^3 (1 + |hh_{yyy}|^{\alpha-1}) h_{yyy})_y = 0 \quad \text{in } \{h > 0\},
\end{equation*}
admits advancing traveling-wave solutions, whereby they suggest shear-thinning behavior as a possible remedy for the no-slip paradox. Moreover, they derive a scaling law in time for macroscopic quantities by analyzing a class of quasi-self-similar solutions to the Ellis thin-film equation in the limit of a Newtonian rheology. Existence and uniqueness of solutions to the Ellis thin-film equation emerging from positive initial values is provided in \cite{LienstrombergMueller2020}. An analysis of the long-time behavior of such solutions is provided in \cite{jansen_lienstromberg_nik2023} for both the power-law thin-film equation and the Ellis thin-film equation.

\bigskip

We mention that a corresponding theory of classical solutions (existence and uniqueness of strong solutions) has been developed for the Newtonian thin-film equation \eqref{tfe-n} starting with \cite{GiacomelliKnuepferOtto2008,BringmannGiacomelliKnuepferOtto2016,GnannIbrahimMasmoudi2019} treating \eqref{tfe-n} with linear mobility ($n = 1$) and zero contact angle in the half space, which was further upgraded to compactly supported solutions and higher dimensions in \cite{Gnann2015,John2015,Seis2018}, and to nonlinear mobilities in \cite{GiacomelliGnannKnuepferOtto2014,Gnann2016,GnannPetrache2018,GnannWisse2025}. The case of nonzero contact angles was covered in \cite{Knuepfer2011,Knuepfer2015}. Our hope is that the analysis at hand may serve as a natural first step towards a corresponding analytic theory in the non-Newtonian case.

\bigskip

In \cite{Knuepfer_Velazquez2025}, the authors study the no-slip thin-film equation as the asymptotic limit of models with different regularized slip conditions at the liquid-solid interface. Their results suggest that the Newtonian thin-film equation with no-slip boundary condition admits a large class of physically admissible solutions but that sufficient regularity of solutions implies that the contact line cannot move. This is in line with our findings for the free-boundary problem to the non-Newtonian thin-film equation \eqref{free} as stated in Theorem~\ref{th-velocity} and the heuristic discussion in \S\ref{sec:kernel}, where a singularity drives the contact line. The afore-mentioned well-posedness theory to the Newtonian thin-film equation \eqref{tfe-n} is, unless $n = 1$, in line with this finding, too.

\subsection{Outlook}
In future work, we would be very interested in lifting the analysis at hand to also address uniqueness of solutions. Here, our leading-order estimates need to be more refined, and perturbations around the solution $u$ including additional regularity need to be addressed in order to obtain a contraction. Notably, in case of Ellis' law, local well-posedness of the corresponding thin-film model without contact line has been proved in \cite{LienstrombergMueller2020}, and we expect this case to be more feasible compared to the situation at hand when it comes to proving uniqueness.

\bigskip

A future research direction concerns the limit $\alpha \downarrow 1$, or which, as pointed out in \S\ref{sec:kernel}, the singular expansion at the contact line comes into play in a more pronounced way, and a more careful subtraction of these terms is necessary even for the leading-order estimates.

\bigskip

Another direction of research could be the rigorous analysis of droplet rupture. A rigorous result in this regard for the thin-film equation coming from Darcy dynamics in the Hele-Shaw cell (the PDE \eqref{tfe-n} with $n = 1$) is due to Constantin, Elgindi, Nguyen, and Vicol in \cite{ConstantinElgindiNguyenVicol2018}, in which finite- or infinite-time pinch off in a confined geometry is proved. It would be desirable to extend these results to the non-Newtonian setting or nonlinear mobilities. The analysis at hand may then serve as a first step allowing for an extension of the solution beyond pinch off.

\subsection{Conventions}
We write $A \lesssim_S B$, whenever a constant $C < \infty$ only depending on the set of parameters $S$ exists such that $A \le C B$. We write $\N = \{1,2,3,\ldots\}$ for the integers larger than $0$.

\subsection{Outline}
The rest of the paper is structured as follows: In \S\ref{sec:approx-embed} we derive approximation estimates and weighted interpolation estimates that are essential for treating higher-order nonlinear terms. Specifically, in \S\ref{sec:ualpha} we characterize the initial-data or trace space $U_\alpha$, for which we derive suitable embeddings. Notably, here an approximation with test functions (cf.~Lemma~\ref{lem-approx-ualpha} below) is possible, so that the boundary asymptotics as in \S\ref{sec:kernel} do not come into play. This is different when additionally assuming finiteness of the integral $\int_0^\infty (\partial_x^2 (x^{\alpha+2} g_\alpha(\partial_x^2 v)))^2 \, \d x$ for $v \in U_\alpha$. We write $V_\alpha$ for the corresponding linear space and derive suitable embeddings and interpolation estimates in \S\ref{sec:valpha}. These are then utilized in \S\ref{sec:higher-est} to bound the higher-order nonlinear terms $N(u)$ from \eqref{def-nu}. The paper is concluded with a time-stepping procedure in \S\ref{sec:time-discretization}, in which the higher-order nonlinear estimates are employed and passage to the time-continuum version \eqref{problem-u} is achieved with a compactness argument. As a result, we obtain Theorems~\ref{th:ex} and \ref{th-velocity}.

\section{Approximation and weighted interpolation estimates}\label{sec:approx-embed}
%
\subsection{Approximation and interpolation estimates in $U_\alpha$}\label{sec:ualpha}
We will repeatedly use the following Hardy-type inequalities, obtained from the Hardy inequality applied twice for the first estimate (using the boundary condition $u = 0$ at $x = 0$) and once for the second.
\begin{lemma}\label{lem:Hardy}
For $\alpha > 1$, we have 
\begin{equation}\label{eq:Hardy}
\int_0^\infty x^{-\alpha} |u|^{\alpha+1} \, \d x \lesssim_\alpha \int_0^\infty x^{\alpha+2} |\partial_x^2 u|^{\alpha+1} \, \d x \quad \text{and} \quad \int_0^\infty x |\partial_x u|^{\alpha+1} \, \d x \lesssim_\alpha \int_0^\infty x^{\alpha+2} |\partial_x^2 u|^{\alpha+1} \, \d x
\end{equation}
for all $u \in U_\alpha$.
\end{lemma}

We have the following approximation result for the space $U_\alpha$ of the initial data:
\begin{lemma}\label{lem-approx-ualpha}
For $\alpha > 1$ the space of test functions $C_\mathrm{c}^\infty((0,\infty))$ is dense in $U_\alpha$.
\end{lemma}

The proof follows the lines of standard approximation results of Lebesgue and Sobolev spaces.
\begin{proof}[Proof of Lemma~\ref{lem-approx-ualpha}]
For $u \in U_\alpha$ define $v(s) \coloneq u(e^s)$. It then holds
\begin{align*}
\|u\|_{L^2(0,\infty)}^2 &= \int_\R e^s v^2 \, \d s, \\
\big\|x^{\frac{\alpha+2}{\alpha+1}} \partial_x^2 u\big\|_{L^{\alpha+1}(0,\infty)}^{\alpha+1} &= \int_0^\infty x^{-\alpha} |(x \partial_x - 1) x \partial_x u|^{\alpha+1} \, \d x = \int_\R e^{(1-\alpha) s} |(\partial_s - 1) \partial_s v|^{\alpha+1} \, \d s.
\end{align*}
By \eqref{eq:Hardy} of Lemma~\ref{lem:Hardy} we additionally have
\[
\int_0^\infty x^{-\alpha} |u|^{\alpha+1} \, \d x = \int_\R e^{(1-\alpha)s} |v|^{\alpha+1} \, \d s < \infty \quad \text{and} \quad \int_0^\infty x |\partial_x u|^{\alpha+1} \, \d x = \int_\R e^{(1-\alpha)s} |\partial_s v|^{\alpha+1} \, \d s < \infty.
\]
Now take $\eta \in C^\infty_\mathrm{c}(\R)$ with $0 \le \eta \le 1$, $\eta|_{[-1,1]} = 1$ and $\eta|_{(-\infty,-2] \cup [2,\infty)} = 0$, define for $k \in \N$ the function $\eta_k(s) \coloneq \eta(s/k)$, where $s \in \R$. It then holds $0 \le \eta_k \le 1$ and $\eta_k \to 1$ as $k \to \infty$ point-wise. Define $u_k(x) \coloneq \eta_k(\ln x) v(\ln x)$. Then
\[
\|u-u_k\|_{L^2(0,\infty)}^2 = \int_\R (1-\eta_k)^2 e^s v^2 \, \d s \to 0 \quad \text{as } k \to \infty
\]
by dominated convergence with dominating function $e^s v^2$. Furthermore,
\begin{align*}
\big\|x^{\frac{\alpha+2}{\alpha+1}} \partial_x^2 (u-u_k)\big\|_{L^{\alpha+1}(0,\infty)}^{\alpha+1} &\lesssim_\alpha \int_\R e^{(1-\alpha) s} |1-\eta_k|^{\alpha+1} |(\partial_s - 1) \partial_s v|^{\alpha+1} \, \d s \\
&\phantom{\lesssim_\alpha} + \frac{1}{k^{\alpha+1}} \int_\R e^{(1-\alpha)s} (|v|^{\alpha+1} + |\partial_s v|^{\alpha+1}) \, \d s \to 0 \quad \text{as } k \to \infty,
\end{align*}
by dominated convergence. Since by Sobolev embedding $u$ and $v$ are continuous, this proves that $C^0_\mathrm{c}((0,\infty))$ is dense in $U_\alpha$.

\bigskip

Next, we take $\phi \in C^\infty_\mathrm{c}(\R)$ with $\phi \ge 0$ and $\int_\R \phi \, \d s = 1$. We set $\phi_k(s) \coloneq k \phi(k s)$ and have for $u \in U_\alpha \cap C^0_\mathrm{c}((0,\infty))$ (note that functions in $U_\alpha$ are continuous by Sobolev embedding), $v(s) \coloneq u(e^s)$, $v_k(s) \coloneq (\phi_k * v)(s) \coloneq \int_\R \phi_k(s') v(s-s') \, \d s'$, and $u_k(x) \coloneq v_k(\ln x)$. We then use
\[
v(s) - v_k(s) = \int_{\R} \phi_k(s') (v(s)-v(s-s')) \, \d s' = \int_\R \phi(\tau) (v(s)-v(s-\tau/k)) \, \d\tau \to 0 \quad \text{as } k \to \infty
\]
by dominated convergence, as $v$ is continuous and bounded. Furthermore,
\[
\|v_k\| \le \|\phi_k\|_{L^1(\R)} \|v\|_{L^\infty(\R)} = \|v\|_{L^\infty(\R)}
\]
by Young's convolution inequality and since the set $\{s \in \R \colon v(s) \ne 0 \text{ or } v_k(s) \ne 0 \text{ for a } k \in \N\}$ is pre-compact, it follows
\[
\|u-u_k\|_{L^2(0,\infty)}^2 = \int_\R e^s (v-v_k)^2 \, \d s \to 0 \quad \text{as } k \to \infty
\]
by applying dominated convergence again. Next, observe
\begin{align*}
\big\|x^{\frac{\alpha+2}{\alpha+1}} \partial_x^2 (u-u_k)\big\|_{L^{\alpha+1}(0,\infty)}^{\alpha+1} = \int_\R e^{(1-\alpha) s} |(\partial_s - 1) \partial_s v - \phi_k * ((\partial_s - 1) \partial_s v)|^{\alpha+1} \, \d s.
\end{align*}
We use
\begin{align*}
e^{\frac{1-\alpha}{1+\alpha} s} (\phi_k * ((\partial_s - 1) \partial_s v))(s) &= e^{\frac{1-\alpha}{1+\alpha} s} \int_\R \phi_k(s') ((\partial_s - 1) \partial_s v)(s-s') \, \d s' \\
&= \int_\R (e^{\frac{1-\alpha}{1+\alpha} s'} \phi_k(s')) (e^{\frac{1-\alpha}{1+\alpha} (s-s')} ((\partial_s - 1) \partial_s v)(s-s')) \, \d s' \\
&= (\psi_k * w)(s),
\end{align*}
where
\[
\psi_k(s) \coloneq e^{\frac{1-\alpha}{1+\alpha} s} \phi_k(s) \qquad \text{and} \qquad w(s) \coloneq e^{\frac{1-\alpha}{1+\alpha} s} ((\partial_s - 1) \partial_s v)(s).
\]
We then have
\begin{align*}
\big\|x^{\frac{\alpha+2}{\alpha+1}} \partial_x^2 (u-u_k)\big\|_{L^{\alpha+1}(0,\infty)}^{\alpha+1} = \int_\R |w(s) - (\psi_k * w
)(s)|^{\alpha+1} \, \d s.
\end{align*}
Lusin's theorem implies that there exists for every $\ell \in \N$ a function $w_\ell \in C^0_\mathrm{c}(\R)$ such that $\|w - w_\ell\|_{L^{\alpha+1}(\R)} < \frac 1 \ell$.  It then holds
\begin{align*}
w - \psi_k * w = (w - w_\ell) + (w_\ell - \psi_k * w_\ell) + \psi_k * (w_\ell-w).
\end{align*}
For $\eps > 0$ we choose $\ell \in \N$ such that $\ell > \frac 3 \eps$ and thus $\|w - w_\ell\|_{L^{\alpha+1}(\R)} < \frac \eps 3$. Secondly, by Young's convolution inequality we have
\[
\|\psi_k * (w-w_\ell)\|_{L^{\alpha+1}(\R)} \le \|\psi_k\|_{L^1(\R)} \|w - w_\ell\|_{L^{\alpha+1}(\R)} \le \|w - w_\ell\|_{L^{\alpha+1}(\R)} \int_\R e^{\frac{|1-\alpha|}{1+\alpha} \tau} \phi(\tau) \, \d \tau,
\]
so that $\|\psi_k * (w-w_\ell)\|_{L^{\alpha+1}(\R)} < \frac \eps 3$ upon enlarging $\ell$. Lastly, observe
\begin{align*}
(w_\ell - \psi_k * w_\ell)(s) &= \int_{\R} \psi_k(s') (w_\ell(s)-w_\ell(s-s')) \, \d s' + w_\ell(s) \int_{\R} (1-\psi_k(s')) \, \d s' \\
&= \int_{\R} e^{\frac{1-\alpha}{1+\alpha} \frac \tau k} \phi(\tau) (w_\ell(s)-w_\ell(s-\tfrac \tau k)) \, \d \tau + w_\ell(s) \int_{\R} (1-e^{\frac{1-\alpha}{1+\alpha} \frac \tau k} \phi(\tau)) \, \d \tau 
\end{align*}
We note that for every $s \in \R$
\[
\int_{\R} e^{\frac{1-\alpha}{1+\alpha} \frac \tau k} \phi(\tau) (w_\ell(s)-w_\ell(s-\tfrac \tau k)) \, \d \tau \to 0 \qquad \text{and} \qquad w_\ell(s) \int_{\R} \big(1-e^{\frac{1-\alpha}{1+\alpha} \frac \tau k} \phi(\tau)\big) \, \d \tau \to 0 \quad \text{as } k \to \infty
\]
by dominated convergence, as $w_\ell$ and $\phi$ are continuous with compact support. Furthermore,
\[
|(w_\ell - \psi_k * w_\ell)(s)| = \Big|w_\ell(s) - \int_\R e^{\frac{1-\alpha}{1+\alpha} \frac \tau k} \phi(\tau) w_\ell(s-\tfrac \tau k) \, \d\tau\Big| \le \Big|1 + \int_\R e^{\frac{|1-\alpha|}{1+\alpha} \tau} \phi(\tau) \, \d\tau\Big| \|w_\ell\|_{L^\infty(\R)} < \infty,
\]
so that by bounded convergence on a compact interval we have $\|w_\ell - \psi_k * w_\ell\|_{L^{\alpha+1}(\R)} < \frac \eps 3$ for $k$ sufficiently large. Thus we obtain
\[
\big\|x^{\frac{\alpha+2}{\alpha+1}} \partial_x^2 (u-u_k)\big\|_{L^{\alpha+1}(0,\infty)} \to 0 \quad \text{as } k \to \infty. \qedhere
\]
\end{proof}

We need the following auxiliary result:
\begin{lemma}\label{lem:recursion-dxu}
We have for $\alpha > 1$, $\nu \in \R$, and $r > 1$
\begin{equation}\label{eq:recursion-dxu}
\begin{split}
\int_0^\infty x^{\nu r} |\partial_x u|^r \, \d x
\lesssim_{\alpha,\nu,r} & \Big(  \int_0^\infty  u^2 \, \d x \Big)^{\frac{\alpha+1}{\alpha+3}(  1- \theta^k)} 
\Big( \int_0^\infty x^{\gamma_k} |\partial_x u|^{r_k} \, \d x \Big)^{\theta^k}
\Big( \int_0^\infty  x^{\alpha+2} |\partial_x^2 u|^{\alpha+1} \, \d x\Big)^{\frac{2}{\alpha+3}(1- \theta^k)}
\end{split}
\end{equation}
for all $u \in C^\infty_\mathrm{c}((0,\infty))$, where $\theta = \frac{\alpha-1}{2 (\alpha+1)}$ and
\begin{subequations}\label{eq:gamma1-r1}
\begin{align}
\gamma_k &
= \theta^{-k} \Big( \nu r - \frac{2(\alpha+2)}{\alpha+3} \Big) + \frac{2(\alpha+2)}{\alpha+3}, \label{eq:gamma1} \\
r_k &
= \theta^{-k} \Big( r - \frac{4(\alpha+1)}{\alpha+3} \Big) + \frac{4(\alpha+1)}{\alpha+3}.\label{eq:r1}
\end{align}
\end{subequations}
\end{lemma}

\begin{proof}
Consider the integral
\[
\int_0^\infty x^{\gamma_k} |\partial_x u|^{r_k} \, \d x, \quad k\in \N,
\]
where $\gamma_k \in \R$ and $r_k > 1$ are parameters depending on $\alpha$, $\nu$, $r$, and $k$, and where $\gamma_0 = \nu r$ and $r_0 = r$. Using integration by parts and H\"older’s inequality, we get
\begin{align}
\int_0^\infty x^{\gamma_k} |\partial_x u|^{r_k} \, \d x 
&= \int_0^\infty x^{\gamma_k} |\partial_x u|^{r_k-2} (\partial_x u)^2 \, \d x \nonumber \\
&= - (r_k-1) \int_0^\infty x^{\gamma_k} u |\partial_x u|^{r_k-2} \partial_x^2 u \, \d x - \gamma_k \int_0^\infty x^{\gamma_k-1} |\partial_x u|^{r_k-2} (\partial_x u) u \, \d x \nonumber \\
&\lesssim_{\gamma_k, r_k} 
\int_0^\infty |u| x^{\gamma_k-\frac{\alpha+2}{\alpha+1}}  |\partial_x u|^{r_k -2} 
x^{\frac{\alpha+2}{\alpha+1}} \left( |\partial_x^2 u| + x^{-1} |\partial_x u|\right) \d x \nonumber \\
&\leq \Big(  \int_0^\infty  u^2 \, \d x \Big)^{\frac{1}{2}} 
\Big( \int_0^\infty x^{(\gamma_k - \frac{\alpha+2}{\alpha +1})\frac{2(\alpha+1)}{\alpha-1}}
 |\partial_x u|^{(r_k-2)\frac{2(\alpha +1)}{\alpha -1}} \, \d x \Big)^{\frac{\alpha -1}{2(\alpha+1)}} \nonumber \\
&\quad
\times \Big( \int_0^\infty x^{\alpha +2} 
\big( |\partial_x^2 u| + x^{-1} |\partial_x u|\big)^{\alpha +1} \, \d x \Big)^{\frac{1}{\alpha+1}}. \label{eq:weight xu9}
\end{align}
Next, we apply Jensen’s inequality and Hardy’s inequality to the last integral on the right-hand side of \eqref{eq:weight xu9}, leading to
\begin{align} \label{eq:weight xu10}
\int_0^\infty x^{\gamma_k} |\partial_x u|^{r_k} \, \d x &\lesssim_{\alpha,\gamma_k, r_k}  \Big(  \int_0^\infty  u^2 \, \d x \Big)^{\frac{1}{2}} \Big( \int_0^\infty x^{\gamma_{k+1}}
|\partial_x u|^{r_{k+1}} \, \d x \Big)^{\frac{\alpha -1}{2(\alpha+1)}} 
\Big( \int_0^\infty  x^{\alpha+2} |\partial_x^2 u|^{\alpha+1} \, \d x\Big)^{\frac{1}{\alpha+1}},
\end{align}
where we have introduced the recursive relations for the exponents $\gamma_k$ and $r_k$ as 
\begin{equation}\label{eq:gamma r}
	\gamma_{k+1} = \theta^{-1} \gamma_k - \frac{2(\alpha +2)}{\alpha -1} 
\qquad \text{and}\qquad r_{k+1} = \theta^{-1} r_k - \frac{4(\alpha +1)}{\alpha -1}. 
\end{equation}
Both follow the form 
\begin{equation*}
	x_k= \theta^{-1} x_{k-1} -b, \quad \theta,b > 0. 
\end{equation*}
To express $x_k$ explicitly in terms of $x_0$, $\theta$, and $b$, we introduce the substitution $x_k = y_k \theta^{-k}$. Using the recursive definition of $x_k$ and the formula for the geometric sum yields
\begin{equation*}
	y_k = y_{k-1} - b \theta^k = y_{k-2} - b \theta^k - b \theta^{k-1} = \dots = y_0 - b \theta \sum_{j=0}^{k-1} \theta^j = y_0 - b \theta \frac{1-\theta^k}{1-\theta} 
\end{equation*}
Substituting back $y_k = x_k \theta^{k}$ gives the explicit form of $x_k$, namely
\begin{equation*}
	x_k= \theta^{-k} y_k = \theta^{-k} \Big( x_0 - \frac{b \theta}{1-\theta} \Big) + \frac{b \theta}{1-\theta},
\end{equation*}
with which we can now rewrite $\gamma_k$ and $r_k$ in terms of their initial values. Using $\gamma_0 = \nu r$ and $b = \frac{2(\alpha+2)}{\alpha-1}$ for the iteration of $\gamma_k$, we obtain \eqref{eq:gamma1}. Likewise, for \(r_k\) with $r_0 = r$ and \(b=\frac{4(\alpha +1)}{\alpha -1}\), we find \eqref{eq:r1}. 

\medskip

Two of the integrals in \eqref{eq:weight xu10} are already in the desired form and it remains to estimate the second integral containing $\partial_x u$. Using the explicit forms \eqref{eq:gamma1-r1}, we can estimate \eqref{eq:weight xu10} iteratively and obtain
\begin{equation*}
\begin{split}
 & \int_0^\infty x^{\nu r} |\partial_x u|^{r}\, \d x \\
 & \quad \lesssim_{\alpha,\beta,r}  \Big(  \int_0^\infty  u^2 \, \d x \Big)^{\frac{1}{2}\sum_{j=0}^{k-1}\theta^j}
\Big( \int_0^\infty x^{\gamma_k}
 |\partial_x u|^{r_k} \, \d x \Big)^{\theta^k}
 \Big( \int_0^\infty  x^{\alpha+2} |\partial_x^2 u|^{\alpha+1} \, \d x\Big)^{\frac{1}{\alpha+1}\sum_{j=0}^{k-1}\theta^j}
 \\
& \quad = \Big(  \int_0^\infty  u^2 \, \d x \Big)^{\frac{\alpha+1}{\alpha+3}(  1-\theta^k)}
\Big( \int_0^\infty x^{\gamma_k}
|\partial_x u|^{r_k} \, \d x \Big)^{\theta^k}
\Big( \int_0^\infty  x^{\alpha+2} |\partial_x^2 u|^{\alpha+1} \, \d x\Big)^{\frac{2}{\alpha+3}(1- \theta^k)}
 \end{split}
 \end{equation*}
which is \eqref{eq:recursion-dxu}.
\end{proof}
\begin{lemma}\label{lem:weight xu}
For \(\alpha >1\) and \(\beta \in \big[\frac{1-\alpha}{\alpha +1}, \frac{\alpha+2}{5\alpha +3}\big)\),
we have
\begin{equation}\label{eq:weight xu}
\sup_{x \ge 0} x^\beta |u| \lesssim_{\alpha,\beta} \Big(\int_0^\infty u^2 \, \d x\Big)^{\frac{\alpha -1 + (\alpha +1)\beta}{3\alpha -1}} \Big(\int_0^\infty x^{\alpha+2} \, |\partial_x^2 u|^{\alpha+1} \, \d x\Big)^{\frac{1-2\beta}{3\alpha -1}}
\end{equation}
for all $u \in U_\alpha$.
\end{lemma}
\begin{proof}
We first establish \eqref{eq:weight xu} for $u \in C^\infty_\mathrm{c}((0,\infty))$
and then extend to $u \in U_\alpha$ by approximation in Step~4.

\bigskip

\noindent\textbf{Step~1. Weighted estimate for \(\beta= \frac{1-\alpha}{\alpha +1}\).}
Let $\beta = \frac{1-\alpha}{\alpha+1}$. In this step, we derive a weighted estimate for the critical case where the $L^2$-norm of $u$ does not appear on the right-hand side of the estimate \eqref{eq:weight xu}. To this end, we apply the fundamental theorem of calculus leading to
\begin{equation*}
\sup_{x \ge 0} x^{1-\alpha} |u|^{\alpha +1} \lesssim_{\alpha,\beta} \int_0^\infty \left( x^{-\alpha} |u|^{\alpha +1}  +  x^{1-\alpha} |u|^{\alpha}|\partial_x u| \right) \d x \lesssim_{\alpha} \int_0^\infty \left( x^{-\alpha} |u|^{\alpha +1}  +  x |\partial_x u|^{\alpha+1} \right) \d x,
\end{equation*}
where Young's inequality was used in the second step. With help of Hardy’s inequality \eqref{eq:Hardy} of Lemma~\ref{lem:Hardy}, we obtain the upgraded weighted estimate
\begin{equation}\label{eq:weight xu7}
\sup_{x \ge 0} x^{\frac{1-\alpha}{\alpha + 1}} |u| \lesssim_{\alpha} \Big(\int_0^\infty  x^{\alpha+2} \, |\partial_x^2 u|^{\alpha+1} \, \d x\Big)^{\frac{1}{\alpha +1}}.
\end{equation}
\bigskip

\noindent\textbf{Step~2. Critical bound $\beta \uparrow \frac{\alpha+2}{5\alpha+3}$.} Next we raise the value of $\beta$. Therefore, suppose $p > 1$ and $\beta \in \R$. Applying the fundamental theorem of calculus yields
\begin{align*}
\sup_{x \ge 0} x^{\beta p} |u|^p & \le \int_0^\infty \big|\partial_x \big(x^{\beta p} (u^2)^{\frac p 2}\big)\big| \, \d x 
\le \int_0^\infty |\beta| p \, x^{\beta p -1} |u|^p \, \d x + \int_0^\infty x^{\beta p} \big|\tfrac p 2 (u^2)^{\frac p 2 - 1} 2 u \partial_x u\big| \, \d x \\
&\lesssim_{\beta,p} \int_0^\infty  |u|^{p-1} x^{\beta p -1} \left( |u| + x |\partial_x u| \right) \d x. 
\end{align*}
Applying H\"older’s, Jensen’s, and Hardy’s inequality, we obtain for $p\in(1,3)$
\begin{align}
	\sup_{x \ge 0} x^{\beta p} |u|^p
& \lesssim_{p,\beta} \Big( \int_0^\infty  u^{2}  \, \d x \Big)^{\frac{p-1}{2}}\Big( \int_0^\infty  x^{\frac{2(\beta p-1)}{3-p}} \left( |u| + x |\partial_x u|\right)^{\frac{2}{3-p}} \, \d x\Big)^{\frac{3-p}{2}} \nonumber \\
& \lesssim_{p,\beta}  \Big( \int_0^\infty  u^{2}  \, \d x \Big)^{\frac{p-1}{2}} 
\Big( \int_0^\infty  x^{\frac{2(\beta p-1)}{3-p}} |u| ^{\frac{2}{3-p}} \, \d x 
+  \int_0^\infty  x^{\frac{2(\beta p-1)}{3-p} + \frac{2}{3-p}} |\partial_x u| ^{\frac{2}{3-p}} \, \d x
\Big)^{\frac{3-p}{2}} \nonumber \\
& \lesssim_{p,\beta}  \Big( \int_0^\infty  u^{2}  \, \d x \Big)^{\frac{p-1}{2}} 
\Big( \int_0^\infty  x^{\frac{2 \beta p}{3-p}} |\partial_x u| ^{\frac{2}{3-p}} \, \d x
\Big)^{\frac{3-p}{2}}.
\label{eq:weight xu8}
\end{align}
To further estimate the second integral on the right-hand side of \eqref{eq:weight xu8}, we apply Lemma~\ref{lem:recursion-dxu} to reduce the integral into a form suitable for Hardy’s inequality \eqref{eq:Hardy} of Lemma~\ref{lem:Hardy}. 
Inserting estimate \eqref{eq:recursion-dxu} with $\nu =\beta p$ and $r = \frac{2}{3-p}$ in estimate  \eqref{eq:weight xu8} leads to
\begin{align}
\sup_{x \ge 0} x^{\beta p} |u|^p 
& \lesssim_{\beta,p}  \Big( \int_0^\infty  u^{2}  \, \d x \Big)^{\frac{p-1}{2}} 
\Big( \int_0^\infty  x^{\frac{2 \beta p}{3-p}} |\partial_x u| ^{\frac{2}{3-p}} \, \d x
\Big)^{\frac{3-p}{2}} \nonumber \\
&\lesssim_{\alpha,\beta,p} \Big(  \int_0^\infty  u^2 \, \d x \Big)^{\frac{p-1}{2} + \frac{(3-p)(\alpha+1)}{2(\alpha+3)}(1- \theta^k)} 
\Big( \int_0^\infty x^{\gamma_k}
|\partial_x u|^{r_k} \, \d x \Big)^{\frac{3-p}{2}\theta^k} \nonumber \\
&\phantom{\lesssim_{\alpha,\beta,p}} \times
\Big( \int_0^\infty  x^{\alpha+2} |\partial_x^2 u|^{\alpha+1} \, \d x\Big)^{\frac{3-p}{\alpha+3}(  1- \theta^k)},
	\label{eq:weight xu11}
\end{align}
where $\theta = \frac{\alpha-1}{2(\alpha+1)}$, and $\gamma_k$ and $r_k$ are as in \eqref{eq:gamma1-r1}. To apply Hardy’s inequality \eqref{eq:Hardy} to the last integral on the right-hand side of \eqref{eq:weight xu11}, we need to enforce the conditions
\begin{equation}\label{eq:gammak rk}
\gamma_k = 1\qquad \text{and} \qquad r_{k} = \alpha +1. 
\end{equation}
By choosing \(\beta=\beta_k\) and \(p=p_k\) in \eqref{eq:weight xu11}, and applying \eqref{eq:gammak rk} alongside Hardy’s inequality, we deduce
\begin{equation}\label{eq:weight xu12}
\sup_{x \ge 0} x^{\beta_k p_k} |u|^{p_k} 
\lesssim_{\alpha,\beta_k,p_k} \Big(  \int_0^\infty  u^2 \, \d x \Big)^{\frac{p_k-1}{2} + \frac{(3-p_k)(\alpha+1)}{2(\alpha+3)}(1- \theta^k)} 
\Big( \int_0^\infty x^{\alpha+2}|\partial_x^2 u|^{\alpha+1} \, \d x \Big)^{\frac{3-p_k}{\alpha+3} + \frac{(3-p_k)(\alpha+1)}{2(\alpha+3)}\theta^k}.
\end{equation}
Next, we determine $\beta_k$ and $p_k$ by combining \eqref{eq:gamma1-r1} and \eqref{eq:gammak rk}. Starting with $p_k$, from \eqref{eq:r1} and \eqref{eq:gammak rk}, we obtain
\begin{equation*}
\theta^{-k} \Big( \frac{2}{3-p_k} - \frac{4(\alpha+1)}{\alpha+3} \Big) + \frac{4(\alpha+1)}{\alpha+3} = \alpha +1,
\end{equation*}
which gives
\begin{equation*}
    p_k = \frac{2(5\alpha+3) + 3 (\alpha-1) (\alpha+1) \theta^k}{(\alpha+1) (4+(\alpha-1)\theta^k)} = \frac{5 \alpha +3}{2(\alpha+1)}(1+ O(\theta^k)). 
\end{equation*}
Similarly, using \eqref{eq:gamma1} and \eqref{eq:gammak rk}, we solve for $\beta_k$:
\begin{equation*}
\theta^{-k} \Big( \frac{2 \beta_k  p_k}{3-p_k} - \frac{2(\alpha+2)}{\alpha+3} \Big) + \frac{2(\alpha+2)}{\alpha+3} = 1,
\end{equation*}
which gives
\begin{equation}\label{eq:betak_u}
    \beta_k = \frac{2(\alpha+2) - (\alpha+1) \theta^k}{2 (5\alpha+3) + 3(\alpha+1) (\alpha-1)\theta^k} = \frac{\alpha +2}{5\alpha+3}(1+O(\theta^k)).  
\end{equation}
Substituting $\beta_k$ and $p_k$ back into \eqref{eq:weight xu12} provides the weighted estimate
\begin{equation}\label{eq:weight xu13}
\sup_{x \ge 0} x^{\beta_k} |u| 
\lesssim_{\alpha,k} \Big(  \int_0^\infty  u^2 \, \d x \Big)^{\frac{(\alpha+1) \beta_k+\alpha-1}{3\alpha-1}}  \Big( \int_0^\infty x^{\alpha+2}
|\partial_x^2 u|^{\alpha+1} \, \d x \Big)^{\frac{1-2\beta_k}{3\alpha-1}}, \quad \beta_k \uparrow \frac{\alpha +2}{5\alpha+3}.
\end{equation}
\bigskip

\noindent\textbf{Step~3. Combination of bounds of Steps 1 and 2.}
We employ the weighted estimates \eqref{eq:weight xu7} and \eqref{eq:weight xu13} obtained in Steps 1 and 2
by writing $\beta$ as a convex combination
\[
\beta = \frac{1-\alpha}{\alpha+1} \varrho_k + \beta_k (1-\varrho_k), \quad \varrho_k \in [0,1], \quad \beta_k \uparrow \frac{\alpha +2}{5\alpha+3},
\]
with $\beta_k$ as in \eqref{eq:betak_u} and where
\[
\varrho_k = \frac{(\alpha+1) (\beta_k-\beta)}{(\alpha+1)\beta_k+\alpha-1} \quad \text{and} \quad 1-\varrho_k = \frac{(\alpha+1)\beta+\alpha-1}{(\alpha+1)\beta_k+\alpha-1},
\]
so that
\begin{align*}
\sup_{x \ge 0} x^{\beta} |u| \; \quad &\le_{\phantom{\alpha,k}} \; \big(\sup_{x \ge 0} x^{\frac{1-\alpha}{\alpha+1}} |u|\big)^{\varrho_k} \big(\sup_{x \ge 0} x^{\beta_k} |u|\big)^{1-\varrho_k} \\
&\stackrel{\mathclap{\eqref{eq:weight xu7}, \eqref{eq:weight xu13}}}{\lesssim}_{\alpha,k} \; \Big(\int_0^\infty  x^{\alpha+2} \, |\partial_x^2 u|^{\alpha+1} \, \d x\Big)^{\frac{\varrho_k}{\alpha +1}} \Big( \int_0^\infty  u^{2} \, \d x \Big)^{\frac{(\alpha+1) \beta_k+\alpha-1}{3\alpha-1} (1-\varrho_k)}
\\
&\phantom{\lesssim_{\alpha,k}} \; \times 
\Big(\int_0^\infty  x^{\alpha+2} \, |\partial_x^2 u|^{\alpha+1} \, \d x\Big)^{\frac{1-2\beta_k}{3\alpha-1} (1-\varrho_k)} \\
&=_{\phantom{\alpha,k}} \; \Big(\int_0^\infty   u^{2} \, \d x\Big)^{\gamma}  \Big(\int_0^\infty x^{\alpha+2} \, |\partial_x^2 u|^{\alpha+1} \, \d x\Big)^{\delta}, 
\end{align*}
where an straight-forward computation shows that the exponents \(\gamma\) and \(\delta\) are given by
\begin{equation*}
\gamma =\frac{(\alpha+1) \beta_k+\alpha-1}{3\alpha-1} (1-\varrho_k) = \frac{(\alpha +1)\beta+\alpha-1}{3\alpha -1}, \quad \delta = \frac{\varrho_k}{\alpha+1} + \frac{1-2\beta_k}{3\alpha-1} (1-\varrho_k)= \frac{1-2\beta}{3\alpha -1}.
\end{equation*}
Since $\beta_k \uparrow \frac{\alpha+2}{5\alpha+3}$, this finishes the proof of \eqref{eq:weight xu} for \(\beta \in \big[\frac{1-\alpha}{\alpha +1}, \frac{\alpha+2}{5\alpha +3}\big)\).

\bigskip

\noindent\textbf{Step~4. Extension from $C^\infty_\mathrm{c}((0,\infty))$ to $U_\alpha$.}
Steps~1--3 establish \eqref{eq:weight xu} for $u \in C^\infty_\mathrm{c}((0,\infty))$.
For general $u \in U_\alpha$, Lemma~\ref{lem-approx-ualpha} provides a sequence
$(u_k)_{k \in \N} \subset C^\infty_\mathrm{c}((0,\infty))$ with $u_k \to u$ in $U_\alpha$. Applying \eqref{eq:weight xu} to
$u_k - u_m \in C^\infty_\mathrm{c}((0,\infty))$ 
yields
\begin{align*}
\sup_{x \ge 0} x^\beta |u_k - u_m|
&\leq
\|u_k - u_m\|_{U_\alpha}^{\frac{\alpha+(\alpha-1)\beta}{3\alpha-1}}
\to 0 \qquad \text{as } k,m \to \infty,
\end{align*}
so that $(x^\beta u_k)_{k \in \N}$ is a Cauchy sequence in $BC^0([0,\infty))$ and thus converges uniformly to some
$\varphi \in BC^0([0,\infty))$. Thus $x^\beta u$ admits a representative in $\varphi \in BC^0([0,\infty))$, again
denoted by $x^\beta u$, and 
the reverse triangle inequality gives
\begin{equation*}
    \sup_{x \ge 0} x^\beta |u_k| \to \sup_{x \ge 0} x^\beta |u|. 
\end{equation*}
Since $u_k \to u$ in $U_\alpha$ implies convergence of both norms on the
right-hand side of \eqref{eq:weight xu}, passing to the limit in \eqref{eq:weight xu} applied to $u_k$ yields the estimate for $u \in U_\alpha$.
\end{proof}
\begin{lemma}\label{lem:weight xu'}
For \(\alpha >1\) and \(\beta \in \big[\frac{2}{\alpha +1}, \frac{3(\alpha+2)}{5\alpha +3}\big)\), we have
\begin{equation}\label{eq:weight xu'}
\sup_{x \ge 0} x^\beta |\partial_x u| \lesssim_{\alpha,\beta} \Big(\int_0^\infty u^2 \, \d x\Big)^{\frac{(\alpha +1)\beta -2}{3\alpha -1}} \Big(\int_0^\infty x^{\alpha+2} \, |\partial_x^2 u|^{\alpha+1} \, \d x\Big)^{\frac{3-2\beta}{3\alpha -1}}
\end{equation}
for all  $u \in U_\alpha$.
\end{lemma}
\begin{proof}
The proof follows a similar approach as the one for Lemma~\ref{lem:weight xu}, but we include it here for completeness. We first establish \eqref{eq:weight xu'} for $u \in C^\infty_\mathrm{c}((0,\infty))$ and then extend to $u \in U_\alpha$ by approximation in Step~4.

\bigskip 

\noindent\textbf{Step~1. Weighted estimate and lower bound for \(\beta\).}
Let $p>1$ and $\beta \in \R$. In this step, we derive a weighted estimate for the critical case where the $L^2$-norm of $u$ does not appear on the right-hand side of inequality \eqref{eq:weight xu'}. 
Applying the fundamental theorem of calculus and H\"older’s inequality, we obtain
\begin{align}
\sup_{x \ge 0} x^{\beta p} |\partial_x u|^p &\le \int_0^\infty \big|\partial_x (x^{\beta p} ((\partial_x u)^2)^{\frac p 2})\big| \, \d x \nonumber \\
&= \beta p \int_0^\infty x^{\beta p - 1} |\partial_x u|^p \, \d x + \int_0^\infty x^{\beta p} \big| \tfrac p 2 ((\partial_x u)^2)^{\frac p 2 - 1} 2 (\partial_x u) (\partial_x^2 u)\big| \, \d x \nonumber \\
&\lesssim_{\beta,p} \int_0^\infty x^{\beta p -1} |\partial_x u|^{p}  \, \d x + \int_0^\infty x^{\beta p} |\partial_x u|^{p-1} |\partial_x^2 u| \, \d x \nonumber \\
&= \int_0^\infty x^{\beta p -\frac{\alpha+2}{\alpha +1}} |\partial_x u|^{p-1} x^{\frac{\alpha+2}{\alpha +1}}\left(|\partial_x^2 u| + x^{-1}|\partial_x u|\right)  \, \d x \nonumber \\
& \leq \Big( \int_0^\infty x^{\frac{\beta p (\alpha+1) - (\alpha+2)}{\alpha}} |\partial_x u|^{\frac{(p-1)(\alpha +1)}{\alpha}} \, \d x \Big)^{\frac{\alpha}{\alpha+1}} \Big( \int_0^\infty x^{\alpha+2} \left( |\partial_x^2 u| + x^{-1}|\partial_x u| \right)^{\alpha +1} \, \d x\Big)^{\frac{1}{\alpha +1}}. \label{eq:weight xu'1}
\end{align}
Now, applying Jensen’s inequality followed by Hardy’s inequality \eqref{eq:Hardy} of Lemma~\ref{lem:Hardy}, we estimate the second integral on the right-hand side of \eqref{eq:weight xu'1} and obtain
\begin{equation}\label{eq:weight xu'2}
\begin{split}
   &\sup_{x \ge 0} x^{\beta p} |\partial_x u|^p \lesssim_{\alpha,\beta,p} \Big( \int_0^\infty x^{\frac{\beta p(\alpha +1) -(\alpha+2)}{\alpha}} |\partial_x u|^{\frac{(p-1)(\alpha +1)}{\alpha}} \, \d x \Big)^{\frac{\alpha}{\alpha+1}}\Big(\int_0^\infty  x^{\alpha+2} \, |\partial_x^2 u|^{\alpha+1} \, \d x\Big)^{\frac{1}{\alpha +1}}.
\end{split}
\end{equation}
To apply Hardy’s inequality again to the first integral on the right-hand side of \eqref{eq:weight xu'2}, we require the exponents to satisfy
\begin{equation*} 
\frac{(p-1)(\alpha +1)}{\alpha} = \alpha +1 \qquad \text{and} \qquad 
\frac{\beta p(\alpha +1) -(\alpha+2)}{\alpha} = 1. 
\end{equation*}
Solving for $\beta$ and $p$, we find  
\begin{equation*}
    \beta = \frac{2}{\alpha + 1} \quad \text{and} \quad p = \alpha + 1,
\end{equation*}
and substituting these values into \eqref{eq:weight xu'2}, we deduce the weighted estimate
\begin{equation}\label{eq:weight xu'3}
\sup_{x \ge 0} x^{\frac{2}{\alpha + 1}} |\partial_x u| \lesssim_{\alpha} \Big(\int_0^\infty  x^{\alpha+2} \, |\partial_x^2 u|^{\alpha+1} \, \d x\Big)^{\frac{1}{\alpha +1}}. 
\end{equation}
\bigskip

\noindent\textbf{Step~2. Weighted estimate and upper bound for $\beta$.} 
We now use estimate~\eqref{eq:recursion-dxu} of Lemma~\ref{lem:recursion-dxu}, where we set $\nu r = \frac{\beta p(\alpha +1) -(\alpha+2)}{\alpha}$ and $r = \frac{(p-1) (\alpha+1)}{\alpha}$, and where $\gamma_k$ and $r_k$ are given by \eqref{eq:gamma1-r1}. Thus, \eqref{eq:weight xu'2} can be further estimated as
\begin{align}
\sup_{x \ge 0} x^{\beta p} |\partial_x u|^p 
&\lesssim_{\alpha,p,\beta} \Big(  \int_0^\infty  u^2 \, \d x \Big)^{\frac{\alpha}{\alpha+3} (1-\theta^k)} \Big( \int_0^\infty x^{\gamma_k}
|\partial_x u|^{r_k} \, \d x \Big)^{\frac{\alpha}{\alpha+1} \theta^k}
\nonumber \\
&\phantom{\lesssim_{\alpha,p,\beta}} \times
\Big( \int_0^\infty  x^{\alpha+2} |\partial_x^2 u|^{\alpha+1} \, \d x\Big)^{\frac{3}{\alpha+3} - \frac{2 \alpha}{(\alpha+1)(\alpha+3)} \theta^k}. \label{eq:weight xu'4}
\end{align}
In order to use Hardy’s inequality \eqref{eq:Hardy} of Lemma~\ref{lem:Hardy} on the second integral on the right-hand side of \eqref{eq:weight xu'4}, we impose the conditions
\begin{equation}\label{eq:gammak rk'}
    \gamma_k = 1\qquad \text{and} \qquad r_{k} = \alpha +1
 \end{equation}
which we solve for $\beta = \beta_k$ and $p = p_k$ further below. 
Using this in \eqref{eq:weight xu'4} and applying \eqref{eq:gammak rk'} along with Hardy’s inequality, we deduce
\begin{equation}\label{eq:weight xu'5}
 \sup_{x \ge 0} x^{\beta_k p_k} |\partial_x u|^{p_k} 
 \lesssim_{\alpha,\beta_k,p_k} \Big(  \int_0^\infty  u^2 \, \d x \Big)^{\frac{\alpha}{\alpha+3} (1-\theta^k)}  \Big( \int_0^\infty x^{\alpha+2}
 |\partial_x^2 u|^{\alpha+1} \, \d x \Big)^{\frac{3}{\alpha+3} + \frac{\alpha}{\alpha+3}\theta^k}.
 \end{equation}
 Next, we determine $\beta_k$ and $p_k$ by combining \eqref{eq:gamma1-r1} and \eqref{eq:gammak rk'}. Starting with $p_k$, from \eqref{eq:r1} and \eqref{eq:gammak rk'}, we obtain
\begin{equation*}
\theta^{-k} \Big( \frac{(p_k-1)(\alpha+1)}{\alpha} - \frac{4(\alpha+1)}{\alpha+3} \Big) + \frac{4(\alpha+1)}{\alpha+3} = \alpha +1
 \end{equation*}
which simplifies to
\begin{equation*}
p_k=\frac{5 \alpha +3 + \alpha (\alpha-1) \theta^k}{\alpha+3} \downarrow \frac{5\alpha+3}{\alpha+3}.
\end{equation*}
Similarly, using \eqref{eq:gamma1} and \eqref{eq:gammak rk'}, we solve for $\beta_k$
\begin{equation*}
\theta^{-k} \Big(\frac{\beta_k p_k (\alpha+1)}{\alpha} - 3 \frac{(\alpha+1)(\alpha+2)}{\alpha (\alpha+3)} \Big) + 2 \frac{\alpha+2}{\alpha+3} = 1
 \end{equation*}
and obtain
\begin{equation}\label{eq:betak_dxu}
\beta_k = \frac{3(\alpha+2) - \alpha \theta^k}{5\alpha+3+\alpha (\alpha-1)\theta^k} \uparrow 3 \frac{\alpha +2}{5\alpha+3} .
\end{equation}
Substituting $\beta_k$ and $p_k$ back into \eqref{eq:weight xu'5} provides the weighted estimate
\begin{equation}\label{eq:weight xu'6}
\sup_{x \ge 0} x^{\beta_k} |\partial_x u| 
\lesssim_{\alpha,k} \Big(\int_0^\infty  u^2 \, \d x \Big)^{\frac{(\alpha+1)\beta_k-2}{3\alpha-1}}  \Big( \int_0^\infty x^{\alpha+2}|\partial_x^2 u|^{\alpha+1} \, \d x \Big)^{\frac{3-2\beta_k}{3\alpha-1}}.
\end{equation}
\bigskip

\noindent\textbf{Step~3. Combination of estimates of Steps 1 and 2.}
We write $\beta$ as a convex combination
\begin{equation*}
\beta = \frac{2}{\alpha+1} \varrho_k + \beta_k (1-\varrho_k), \quad \varrho_k \in [0,1], \quad \beta_k \uparrow 3 \frac{\alpha+2}{5\alpha+3},
\end{equation*}
where $\beta_k$ is as in \eqref{eq:betak_dxu} and where
\[
\varrho_k = \frac{(\alpha+1) (\beta_k-\beta)}{(\alpha+1) \beta_k - 2} \quad \text{and} \quad 1-\varrho_k = \frac{(\alpha+1) \beta - 2}{(\alpha+1) \beta_k - 2},
\]
so that we obtain the estimate 
\begin{align*}
\sup_{x \ge 0} x^{\beta} |\partial_x u| \;\quad &\le_{\phantom{\alpha,k}} \; \big(\sup_{x \ge 0} x^{\frac{2}{\alpha+1}} |\partial_x u|\big)^{\varrho_k} \big(\sup_{x \ge 0} x^{\beta_k} |\partial_x u|\big)^{1-\varrho_k} \\
&\stackrel{\mathclap{\eqref{eq:weight xu'3}, \eqref{eq:weight xu'6}}}{\lesssim}_{\alpha,k} \; \Big(\int_0^\infty  x^{\alpha+2} \, |\partial_x^2 u|^{\alpha+1} \, \d x\Big)^{\frac{\varrho_k}{\alpha +1}} \Big(  \int_0^\infty  u^2 \, \d x \Big)^{\frac{(\alpha+1)\beta_k-2}{3\alpha-1} (1-\varrho_k)} \\
&\phantom{\lesssim_{\alpha,k}} \; \times \Big( \int_0^\infty x^{\alpha+2}
 		|\partial_x^2 u|^{\alpha+1} \, \d x \Big)^{\frac{3-2\beta_k}{3\alpha-1} (1-\varrho_k)} \\
   &=_{\phantom{\alpha,k}} \; \Big(\int_0^\infty u^{2} \, \d x\Big)^\gamma  \Big(\int_0^\infty x^{\alpha+2} \, |\partial_x^2 u|^{\alpha+1} \, \d x\Big)^\delta. 
\end{align*}
The exponents \(\gamma\) and \(\delta\) can be readily computed as
\begin{equation*}
\gamma = \frac{(\alpha+1) \beta_k - 2}{3 \alpha - 1} (1-\varrho_k) = \frac{(\alpha +1)\beta-2}{3\alpha -1}, \qquad \delta = \frac{\varrho_k}{\alpha+1} + \frac{3-2\beta_k}{3 \alpha - 1} (1-\varrho_k) = \frac{3-2\beta}{3\alpha -1}.
\end{equation*}
Since $\beta_k \uparrow 3 \frac{\alpha+2}{5\alpha+3}$, this proves \eqref{eq:weight xu'} for $\beta \in \big[\frac{2}{\alpha +1}, \frac{3(\alpha+2)}{5\alpha +3}\big)$. 

\bigskip

\noindent\textbf{Step~4. Extension from $C^\infty_\mathrm{c}((0,\infty))$ to $U_\alpha$.}
The argument is analogous to Step~4 of Lemma~\ref{lem:weight xu}, that is, a sequence $u_k \in C_\mathrm{c}^\infty((0,\infty))$ is such that $u_k \to u$ in $U_\alpha$, which implies that $x^\beta \partial_x u_k \to \varphi$ uniformly by \eqref{eq:weight xu'}. This implies $x^\beta \partial_x u = \varphi$ and taking the limit in \eqref{eq:weight xu'} with $u = u_k$ finishes the proof.
\end{proof}
\begin{lemma}\label{lem:xu-int}
For $\alpha >1$, the following weighted estimate holds:
\begin{equation}\label{eq:weight xu lp}
\int_0^\infty x^{\nu p} | u|^p \, \d x \lesssim_{\alpha,\nu,p} \Big(\int_0^\infty u^2 \, \d x\Big)^{\gamma} \Big(\int_0^\infty x^{\alpha+2} \, |\partial_x^2 u|^{\alpha+1} \, \d x\Big)^{\delta}
\end{equation}
for all $u \in U_\alpha$, where the exponents $\gamma$ and $\delta$ are given by 
\begin{equation*}
\gamma = \frac{(\alpha -1)p + (\alpha +1)(\nu p +1)}{3\alpha -1} \quad \text{ and } \quad \delta= \frac{(1-2\nu)p-2}{3\alpha -1}
\end{equation*}
and one of the following conditions is satisfied:
\begin{enumerate}[(i)]
\item\label{it:lp-cond1} it holds $2 \leq p \leq \alpha +1$ and $\nu$ must satisfy $\nu = 0$ if $p = 2$ and
\begin{equation*}
\nu \in \Big[- \frac{\alpha(p-2)}{(\alpha -1)p}, \frac{(\alpha +2)(p-2)}{(5\alpha +3)p}\Big)
\end{equation*}
if $p > 2$;
\item\label{it:lp-cond2} it holds $p \geq \alpha +1$ and $\nu$  must satisfy 
\begin{equation*}
 \nu \in \Big[- \frac{\alpha +1 + (\alpha-1)p}{(\alpha +1)p}, \frac{(\alpha +2)(p-2)}{(5\alpha +3)p}\Big).
\end{equation*}
\end{enumerate}
\end{lemma}

\begin{proof}
We first establish \eqref{eq:weight xu lp} for $u \in C^\infty_\mathrm{c}((0,\infty))$
and then extend to $u \in U_\alpha$ by approximation at the end of the proof. The proof follows the same reasoning as the proofs of Lemmata \ref{lem:weight xu} and \ref{lem:weight xu'}.

\bigskip

\noindent\textbf{Step~1. Weighted estimate and lower bound for \(\nu\).}
We begin with the case $2\leq p \leq \alpha+1$ and set $\nu =- \frac{(p-2)\alpha}{p(\alpha-1)}$.
Applying H\"older’s inequality followed by Hardy’s inequality (cf.~Lemma~\ref{lem:Hardy}), we obtain
\begin{align}
\int_0^\infty x^{- \frac{\alpha (p-2)}{\alpha-1}} \, |u|^p \, \d x 
\;\, &=_{\phantom\alpha} \int_0^\infty |u|^{\frac{2(\alpha +1 -p)}{\alpha-1}} \, x^{- \frac{\alpha (p-2)}{\alpha-1}} \,|u|^{\frac{(\alpha+1)(p-2)}{\alpha-1}} \, \d x \nonumber \\
& \leq_{\phantom\alpha} \Big( \int_0^\infty u^2 \, \d x\Big)^{\frac{\alpha+1-p}{\alpha -1}} \Big( \int_0^\infty x^{-\alpha} | u|^{\alpha +1} \, \d x \Big)^{\frac{p-2}{\alpha-1}} \nonumber \\
& \stackrel{\mathclap{\eqref{eq:Hardy}}}{\lesssim}_{\alpha} \Big( \int_0^\infty u^2 \, \d x\Big)^{\frac{\alpha+1-p}{\alpha -1}} \Big( \int_0^\infty x^{\alpha+2} |\partial_x^2 u|^{\alpha +1} \, \d x \Big)^{\frac{p-2}{\alpha-1}}, \label{eq:weight xu lp1}
\end{align}
where we have used
\[
p \ge 2 \quad \Rightarrow \quad \frac{2(\alpha +1 -p)}{\alpha-1} \le 2 \qquad \text{and} \qquad p \le \alpha+1 \quad \Rightarrow \quad \frac{(\alpha+1)(p-2)}{\alpha-1} \le \alpha+1.
\]
For $p\geq \alpha +1$, we apply H\"older's inequality again, which gives
\begin{equation}\label{eq:weight xu lp2}
\int_0^\infty x^{\nu p} \, |u|^p \, \d x 
\leq   \Big( \sup_{x \ge 0} x^{ \frac{1-\alpha}{\alpha+1}} \, |u|\Big)^{p-(\alpha +1)} 
\int_0^\infty x^{\nu p + (p-(\alpha +1))\frac{(\alpha-1)}{\alpha+1} } \, |u|^{\alpha +1} \, \d x.
\end{equation}
We require the condition
\begin{equation*}
\nu p + (p-(\alpha +1))\frac{(\alpha-1)}{\alpha+1} = -\alpha
\end{equation*}
so that Hardy's inequality \eqref{eq:Hardy} can be used on the second term. Solving for $\nu$, we obtain
\begin{equation*}
\nu = - \frac{(\alpha -1) p + \alpha +1}{p(\alpha +1)}.
\end{equation*}
With this choice of $\nu$, it follows from \eqref{eq:weight xu lp2} and Lemma~\ref{lem:weight xu}
\begin{align}\label{eq:weight xu lp3}
\int_0^\infty x^{ - \frac{p(\alpha -1) + \alpha +1}{\alpha +1}} \, |u|^p \, \d x 
\;\, &\leq_{\phantom\alpha}   \Big( \sup_{x \ge 0} x^{ \frac{1-\alpha}{\alpha+1}} \, |u|\Big)^{p-(\alpha +1)}
\int_0^\infty x^{-\alpha} \, |u|^{\alpha +1} \, \d x  \nonumber \\
& \stackrel{\mathclap{\eqref{eq:weight xu}}}{\lesssim}_{\alpha} \Big( \int_0^\infty x^{\alpha+2} |\partial_x^2 u|^{\alpha +1} \, \d x \Big)^{\frac{p}{\alpha+1}}.
\end{align}
This is confirms the claimed inequality \eqref{eq:weight xu lp} for $p\geq \alpha +1$, with the critical case where the $L^2$-norm of $u$ does not appear at the lower bound $\nu = - \frac{p(\alpha -1) + \alpha +1}{p(\alpha +1)}$. 
Thus, we have proved \eqref{eq:weight xu lp} for both cases \eqref{it:lp-cond1} and \eqref{it:lp-cond2} and the respective lower bounds of $\nu$, completing Step~1. 

\bigskip

\noindent\textbf{Step~2. Weighted estimate and upper bound for $\nu$.} Let $p > 2$  and $\nu \in \R$. 
Applying H\"older's inequality, we arrive at
\begin{equation*}
\int_0^\infty x^{\nu p} \, |u|^p \, \d x 
\leq \Big( \sup_{x \ge 0}x^{\frac{\nu p}{p-2}} \,|u|\Big)^{p-2} \int_0^\infty u^2 \, \d x.
\end{equation*}
For $\nu \in \big[\frac{(1-\alpha)(p-2)}{(\alpha +1) p}, \frac{(\alpha+2) (p-2)}{(5\alpha +3) p}\big)$ we can apply Lemma \ref{lem:weight xu} to the first-term on the right-hand side, which gives
\begin{align}
\int_0^\infty x^{\nu p} \, |u|^p \, \d x 
\,\; &\leq_{\phantom\alpha} \Big( \sup_{x \ge 0}x^{\frac{\nu p}{p-2}} \,|u|\Big)^{p-2} \int_0^\infty u^2 \, \d x \nonumber \\
&\stackrel{\mathclap{\eqref{eq:weight xu}}}{\lesssim}_{\alpha} \Big( \int_0^\infty u^2 \, \d x\Big)^{\frac{(\alpha -1) p + (\alpha+1) (\nu p+1)}{3\alpha -1}} \Big( \int_0^\infty x^{\alpha+2} |\partial_x^2 u|^{\alpha +1} \, \d x \Big)^{\frac{p-2 - 2 \nu p}{3\alpha-1}}. \label{eq:weight xu lp4}
\end{align}
This establishes the upper bound for $\nu$ in \eqref{eq:weight xu lp}. 
\bigskip

\noindent\textbf{Step~3. Combination of estimates of Steps 1 and 2.}  We write $\nu$ as a convex combination
\begin{equation*}
    \nu = \begin{cases} -\frac{\alpha (p-2)}{(\alpha -1) p} \varrho + \tilde \nu (1-\varrho) & \text{for $2 \leq p \leq \alpha +1$,}\\[0.4cm]
     - \frac{\alpha +1 + (\alpha-1)p}{(\alpha +1) p} \varrho + \tilde \nu (1-\varrho) & \text{for $ p \geq \alpha +1$,}
    \end{cases}
  \end{equation*}
where $\varrho \in [0,1]$ and $\tilde\nu \in \big[\frac{(1-\alpha)(p-2)}{(\alpha +1) p}, \frac{(\alpha+2) (p-2)}{(5\alpha +3) p}\big)$. We can explicitly compute
 \[
 \varrho = \begin{cases} \frac{(\alpha-1) (\tilde\nu-\nu) p}{\alpha (p-2) + (\alpha-1) \tilde\nu p} & \text{for $2 \leq p \leq \alpha +1$,}\\[0.4cm]
 \frac{(\alpha+1) (\tilde\nu-\nu) p}{(\alpha-1) p + (\alpha+1) (\tilde\nu p + 1)} & \text{for $ p \geq \alpha +1$,}
 \end{cases}
 \]
 and
\[
 1-\varrho = \begin{cases} \frac{\alpha (p-2) + (\alpha-1) \nu p}{\alpha (p-2) + (\alpha-1) \tilde\nu p} & \text{for $2 \leq p \leq \alpha +1$,}\\[0.4cm]
 \frac{(\alpha-1) p + (\alpha+1) (\nu p + 1)}{(\alpha-1) p + (\alpha+1) (\tilde\nu p + 1)} & \text{for $ p \geq \alpha +1$.}
 \end{cases}
 \]
Using H\"older's inequality we infer
\begin{align}
   \int_0^\infty x^{\nu p} \, |u|^p \, \d x 
     &\le \begin{cases} \big(\int_0^\infty x^{-\frac{\alpha (p-2)}{(\alpha -1)}} |u|^p \, \d x\big)^\varrho \big(\int_0^\infty x^{\tilde\nu p} |u|^p \, \d x\big)^{1-\varrho} & \text{for $2 \leq p \leq \alpha +1$,} \\[0.4cm]
     \big(\int_0^\infty x^{-\frac{\alpha +1 + (\alpha-1)p}{\alpha +1}} |u|^p \, \d x\big)^\varrho \big(\int_0^\infty x^{\tilde\nu p} |u|^p \, \d x\big)^{1-\varrho} & \text{for $p \geq \alpha +1$,} \end{cases} \nonumber \\[0.2cm]
     &\lesssim_{\alpha,\nu,p} \Big( \int_0^\infty u^2 \, \d x\Big)^{\gamma} \Big( \int_0^\infty x^{\alpha+2} |\partial_x^2 u|^{\alpha +1} \, \d x \Big)^{\delta}, \label{eq:scaling lp}
 \end{align}
 where estimates \eqref{eq:weight xu lp1}, \eqref{eq:weight xu lp3}, and \eqref{eq:weight xu lp4} were used in the last step, and the exponents \(\gamma\) and \(\delta\) satisfy
\begin{equation*}
 \gamma = \begin{cases}\frac{\alpha+1-p}{\alpha-1} \varrho + \frac{(\alpha-1) p + (\alpha+1) (\tilde\nu p + 1)}{3\alpha-1} (1-\varrho) & \text{for $2 \leq p \leq \alpha +1$,} \\[0.2cm]
 \frac{(\alpha-1) p + (\alpha+1) (\tilde\nu p + 1)}{3\alpha-1} (1-\varrho) & \text{for $p \geq \alpha +1$,} \end{cases} \Bigg\} = \frac{(\alpha -1)p + (\alpha +1)(\nu p +1)}{3\alpha -1} 
 \end{equation*}
 and
 \begin{equation*}
 \delta = \begin{cases}\frac{p-2}{\alpha-1} \varrho + \frac{p-2-2\tilde{\nu} p}{3\alpha-1} (1-\varrho) & \text{for $2 \leq p \leq \alpha +1$,} \\[0.2cm]
 \frac{p}{\alpha+1} \varrho + \frac{p-2-2\tilde{\nu} p}{3\alpha-1} (1-\varrho) & \text{for $p \geq \alpha +1$,} \end{cases} \Bigg\} = \frac{p(1-2\nu) -2}{3\alpha -1}.
 \end{equation*}
 This proves \eqref{eq:weight xu lp} provided conditions \eqref{it:lp-cond1} or \eqref{it:lp-cond2} are satisfied.
\bigskip

 \noindent\textbf{Step~4. Extension from $C^\infty_\mathrm{c}((0,\infty))$ to $U_\alpha$.}
 For $u \in U_\alpha$, let $(u_k)_{k \in \N} \subset C^\infty_\mathrm{c}((0,\infty))$
 with $u_k \to u$ in $U_\alpha$. Since Lemma~\ref{lem:weight xu} has been
 extended to $U_\alpha$, applying it to $u_k - u$ gives $u_k \to u$ pointwise,
 and Fatou's lemma yields
 \begin{align*}
 \int_0^\infty x^{\nu p} |u|^p \, \d x
 &\leq \liminf_{k \to \infty} \int_0^\infty x^{\nu p} |u_k|^p \, \d x
 \\
 &\lesssim_{\alpha,\nu,p} 
 \lim_{k \to \infty} \|u_k\|_{L^2((0,\infty))}^{\gamma}
 \big\|x^{\frac{\alpha+2}{\alpha+1}} \partial_x^2 u_k\big\|_{L^{\alpha+1}((0,\infty))}^{\delta}
 = \|u\|_{L^2((0,\infty))}^{\gamma}
 \big\|x^{\frac{\alpha+2}{\alpha+1}} \partial_x^2 u\big\|_{L^{\alpha+1}((0,\infty))}^{\delta},   
 \end{align*}
 which is \eqref{eq:weight xu lp} for $u \in U_\alpha$.
\end{proof}
\begin{lemma}\label{lem:weight xu lp'}
Let $\alpha >1$. Then the following estimate holds:
 \begin{equation}\label{eq:weight xu lp'}
\int_0^\infty x^{\nu p} |\partial _x u|^p \, \d x \lesssim_\alpha \Big(\int_0^\infty u^2 \, \d x\Big)^\gamma \Big(\int_0^\infty x^{\alpha+2} \, |\partial_x^2 u|^{\alpha+1} \, \d x\Big)^\delta
 \end{equation}
 for all $u \in U_\alpha$, where the exponents $\gamma$ and $\delta$ are given by
 \begin{equation*}
     \gamma = \frac{\alpha+1+((\alpha+1)\nu - 2) p}{3\alpha -1} \quad \text{ and } \quad \delta= \frac{(3-2\nu) p - 2}{3\alpha -1},
 \end{equation*}
where either $p = \frac{4(\alpha+1)}{\alpha+3}$ and $\nu = \frac{\alpha+2}{2(\alpha+1)}$, or $p>\frac{4(\alpha+1)}
{\alpha+3}$ and 
 \begin{equation*}
\nu \in \Big[\frac{2}{\alpha +1} + \frac{2(\alpha-2)}{(\alpha +3) p}, \frac{\alpha +2}{5\alpha +3} \big(3 - \tfrac 2 p \big)\Big) 
\end{equation*}
\end{lemma}
\begin{proof}
We first prove the estimate for $u \in C^\infty_\mathrm{c}((0,\infty))$ and extend to $u \in U_\alpha$ at the end of the proof by approximation.

\bigskip

\noindent\textbf{Step~1. Critical case $p = \frac{4(\alpha+1)}{\alpha+3}$.}
We start by approaching the upper bound on $\nu$.
From estimate~\eqref{eq:recursion-dxu} of Lemma~\ref{eq:recursion-dxu}, that is,
\begin{equation*}
\int_0^\infty x^{\nu p} |\partial_x u|^p \, \d x
 \lesssim_{\alpha,\nu,p}   \Big(  \int_0^\infty  u^2 \, \d x \Big)^{\frac{\alpha+1}{\alpha+3}(  1- \theta^k)}
\Big( \int_0^\infty x^{\gamma_k}
 |\partial_x u|^{r_k} \, \d x \Big)^{\theta^k}
 \Big( \int_0^\infty  x^{\alpha+2} |\partial_x^2 u|^{\alpha+1} \, \d x\Big)^{\frac{2}{\alpha+3}(1- \theta^k)}
\end{equation*}
and equations~\eqref{eq:gamma1-r1}, that is,
\begin{align*}
\gamma_k = \theta^{-k} \Big( \nu p - \frac{2(\alpha+2)}{\alpha+3} \Big) + \frac{2(\alpha+2)}{\alpha+3} \quad \text{and} \quad
r_k = \theta^{-k} \Big( p - \frac{4(\alpha+1)}{\alpha+3} \Big) + \frac{4(\alpha+1)}{\alpha+3},
\end{align*}
where $\theta = \frac{\alpha-1}{2(\alpha+1)}$, we find for the critical case $p=\frac{4(\alpha +1)}
{\alpha +3}$ and $\nu = \frac{\alpha +2}{2 (\alpha +1)}$ that $\gamma_k = \nu p$ and $r_k = p$, so 
that \eqref{eq:recursion-dxu} simplifies to
\begin{equation}\label{eq:weight xu lp'2}
\int_0^\infty x^{\frac{2(\alpha +2)}{\alpha+3}} |\partial_x u|^{\frac{4(\alpha +1)}{\alpha+3}} \, \d x \lesssim_{\alpha} \Big( \int_0^\infty  u^2 \, \d x \Big)^{\frac{\alpha +1}{\alpha +3}}  \Big( \int_0^\infty x^{\alpha+2}
|\partial_x^2 u|^{\alpha+1} \, \d x \Big)^{\frac{2}{\alpha+3}}.
\end{equation}
\bigskip

\noindent\textbf{Step~2. Super-critical Case $p > \frac{4(\alpha+1)}{\alpha+3}$.}
With estimate~\eqref{eq:weight xu lp'2} in hand, we improve upon it by considering general $p > \frac{4(\alpha+1)}{\alpha+3}$ and $\nu$, allowing us to obtain a refined bound
\begin{align}
\int_0^\infty x^{\nu p} |\partial_x u|^{p} \, \d x &\leq \Big(\sup_{x \ge 0} x^{\nu p -\frac{2(\alpha +2)}{\alpha+3}} |\partial_x u|^{p-\frac{4(\alpha +1)}{\alpha+3}}\Big) \int_0^\infty x^{\frac{2(\alpha +2)}{\alpha+3}} |\partial_x u|^{\frac{4(\alpha +1)}{\alpha+3}} \, \d x \nonumber \\
& \lesssim_{\alpha} \sup_{x \ge 0} x^{\nu p -\frac{2(\alpha +2)}{\alpha+3}} |\partial_x u|^{p-\frac{4(\alpha +1)}{\alpha+3}} \Big( \int_0^\infty  u^2 \, \d x \Big)^{\frac{\alpha +1}{\alpha +3}}  \Big( \int_0^\infty x^{\alpha+2}
|\partial_x^2 u|^{\alpha+1} \, \d x \Big)^{\frac{2}{\alpha+3}}. \label{eq:weight xu lp'3}
\end{align}
For the first term on the right-hand side of \eqref{eq:weight xu lp'3}, we apply Lemma \ref{lem:weight xu'} 
and therefore choose $\nu$ such that 
\begin{equation*}
    \frac{2}{\alpha +1} \le \frac{\nu p - \frac{2(\alpha +2)}{\alpha +3}}{p-\frac{4(\alpha +1)}{\alpha +3}} < \frac{3(\alpha +2)}{5\alpha +3}
\end{equation*}
Since $p > \frac{4(\alpha +1)}{\alpha +3}$, we get after rearranging 
\begin{equation}\label{eq:nu}
 \frac{2}{\alpha +1} + \frac{2(\alpha-2)}{(\alpha +3) p} 
 \le \nu < \frac{\alpha+2}{5\alpha+3} \big(3 - \tfrac 2 p\big) .
\end{equation}
Thus, the combination of estimate \eqref{eq:weight xu'} of Lemma \ref{lem:weight xu'} and \eqref{eq:weight xu lp'3} gives the estimate 
\begin{equation}\label{eq:weight xu lp'4}
\int_0^\infty x^{\nu p} |\partial_x u|^{p} \, \d x
\lesssim_{\alpha}  \Big( \int_0^\infty  u^2 \, \d x \Big)^{\gamma}  \Big( \int_0^\infty x^{\alpha+2}
|\partial_x^2 u|^{\alpha+1} \, \d x \Big)^{\delta},
\end{equation}
when either $p = \frac{4(\alpha+1)}{\alpha+3}$ and $\nu = \frac{\alpha+2}{2(\alpha+1)}$, or
$p>\frac{4(\alpha+1)}{\alpha+3}$ and 
 \begin{equation*}
\nu \in \Big[\frac{2}{\alpha +1} + \frac{2(\alpha-2)}{(\alpha +3) p}, \frac{\alpha +2}{5\alpha +3}
\big( 3-\tfrac{2}{p}\big)\Big).
\end{equation*}
The exponents $\gamma$ and $\delta$ follow from scaling balance and satisfy 
\begin{equation*}
p= 2\gamma + (\alpha +1)\delta \qquad \text{and} \qquad 
\nu p +1 -p= \gamma + \delta (1-\alpha). 
\end{equation*}
Solving for \(\gamma\) and \(\delta\), we obtain
\begin{equation*}
\gamma = \frac{\alpha+1+((\alpha+1)\nu - 2) p}{3\alpha -1} \quad \text{ and } \quad \delta= \frac{(3-2\nu) p - 2}{3\alpha -1}
\end{equation*}
as desired. 
\bigskip

\noindent\textbf{Step~3. Extension from $C^\infty_\mathrm{c}((0,\infty))$ to $U_\alpha$.}
This follows as in Step~4 of the proof of Lemma~\ref{lem:xu-int}, that is, for $u \in U_\alpha$ and $(u_k)_{k \in \N} \subset C^\infty_\mathrm{c}((0,\infty))$
with $u_k \to u$ in $U_\alpha$ we have $\partial_x u_k \to \partial_x u$ point-wise by \eqref{eq:weight xu'} of Lemma~\ref{lem:weight xu'}, and the claim \eqref{eq:weight xu lp'} for $u$ follows with Fatou's lemma.
\end{proof}
\subsection{Asymptotics and interpolation estimates in $V_\alpha$}
\label{sec:valpha}
In what follows, it is convenient to introduce the space
\begin{equation}\label{def-valpha}
V_\alpha \coloneq \Big\{v \in U_\alpha \colon \int_0^\infty (\partial_x^2 (x^{\alpha+2} g_\alpha(\partial_x^2 v)))^2 \, \d x < \infty\Big\}.
\end{equation}
\begin{lemma}\label{lem-valpha}
Suppose $\alpha > 1$. For $v \in V_\alpha$ define $w \coloneq x^{\alpha+2} g_\alpha(\partial_x^2 v)$. Then
\begin{subequations}\label{asymptotic-vw}
\begin{align}
w &= a x (1+o(x^{\frac 1 2})) && \text{as } x \downarrow 0, \label{as-w} \\
\partial_x w &= a (1+o(x^{\frac 1 2})) && \text{as } x \downarrow 0, \label{as-dxw} \\
v &= c x^{\frac{\alpha-1}\alpha} (1+o(x^{\frac 1 2})) && \text{as } x \downarrow 0, \label{as-d2xw} \\
\partial_x v &= \frac{\alpha-1}{\alpha} c x^{-\frac 1 \alpha} (1+o(x^{\frac 1 2})) && \text{as } x \downarrow 0, \\
\partial_x^2 v &= \frac{1-\alpha}{\alpha^2} c x^{- \frac{\alpha+1}{\alpha}} (1+o(x^{\frac 1 2})) && \text{as } x \downarrow 0,
\end{align}
\end{subequations}
where $\partial_x w(0) = a = g_\alpha((1-\alpha) c/\alpha^2)$ and $c \in \R$. Furthermore,
\begin{equation}\label{fundamental-w}
w(x) = a x + \int_0^x \int_0^{x'} (\partial_x^2 w)(x'') \, \d x'' \, \d x' \qquad \text{and} \qquad (\partial_x w)(x) = a + \int_0^x (\partial_x^2 w)(x') \, \d x'.
\end{equation}
\end{lemma}
\begin{proof}
We first begin by noticing that the regularity assumptions on $v$ entail that $\partial_x^2 w \in L^2(0,\infty)$, so that two integrations yield
\[
w(x) = a x + b + \int_0^x \int_0^{x'} (\partial_x^2 w)(x'') \, \d x'' \, \d x',
\]
for constants $a,b \in \R$. If $b \ne 0$, we have
\[
x^{\alpha+2} |\partial_x^2 v|^{\alpha+1} = x^{-\frac{\alpha+2}{\alpha}} |w|^{\frac{\alpha+1}{\alpha}} = |b|^{\frac{\alpha+1}{\alpha}} x^{-\frac{\alpha+2}{\alpha}} (1+o(1)) \quad \text{as } x \downarrow 0,
\]
which contradicts $\int_0^\infty x^{\alpha+2} |\partial_x^2 v|^{\alpha+1} \, \d x \stackrel{\eqref{def-ualpha}}{<} \infty$ because $v \in U_\alpha \stackrel{\eqref{def-valpha}}{\subset} V_\alpha$. Hence, we must have $b = 0$, but do have $a \ne 0$ in general (which physically corresponds to a nonzero contact-line velocity, see the heuristic discussion in \S\ref{sec:kernel}). This yields \eqref{fundamental-w}. For $w$ this entails the asymptotics \eqref{as-w}, \eqref{as-dxw}, \eqref{as-d2xw} on using
\begin{align*}
\int_0^x \int_0^{x'} (\partial_x^2 w)(x'') \, \d x'' \, \d x' &\le \|\partial_x^2 w\|_{L^2(0,x)} \int_0^x \Big(\int_0^{x'} \d x''\Big)^{\frac 1 2} \, \d x' = o(x^{\frac 3 2}) && \text{as } x \downarrow 0, \\
\int_0^x (\partial_x w)(x') \, \d x' &\le \|\partial_x w\|_{L^2(0,x)} \Big(\int_0^{x} \d x'\Big)^{\frac 1 2} = o(x^{\frac 1 2}) && \text{as } x \downarrow 0, \\
\int_0^\infty (\partial_x^2 w)^2 \, \d x &< \infty \quad \Rightarrow \quad \partial_x^2 w = o(x^{-\frac 1 2}) && \text{as } x \downarrow 0.
\end{align*}
The terms with $v$ follow on identifying $a = g_\alpha((1-\alpha) c/\alpha^2)$.
\end{proof}
\begin{lemma}\label{lem-convex}
For $\alpha > 1$ it holds $|x_1-x_2| \le 2^{\frac{\alpha-1}{\alpha}} |g_\alpha(x_1)-g_\alpha(x_2)|^{\frac 1 \alpha}$ for $x_1,x_2 \in \R$.
\end{lemma}
\begin{proof}
For $x_1,x_2 \in \R$ with $x_1 \ne x_2$ it holds
\begin{align*}
(g_\alpha(x_1)-g_\alpha(x_2))(x_1-x_2) &= (x_1-x_2)^2 \int_0^1 g_\alpha'(x_2 + \tau (x_1-x_2)) \, \d \tau \\
&= \alpha |x_1-x_2|^{\alpha+1} \int_0^1 \big|\tfrac{x_2}{x_1-x_2} + \tau\big|^{\alpha-1} \, \d\tau \\
&\ge \alpha |x_1-x_2|^{\alpha+1} \int_{-\frac 1 2}^{\frac 1 2} |\tau|^{\alpha-1} \, \d\tau = 2^{1-\alpha} |x_1-x_2|^{\alpha+1},
\end{align*}
which implies the desired result.
\end{proof}
\begin{lemma}\label{lem:controlhigher}
For $\alpha >1$ suppose $v \in V_\alpha$. Then with the notation of Lemma~\ref{lem-valpha} it holds
 \begin{align}
\int_0^\infty x^{2\alpha} |\partial_x^2 (v -c x^{\frac{\alpha-1}{\alpha}})|^{2\alpha} \, \d x \lesssim_\alpha \int_0^\infty x^{2\alpha} (g_\alpha(\partial_x^2 v) - a x^{-1-\alpha})^2 \, \d x
&\lesssim_{\alpha} \int_0^\infty (\partial_x^2 w)^2 \, \d x \,  \label{eq:proofcontrolhigher1}
\end{align}
where $w \coloneq x^{\alpha+2} g_\alpha(\partial_x^2 v)$.
\end{lemma}
\begin{proof}
We use \eqref{fundamental-w} of Lemma~\ref{lem-valpha} and apply Hardy’s inequality twice, which allows us to estimate the weighted integral on the left-hand side of \eqref{eq:proofcontrolhigher1},
\begin{align}
\int_0^\infty x^{2\alpha} (g_\alpha(\partial_x^2 v) - a x^{-1-\alpha})^2 \, \d x &= \int_0^\infty x^{-4} (w - ax)^2 \, \d x \lesssim_{\alpha} \int_0^\infty x^{-2} (\partial_x w - a)^2 \, \d x \nonumber\\
&\lesssim_{\alpha} \int_0^\infty (\partial_x^2 w)^2 \, \d x, \label{eq:proofcontrolhigher1-alt}
\end{align}
and verifies the second estimate in \eqref{eq:proofcontrolhigher1}. The first follows from $a = g_\alpha((1-\alpha) c/\alpha^2)$ by Lemma~\ref{lem-valpha} and Lemma~\ref{lem-convex}.
\end{proof}
\begin{lemma}\label{lem-sup-wx}
Suppose $\alpha > 1$. For $v \in V_\alpha$ and $w \coloneq x^{\alpha+2} g_\alpha(\partial_x^2 v)$ it holds
\begin{equation}\label{sup-wx}
\sup_{x \ge 0} x^{-\frac 3 2} |w - a x| + \sup_{x \ge 0} x^{-\frac 1 2} |\partial_x w - a| \lesssim \Big(\int_0^\infty (\partial_x^2 w)^2 \, \d x\Big)^{\frac 1 2},
\end{equation}
where $a \coloneq \partial_x w(0)$ is as in Lemma~\ref{lem-valpha}. Furthermore, we have
\begin{equation}\label{est-coeff-a}
|(\partial_x w)(0)| = |a| \lesssim_\alpha \Big(\int_0^\infty x^{\alpha+2} |\partial_x^2 v|^{\alpha+1} \, \d x\Big)^{\frac{\alpha}{3\alpha-1}} \Big(\int_0^\infty (\partial_x^2 w)^2 \, \d x\Big)^{\frac{\alpha-1}{3\alpha-1}}.
\end{equation}
Lastly, we have with $\beta \in [- \frac 3 2, -\frac{3\alpha+4}{2(2\alpha+1)}]$
\begin{align}\label{sup-w-alt}
\sup_{x \ge 0} x^{\beta} |w - \delta a x| \lesssim_\alpha \Big(\int_0^\infty x^{\alpha+2} |\partial_x^2 v|^{\alpha+1} \, \d x\Big)^{\frac{\alpha (2\beta+3)}{3\alpha-1}} \Big(\int_0^\infty (\partial_x^2 w)^2 \, \d x\Big)^{\frac{- (\alpha+1) \beta - 2}{3\alpha-1}}, \quad \delta \coloneq \begin{cases} 1 & \text{if } \beta < - 1, \\ 0 & \text{if } \beta \ge -1. 
\end{cases}
\end{align} 
\end{lemma}
\begin{proof}
We split the proof into three parts.

\bigskip

\noindent\textbf{Step~1. Proof of \eqref{sup-wx}.}
Using Lemma~\ref{lem-valpha}, it follows
\begin{align*}
\sup_{x \ge 0} x^{-3} (w-ax)^2 \quad &\stackrel{\mathclap{\eqref{as-w}}}{\le} \quad 3 \int_0^\infty x^{-4} (w-ax)^2 \, \d x + 2 \int_0^\infty x^{-3} |w-ax| |\partial_x w - a| \, \d x \\
&\le \quad 4 \int_0^\infty x^{-4} (w-ax)^2 \, \d x + \int_0^\infty x^{-2} (\partial_x w - a)^2 \, \d x \lesssim \int_0^\infty (\partial_x^2 w)^2 \, \d x, \\
\sup_{x \ge 0} x^{-1} (\partial_x w - a)^2 \quad &\stackrel{\mathclap{\eqref{as-dxw}}}{\le} \quad \int_0^\infty x^{-2} (\partial_x w - a)^2 \, \d x + 2 \int_0^\infty x^{-1} |\partial_x w - a| |\partial_x^2 w| \, \d x \\
&\le \quad 2 \int_0^\infty x^{-2} (\partial_x w - a)^2 \, \d x + \int_0^\infty (\partial_x^2 w)^2 \, \d x \lesssim \int_0^\infty (\partial_x^2 w)^2 \, \d x,
\end{align*}
where the fundamental theorem of calculus as well as Young's and Hardy's inequality were used.

\bigskip

\noindent\textbf{Step~2. Proof of \eqref{est-coeff-a}.}
Observe that by the triangle and Cauchy-Schwarz inequality
\begin{align*}
|a| &\lesssim \int_1^2 |w - a x| \, \d x + \int_1^2 |w| \, \d x \lesssim_{\alpha}  \Big(\int_0^\infty x^{-4} (w-ax)^2 \, \d x\Big)^{\frac 1 2} + \Big(\int_0^\infty x^{\alpha+2} |\partial_x^2 v|^{\alpha+1} \, \d x\Big)^{\frac{\alpha}{\alpha+1}} \\
&\lesssim_{\alpha}  \Big(\int_0^\infty (\partial_x^2 w)^2 \, \d x\Big)^{\frac 1 2} + \Big(\int_0^\infty x^{\alpha+2} |\partial_x^2 v|^{\alpha+1} \, \d x\Big)^{\frac{\alpha}{\alpha+1}},
\end{align*}
where we have applied Hardy's inequality in the last line. Scaling $x \mapsto \lambda x$ with $\lambda > 0$ leads to
\[
|a| \lesssim_\alpha \lambda^{-\frac 1 2} \Big(\int_0^\infty (\partial_x^2 w)^2 \, \d x\Big)^{\frac 1 2} + \lambda^{\frac{\alpha-1}{\alpha+1}} \Big(\int_0^\infty x^{\alpha+2} |\partial_x^2 v|^{\alpha+1} \, \d x\Big)^{\frac{\alpha}{\alpha+1}},
\]
so that after optimizing in $\lambda$ we arrive at \eqref{est-coeff-a}.

\bigskip

\noindent\textbf{Step~3. Proof of \eqref{sup-w-alt}.}
For a cut-off function $\eta \in C^\infty([0,\infty))$ such that $0 \le \eta \le 1$, $\eta|_{[0,1]} = 1$, and $\eta_{|[2,\infty)} = 0$, we have
\begin{align*}
\sup_{x \ge 0} x^{- \frac{3\alpha+4}{2\alpha}} |w-\eta a x|^{\frac{2\alpha+1}{\alpha}} 
&\leq \frac{3\alpha+4}{2\alpha} \int_0^\infty x^{- \frac{5\alpha+4}{2\alpha}} |w-\eta a x|^{\frac{2\alpha+1}{\alpha}} \, \d x \\
&\phantom{\lesssim} + \frac{2\alpha+1}{\alpha} \int_0^\infty x^{- \frac{3\alpha+4}{2\alpha}} |w-\eta a x|^{\frac{\alpha+1}{\alpha}} |\partial_x (w - \eta a x)| \, \d x \\
&\lesssim_\alpha \big(\sup_{x \ge 0} x^{-\frac 3 2} |w-\eta ax| + \sup_{x \ge 0} x^{-\frac 1 2} |\partial_x (w - \eta a x)|\big) \\
&\phantom{\lesssim_\alpha} \times \int_0^\infty x^{- \frac{\alpha+2}{\alpha}} |w-\eta a x|^{\frac{\alpha+1}{\alpha}} \, \d x \\
&\lesssim_\alpha \big(\sup_{x \ge 0} x^{-\frac 3 2} |w- ax| + \sup_{x \ge 0} x^{-\frac 1 2} |\partial_x w - a| + |a| \sup_{x \ge 1} (x^{-\frac 1 2} + x^{\frac 1 2} |\partial_x \eta|)\big) \\
&\phantom{\lesssim_\alpha} \times \Big( \int_0^\infty x^{\alpha+2} |\partial_x^2 v|^{\alpha+1} \, \d x + |a|^{\frac{\alpha+1}{\alpha}} \int_0^2 x^{-\frac 1 \alpha} \, \d x \Big),
\end{align*}
which implies with \eqref{sup-wx}, \eqref{est-coeff-a}, $\delta = 1$ if $1 < \alpha < 2$, and $\delta = 0$ if $\alpha \ge 2$,
\begin{align*}
& \sup_{x \ge 0} x^{- \frac{3\alpha+4}{2\alpha}} |w - \delta a x|^{\frac{2\alpha+1}{\alpha}} \\
&\quad \lesssim_\alpha \bigg(\Big(\int_0^\infty (\partial_x^2 w)^2 \, \d x\Big)^{\frac 1 2} + \Big(\int_0^\infty x^{\alpha+2} |\partial_x^2 v|^{\alpha+1} \, \d x\Big)^{\frac{\alpha}{3\alpha-1}} \Big(\int_0^\infty (\partial_x^2 w)^2 \, \d x\Big)^{\frac{\alpha-1}{3\alpha-1}} \bigg) \\
& \quad \phantom{\lesssim_\alpha} \times \bigg(\int_0^\infty x^{\alpha+2} |\partial_x^2 v|^{\alpha+1} \, \d x  + \Big(\int_0^\infty x^{\alpha+2} |\partial_x^2 v|^{\alpha+1} \, \d x\Big)^{\frac{\alpha+1}{3\alpha-1}} \Big(\int_0^\infty (\partial_x^2 w)^2 \, \d x\Big)^{\frac{\alpha^2-1}{\alpha (3\alpha-1)}}\bigg).
\end{align*}
Scaling both $v$ and $x$ leads to
\[
\sup_{x \ge 0} x^{- \frac{3\alpha+4}{2 (2\alpha+1)}} |w - \delta a x| \lesssim_\alpha \Big(\int_0^\infty x^{\alpha+2} |\partial_x^2 v|^{\alpha+1} \, \d x\Big)^{\frac{\alpha}{2\alpha+1}} \Big(\int_0^\infty (\partial_x^2 w)^2 \, \d x\Big)^{\frac{\alpha}{2 (2\alpha+1)}}.
\]
The combination with \eqref{sup-wx} yields \eqref{sup-w-alt} on noting that the exponents on the right-hand side of \eqref{sup-w-alt} follow from scaling in $x$ and $v$ (see Step~3 in the proof of Lemma~\ref{lem:weight xu} for an analogous case).
\end{proof}
\begin{lemma}\label{lem-sup-uxx}
We have for $\beta \in [\frac{2\alpha+1}{2\alpha},\frac{4\alpha+7}{2(2\alpha+1)}]$,
\begin{align}\nonumber
\sup_{x \ge 0} x^\beta |\partial_x^2 (v - \delta c x^{\frac{\alpha-1}{\alpha}})| &\lesssim_\alpha \Big(\int_0^\infty x^{\alpha+2} |\partial_x^2 v|^{\alpha+1} \, \d x\Big)^{\frac{2\beta\alpha-2\alpha-1}{3\alpha-1}} \Big(\int_0^\infty (\partial_x^2 w)^2 \, \d x\Big)^{\frac{\alpha+3 - (\alpha+1)\beta}{3\alpha-1}},  \\
\delta &\coloneq \begin{cases} 1 & \text{if } \beta < \frac{\alpha+1}{\alpha}, \\ 0 & \text{if } \beta \ge \frac{\alpha+1}{\alpha}, \end{cases} \label{sup-uxx-alt}
\end{align}
for all $v \in V_\alpha$, where $w \coloneq x^{\alpha+2} g_\alpha(\partial_x^2 v)$, and $c$ is as in Lemma~\ref{lem-valpha}. Furthermore,
\begin{equation}\label{est-coeff-c}
|c| \lesssim_\alpha \Big(\int_0^\infty x^{\alpha+2} |\partial_x^2 v|^{\alpha+1} \, \d x\Big)^{\frac{1}{3\alpha-1}} \Big(\int_0^\infty (\partial_x^2 w)^2 \, \d x\Big)^{\frac{\alpha-1}{\alpha(3\alpha-1)}}.
\end{equation}
\end{lemma}
\begin{proof}
Note that estimate~\eqref{est-coeff-c} is immediate from \eqref{est-coeff-a} of Lemma~\ref{lem-sup-wx} on noting that by Lemma~\ref{lem-valpha} it holds $|c| = \frac{\alpha^2}{\alpha-1} |a|^{\frac 1 \alpha}$, where $a = (\partial_x w)(0)$.

\bigskip

Next, observe that from \eqref{sup-wx} of Lemma~\ref{lem-sup-wx} it follows
\[
\sup_{x \ge 0} x^{\alpha\beta} \big| g_\alpha(\partial_x^2 v) - a x^{-\alpha-1}\big|^{\frac{1}{\alpha}} \lesssim_\alpha \Big(\int_0^\infty x^{\alpha+2} |\partial_x^2 v|^{\alpha+1} \, \d x\Big)^{\frac{\alpha (2\beta\alpha-2\alpha-1)}{3\alpha-1}} \Big(\int_0^\infty (\partial_x^2 w)^2 \, \d x\Big)^{\frac{\alpha (\alpha+3 - (\alpha+1)\beta)}{3\alpha-1}}.
\]
With $g_\alpha((\alpha-1) c / \alpha^2) = a$ (cf.~Lemma~\ref{lem-valpha}), we infer by Lemma~\ref{lem-convex}
\begin{align*}
\sup_{x \ge 0} x^\beta \big|\partial_x^2 (v-c x^{\frac{\alpha-1}{\alpha}})\big| &\lesssim_\alpha \sup_{x \ge 0} x^{\alpha\beta} \big| g_\alpha(\partial_x^2 v) - a x^{-\alpha-1}\big|^{\frac{1}{\alpha}} \lesssim_\alpha \Big(\int_0^\infty (\partial_x^2 w)^2 \, \d x\Big)^{\frac{1}{2\alpha}} \\
&\lesssim_\alpha \Big(\int_0^\infty x^{\alpha+2} |\partial_x^2 v|^{\alpha+1} \, \d x\Big)^{\frac{2\beta\alpha-2\alpha-1}{3\alpha-1}} \Big(\int_0^\infty (\partial_x^2 w)^2 \, \d x\Big)^{\frac{\alpha+3 - (\alpha+1)\beta}{3\alpha-1}},
\end{align*}
which is \eqref{sup-uxx-alt}.
\end{proof}
\begin{lemma}\label{lem-sup-u-alt}
Suppose $\alpha >1$. For $\beta \in [\frac{1-2\alpha}{2\alpha},\frac{1-\alpha}{\alpha+1}]$, we have
\begin{align}\nonumber
\sup_{x \ge 0} x^\beta |v - \delta c x^{\frac{\alpha-1}{\alpha}}| &\lesssim_\alpha \Big(\int_0^\infty x^{\alpha+2} |\partial_x^2 v|^{\alpha+1} \, \d x\Big)^{\frac{2\beta\alpha+2\alpha-1}{3\alpha-1}} \Big(\int_0^\infty (\partial_x^2 w)^2 \, \d x\Big)^{\frac{1-\alpha- (\alpha+1)\beta}{3\alpha-1}},  \\
\delta &\coloneq \begin{cases} 1 & \text{if } \beta < \frac{1-\alpha}{\alpha}, \\ 0 & \text{if } \beta \ge \frac{1-\alpha}{\alpha}, \end{cases} \label{sup-u-alt}
\end{align}
for all $v \in V_\alpha$, where $w \coloneq x^{\alpha+2} g_\alpha(\partial_x^2 v)$, and $c$ is as in Lemma~\ref{lem-valpha}.
\end{lemma}
\begin{proof}
We have with the fundamental theorem of calculus, 
\begin{align*}
\sup_{x \ge 0} \big(x^{\frac{1-2\alpha}{2\alpha}} |v - c x^{\frac{\alpha-1}{\alpha}}|\big)^{2\alpha} &\lesssim_\alpha \int_0^\infty x^{1-2\alpha} |v - c x^{\frac{\alpha-1}{\alpha}}|^{2\alpha-1} |\partial_x(v - c x^{\frac{\alpha-1}{\alpha}})| \, \d x \\
&\phantom{\lesssim_\alpha} + \int_0^\infty x^{-2\alpha} |v - c x^{\frac{\alpha-1}{\alpha}}|^{2\alpha} \, \d x \\
&\lesssim_\alpha \int_0^\infty x^{-2\alpha} |v - c x^{\frac{\alpha-1}{\alpha}}|^{2\alpha} \, \d x + \int_0^\infty |\partial_x(v - c x^{\frac{\alpha-1}{\alpha}})|^{2\alpha} \, \d x \\
&\lesssim_\alpha \int_0^\infty x^{2\alpha} |\partial_x^2(v - c x^{\frac{\alpha-1}{\alpha}})|^{2\alpha} \, \d x \\
&\stackrel{\mathclap{\eqref{eq:proofcontrolhigher1}}}{\lesssim_\alpha} \int_0^\infty (\partial_x^2 w)^2 \, \d x,
\end{align*}
where Young's and Hardy's inequality have been used in the second and third step, respectively, and Lemma~\ref{lem:controlhigher} in the last step. Combining this with~\eqref{eq:weight xu} of Lemma~\ref{lem:weight xu} for
$\beta = \tfrac{1-\alpha}{\alpha+1}$, we obtain~\eqref{sup-u-alt} by scaling
in $x$ and $v$ (cf.~Step~3 in the proof of Lemma~\ref{lem-sup-wx}).
\end{proof}
\begin{lemma}\label{lem-sup-ux-alt}
Suppose $\alpha>1$. For $\beta \in [\frac{1}{2\alpha},\frac{2}{\alpha+1}]$, we have
\begin{align}\nonumber
\sup_{x \ge 0} x^\beta |\partial_x (v - \delta c x^{\frac{\alpha-1}{\alpha}})| &\lesssim_\alpha \Big(\int_0^\infty x^{\alpha+2} |\partial_x^2 v|^{\alpha+1} \, \d x\Big)^{\frac{2\beta\alpha-1}{3\alpha-1}} \Big(\int_0^\infty (\partial_x^2 w)^2 \, \d x\Big)^{\frac{2 - (\alpha+1)\beta}{3\alpha-1}},  \\
\delta &\coloneq \begin{cases} 1 & \text{if } \beta < \frac{1}{\alpha}, \\ 0 & \text{if } \beta \ge \frac{1}{\alpha}, \end{cases} \label{sup-ux-alt} 
\end{align}
for all $v \in V_\alpha$, where $w \coloneq x^{\alpha+2} g_\alpha(\partial_x^2 v)$, and $c$ is as in Lemma~\ref{lem-valpha}.
\end{lemma}
\begin{proof}
We have with the fundamental theorem of calculus, 
\begin{align*}
\sup_{x \ge 0} \big(x^{\frac{1}{2\alpha}} |\partial_x(v - c x^{\frac{\alpha-1}{\alpha}})|\big)^{2\alpha} &\lesssim_\alpha \int_0^\infty x |\partial_x(v - c x^{\frac{\alpha-1}{\alpha}})|^{2\alpha-1} |\partial_x^2(v - c x^{\frac{\alpha-1}{\alpha}})| \, \d x \\
&\phantom{\lesssim_\alpha} + \int_0^\infty |\partial_x(v - c x^{\frac{\alpha-1}{\alpha}})|^{2\alpha} \, \d x \\
&\lesssim_\alpha \int_0^\infty |\partial_x(v - c x^{\frac{\alpha-1}{\alpha}})|^{2\alpha} \, \d x + \int_0^\infty x^{2\alpha} |\partial_x^2(v - c x^{\frac{\alpha-1}{\alpha}})|^{2\alpha} \, \d x \\
&\lesssim_\alpha \int_0^\infty x^{2\alpha} |\partial_x^2(v - c x^{\frac{\alpha-1}{\alpha}})|^{2\alpha} \, \d x \\
&\stackrel{\mathclap{\eqref{eq:proofcontrolhigher1}}}{\lesssim_\alpha} \int_0^\infty (\partial_x^2 w)^2 \, \d x,
\end{align*}
where Young's and Hardy's inequality have been used in the second and third step, respectively, and Lemma~\ref{lem:controlhigher} in the last step. Combining this with~\eqref{eq:weight xu'} of Lemma~\ref{lem:weight xu'} for 
$\beta = \tfrac{2}{\alpha+1}$, we obtain~\eqref{sup-ux-alt} by scaling in $x$ and $v$ (cf.~Step~3 in the proof of Lemma~\ref{lem-sup-wx}).
\end{proof}

In the sequel, we will frequently use the following estimate, which can be derived from H\"older's inequality and estimates~\eqref{sup-uxx-alt} and \eqref{est-coeff-c} of Lemma~\ref{lem-sup-uxx} with $\beta= \frac{\alpha+2}{\alpha+1}$.
\begin{lemma}\label{lem:con-lem-sup-uxx}
Let $\alpha > 2$. Then we have
\begin{equation}\label{eq:con-lem-sup-uxx}
    \Big(\int_0^\infty \big(x^{\alpha+2} |\partial_x^2 v|^{\alpha +1}\big)^2 \, \d x\Big)^{\frac{1}{2}} \lesssim_\alpha 
    \Big(\int_0^\infty x^{\alpha +2} \, |\partial_x^2 v|^{\alpha +1} \, \d x\Big)^{\frac{2\alpha-1}{3\alpha-1}} 
    \Big(\int_0^\infty (\partial_x^2 w)^2 \, \d x\Big)^{\frac{\alpha+1}{2(3\alpha-1)}}
\end{equation}
for all $v \in V_\alpha$, where $w = x^{\alpha+2} g_\alpha(\partial_x^2 v) = x^{\alpha+2} |\partial_x^2 v|^{\alpha-1} \partial_x^2 v$ as above.
\end{lemma}

\begin{proof}
Let $\eta \in C^\infty([0,\infty))$ be a cut off with $0 \le \eta \le 1$, $\eta|_{[0,1]} = 1$, and $\eta|_{[2,\infty)} = 0$. Then, with the notation of Lemma~\ref{lem-valpha} we obtain
\begin{align*}
   &\Big( \int_0^\infty \big(x^{\alpha+2} |\partial_x^2 v|^{\alpha +1}\big)^2 \, \d x\Big)^{\frac{1}{2}} 
   \\
   &\quad\lesssim_\alpha 
   \Big( \int_0^\infty \big(x^{\alpha+2} |\partial_x^2 v + \tfrac{\alpha-1}{\alpha^2} \eta c x^{- \frac{\alpha+1}{\alpha}}|^{\alpha +1} + |\eta|^{\alpha+1} |c|^{\alpha+1} x^{-\frac 1 \alpha}\big) \big(x^{\alpha+2} |\partial_x^2 v|^{\alpha +1}\big) \, \d x\Big)^{\frac{1}{2}} 
   \\
   &\quad\lesssim_\alpha 
   \Big( \sup_{x\geq 0} x^{\frac{\alpha+2}{\alpha+1}} \, |\partial_x^2 (v - c x^{\frac{\alpha-1}{\alpha}})| + |c|\Big)^{\frac{\alpha +1}{2}} \Big( \int_0^\infty x^{\alpha+2} |\partial_x^2 v|^{\alpha +1} \, \d x\Big)^{\frac{1}{2}} 
   \\
   &\quad \phantom{\lesssim_\alpha} + |c|^{\frac{\alpha+1}{2}} \Big(\int_0^2 x^{-\frac 2 \alpha} \, \d x\Big)^{\frac 1 2} 
   \Big(\int_0^\infty \big(x^{\alpha+2} |\partial_x^2 v|^{\alpha +1}\big)^2 \, \d x\Big)^{\frac{1}{2}}.
\end{align*}
Since $\alpha > 2$ and thus $\int_0^2 x^{- \frac 2 \alpha} \, \d x < \infty$, we may apply Young's inequality and absorb the last term on the right-hand side into the left-hand side leading to
\begin{align*}
    &\Big( \int_0^\infty \big(x^{\alpha+2} |\partial_x^2 v|^{\alpha +1}\big)^2 \, \d x\Big)^{\frac{1}{2}} \\
    & \quad \lesssim_\alpha 
    \Big( \sup_{x\geq 0} x^{\frac{\alpha+2}{\alpha+1}} \, |\partial_x^2 (v - c x^{\frac{\alpha-1}{\alpha}})| + |c|\Big)^{\frac{\alpha +1}{2}} \Big( \int_0^\infty x^{\alpha+2} |\partial_x^2 v|^{\alpha +1} \, \d x\Big)^{\frac{1}{2}} + |c|^{\alpha+1}.
\end{align*}
Furthermore, applying \eqref{sup-uxx-alt} and \eqref{est-coeff-c} of Lemma~\ref{lem-sup-uxx} yields
\begin{align*}
    \Big( \int_0^\infty \big(x^{\alpha+2} |\partial_x^2 v|^{\alpha +1}\big)^2 \, \d x\Big)^{\frac{1}{2}}
    &\lesssim_{\alpha}
    \Big( \int_0^\infty x^{\alpha +2} \, |\partial_x^2 v|^{\alpha +1} \, \d x\Big)^{\frac{2\alpha-1}{3\alpha-1}} 
    \Big( \int_0^\infty (\partial_x^2 w)^2 \, \d x\Big)^{\frac{\alpha+1}{2(3\alpha-1)}} \\
    &\phantom{\lesssim_\alpha}
    + \Big( \int_0^\infty x^{\alpha +2}\, |\partial_x^2 v|^{\alpha +1}\, \d x\Big)^{\frac{2\alpha}{3\alpha-1}} 
    \Big( \int_0^\infty (\partial_x^2 w)^2 \, \d x\Big)^{\frac{\alpha^2-1}{2\alpha(3\alpha-1)}} 
    \\
    &\phantom{\lesssim_\alpha} + \Big(\int_0^\infty x^{\alpha+2} |\partial_x^2 v|^{\alpha+1} \, \d x\Big)^{\frac{\alpha+1}{3\alpha-1}} 
    \Big(\int_0^\infty (\partial_x^2 w)^2 \, \d x\Big)^{\frac{\alpha^2-1}{\alpha(3\alpha-1)}}.
\end{align*}
Scaling of $x$ entails that the last two lines can be discarded, so that we end up with \eqref{eq:con-lem-sup-uxx}.
\end{proof}
\begin{lemma}\label{lem:B2-AC}
For $v \in V_\alpha$ with $w \coloneq x^{\alpha+2} g_\alpha(\partial_x^2 v)$, it holds 
\begin{equation}\label{eq:B2-AC-1}
\int_0^\infty x^{\alpha+2} |\partial_x^2 v|^{\alpha+1}\, \d x = \int_0^\infty (\partial_x^2 w)\, v\, \d x. 
\end{equation}
In particular,
\begin{equation}\label{eq:B2-AC}
\int_0^\infty x^{\alpha+2} |\partial_x^2 v|^{\alpha+1}\, \d x
\le \Big(\int_0^\infty v^2\, \d x\Big)^{\frac 1 2}
\Big(\int_0^\infty (\partial_x^2 w)^2\, \d x\Big)^{\frac 1 2}.
\end{equation}
\end{lemma}
\begin{proof}
Let $(v_n)_{n \in \N} \subset C^\infty_\mathrm{c}((0,\infty))$ with
$v_n \to v$ in $U_\alpha$ by Lemma~\ref{lem-approx-ualpha}.
Two integrations by parts
give
\[
\int_0^\infty (\partial_x^2 w)\, v_n\, \d x
= \int_0^\infty w\, \partial_x^2 v_n\, \d x
= \int_0^\infty x^{\alpha+2}\, g_\alpha(\partial_x^2 v)\,
\partial_x^2 v_n\, \d x .
\]
As $n \to \infty$, the left-hand side converges to
$\int_0^\infty (\partial_x^2 w)\, v\, \d x$ by the Cauchy-Schwarz inequality, 
while the right-hand side converges to
$\int_0^\infty x^{\alpha+2}\, |\partial_x^2 v|^{\alpha+1}\, \d x$
by H\"older's inequality. 
Hence \eqref{eq:B2-AC-1} follows. 
\end{proof}
\section{Higher-order nonlinear estimates\label{sec:higher-est}}
In \S\ref{sec:leading-order estimates} we have heuristically motivated the leading-order estimates \eqref{est-lead-weak-max} that are sufficient to give the von-Mises transform \eqref{mises} a meaning. Estimates~\eqref{est-lead-weak-max} will be rigorously derived in \S\ref{sec:time-discretization}. However, in order to treat all nonlinear contributions, this requires suitable estimates on the right-hand side terms
\[
\int_0^\infty f \, u \, \d x \qquad \text{and} \qquad \int_0^\infty f^2 \, \d x
\]
of \eqref{est-lead-weak-max}, where $f = N(u)$ is given by \eqref{def-nu}.
\begin{proposition}\label{prop-est-nv}
Let $\alpha > 2$ and suppose that $v \in V_\alpha$ with  $\|v\|_{U_\alpha} \ll_\alpha 1$. Then, we have the inequality
\begin{align}\label{est-nv}
    \int_0^\infty \big(N(v)\big)^2\, \, \d x
    &\lesssim_\alpha
    A^\frac{2(\alpha-1)}{3\alpha-1}
    B^\frac{2}{3\alpha-1}
    C+
    A^\frac{4(\alpha-1)}{3\alpha-1}
    B^\frac{4}{3\alpha-1}
    C, 
\end{align}
where
\begin{equation}\label{eq:def_ABC_2}
    A \coloneq 
    \int_0^\infty v^2\, \d x,
    \quad
    B \coloneq \int_0^\infty x^{\alpha+2} |\partial_x^2 v|^{\alpha+1}\, \d x,
    \quad \text{and} \quad
    C \coloneq \int_0^\infty (\partial_x^2 w)^2\, \d x.
\end{equation}
\end{proposition}
\begin{proof}
Using the definition \eqref{def-nu} of $N(\cdot)$ and the shorthand $w := x^{\alpha+2} g_\alpha(\partial_x^2 v)$, we obtain 
\begin{equation*}
\begin{split}
    \int_0^\infty (N(v))^2\, \d x
    &\lesssim
    \int_0^\infty \bigl(1 -(1+v)^\frac{3}{2}\bigr)^2\, (\partial_x^2 w)^2\, \d x + \int_0^\infty (1+v)^3 \big(\partial_x^2\bigl(\big(1 - (1+v)^\frac{\alpha}{2}\big) w\big)^2\, \d x
    \\
    &
    \eqqcolon I_1 + I_2.
\end{split}
\end{equation*}
We estimate the integrals $I_1$ and $I_2$ separately.

\bigskip

\noindent\textbf{Step 1: Estimate for $I_1$. }
Thanks to estimate~\eqref{eq:weight xu} of Lemma~\ref{lem:weight xu} and $\|v\|_{U_\alpha} \ll_\alpha 1$, we may assume $\sup_{x \geq 0} |v| \leq \frac{1}{2} < 1$, so that in particular $\sup_{x \ge 0} |1-(1+v)^{\frac{3}{2}}| \lesssim \sup_{x \ge 0} |v|$. Thus, with $w = x^{\alpha+2} |\partial_x^2 v|^{\alpha-1} \partial_x^2 v$, we obtain for $I_1$ the estimate
\begin{align}
    I_1
    &\leq
    \bigl(\sup_{x \geq 0} |v|\bigr)^2 \int_0^\infty (\partial_x^2 w)^2\, \d x \; \stackrel{\mathclap{\eqref{eq:weight xu}}}{\lesssim}_{\alpha}
    \Big(\int_0^\infty v^2\, \d x\Big)^\frac{2(\alpha-1)}{3\alpha-1} 
    \Big(\int_0^\infty x^{\alpha+2} |\partial_x^2 v|^{\alpha+1}\, \d x\Big)^\frac{2}{3\alpha-1}
    \int_0^\infty (\partial_x^2 w)^2\, \d x \nonumber \\
    &\stackrel{\mathclap{\eqref{eq:def_ABC_2}}}{=} \;
    A^\frac{2(\alpha-1)}{3\alpha-1}
    B^\frac{2}{3\alpha-1}
    C. \label{nonlinear-i1}
\end{align}
\bigskip

\noindent\textbf{Step 2: Estimate for $I_2$. }
The estimate for $I_2$ is more involved. First, by computing 
\begin{align*}
    \partial_x^2\big(\big(1-(1+v)^{\frac{\alpha}{2}}\big)\, w\big)
    &= 
    - \tfrac{\alpha (\alpha-2)}{4} (1+v)^{\frac{\alpha -4}{2}} (\partial_x v)^2 \, w - \tfrac{\alpha}{2} (1+v)^{\frac{\alpha -2}{2}} (\partial_x^2 v) \, w \\
    & \quad - 
    \alpha (1+v)^{\frac{\alpha -2}{2}} (\partial_x v) \, (\partial_x w) + \big(1-(1+v)^{\frac{\alpha}{2}}\big) \, (\partial_x^2 w), 
\end{align*}
and using again that $\sup_{x \geq 0} |v| \leq \frac{1}{2} < 1$ thanks to estimate~\eqref{eq:weight xu} of Lemma~\ref{lem:weight xu} and $\|v\|_{U_\alpha} \ll_\alpha 1$, we obtain
\begin{align}\nonumber
    I_2
    &\lesssim_\alpha
    \int_0^\infty (\partial_x v)^4\, w^2\, \d x
    +
    \int_0^\infty (\partial_x^2 v)^2\, w^2\, \d x
    +
    \int_0^\infty (\partial_x v)^2\, (\partial_x w)^2\, \d x
    +
    (\sup_{x\geq 0} |v|)^2 \int_0^\infty (\partial_x^2 w)^2\, \d x
    \\
    &\eqcolon
    I_{2,1}
    +
    I_{2,2}
    +
    I_{2,3}
    +
    I_{2,4}. \label{i2-split}
\end{align}
\bigskip

\noindent\textbf{Step 2.1: Estimate for $I_{2,1}$. }
For estimating $I_{2,1}$, observe that through integration by parts we have 
\begin{align}
I_{2,1} &\stackrel{\mathclap{\eqref{i2-split}}}{=} \int_0^\infty (\partial_x v)^4 w^2 \, \d x = - 3 \int_0^\infty v (\partial_x v)^2 (\partial_x^2 v) w^2 \, \d x - 2 \int_0^\infty v (\partial_x v)^3 w (\partial_x w) \, \d x \nonumber \\
&= -\frac 3 2 \int_0^\infty (\partial_x v^2) (\partial_x v) x^{2 \alpha + 4} |\partial_x^2 v|^{2\alpha} (\partial_x^2 v) \, \d x - \int_0^\infty (\partial_x v^2) (\partial_x v)^2 w (\partial_x w) \, \d x \nonumber \\
&= - \frac 3 2 \int_0^\infty (\partial_x v^2) (\partial_x v) x^{-\frac{\alpha+2}{\alpha}} |w|^{\frac{\alpha+1}{\alpha}} w \, \d x + 2 \int_0^\infty v^2 (\partial_x v) (\partial_x^2 v) w (\partial_x w) \, \d x \nonumber \\
&\phantom{=} + \int_0^\infty v^2 (\partial_x v)^2 (\partial_x w)^2 \, \d x + \int_0^\infty v^2 (\partial_x v)^2 w (\partial_x^2 w) \, \d x \nonumber \\
&= \frac 3 2 \int_0^\infty v^2 (x^{\alpha+2} |\partial_x^2 v|^{\alpha+1})^2 \, \d x - \frac{\alpha+2}{2\alpha} \int_0^\infty (\partial_x v^3) x^{- \frac{2 (\alpha+1)}{\alpha}} |w|^{\frac{\alpha+1}{\alpha}} w \, \d x \nonumber \\
&\phantom{=} + \frac{10\alpha+3}{6 \alpha} \int_0^\infty (\partial_x v^3) x^{-\frac{\alpha+2}{\alpha}} |w|^{\frac{\alpha+1}{\alpha}} (\partial_x w) \, \d x + \frac 1 3 \int_0^\infty (\partial_x v^3) (\partial_x v) (\partial_x w)^2 \, \d x \nonumber \\
&\phantom{=} + \int_0^\infty v^2 (\partial_x v)^2 w (\partial_x^2 w) \, \d x \nonumber \\
&= \frac 3 2 \int_0^\infty v^2 (x^{\alpha+2} |\partial_x^2 v|^{\alpha+1})^2 \, \d x - \frac{(\alpha+1)(\alpha+2)}{\alpha^2} \int_0^\infty v^2 x^{2(\alpha+2)} (x^{-2} v) |\partial_x^2 v|^{2\alpha} (\partial_x^2 v) \, \d x \nonumber \\
&\phantom{=} + \frac{(\alpha+2)(8\alpha+3)}{3\alpha^2} \int_0^\infty v^3 (x^{\alpha+2}|\partial_x^2 v|^\alpha (\partial_x^2 v)) (x^{-1} \partial_x w) \, \d x \nonumber\\
&\phantom{=} - \frac{(3\alpha+1)(4\alpha+3)}{6\alpha^2} \int_0^\infty v^3 (\partial_x^2 v) (\partial_x w)^2 \, \d x \nonumber \\
&\phantom{=} - \frac{10\alpha+3}{6\alpha} \int_0^\infty v^3 (x^{\alpha+2} |\partial_x^2 v|^{\alpha+1}) (\partial_x^2 w) \, \d x + \int_0^\infty v^2 (\partial_x v)^2 (x^{\alpha+2}|\partial_x^2 v|^{\alpha-1} (\partial_x^2 v)) (\partial_x^2 w) \, \d x \nonumber \\
&\phantom{=} - \frac 2 3 \int_0^\infty v^3 (\partial_x v) (\partial_x w) (\partial_x^2 w) \, \d x \eqqcolon \sum_{j = 1}^7 I_{2,1,j}. \label{est-i21}
\end{align}
We note that all boundary terms arising from the integration by parts in \eqref{est-i21} vanish. Indeed, as $x\downarrow0$,
Lemma~\ref{lem-valpha} dictates the asymptotic behavior
\begin{equation*}
v=O\big(x^{\frac{\alpha-1}{\alpha}}\big),\qquad
\partial_x v=O\big(x^{-\frac1\alpha}\big),\qquad
w=O(x),\qquad
\partial_x w=O(1),
\end{equation*}
which ensures that every corresponding boundary expression is 
$O\big(x^{\frac{3\alpha-4}{\alpha}}\big)\to 0$ as 
$x\downarrow0$,
since $\alpha>2$.
As $x\to\infty$, using Lemma~\ref{lem:weight xu} and Lemma~\ref{lem:weight xu'}, we choose
\begin{equation*}
\frac{1-\alpha}{\alpha+1} \le \beta_1 < \frac{\alpha+2}{5\alpha+3},
\qquad
\frac{2}{\alpha+1} \le \beta_2 < \frac{3(\alpha+2)}{5\alpha+3}
\end{equation*}
sufficiently close to the upper endpoints, so that by \eqref{eq:weight xu} and \eqref{eq:weight xu'} we have
\begin{equation*}
|v(x)|\lesssim_{\alpha, \beta_1} x^{-\beta_1},
\qquad
|\partial_x v(x)|\lesssim_{\alpha, \beta_2} x^{-\beta_2}
\qquad\text{for }x>0. 
\end{equation*}
Moreover, \eqref{sup-uxx-alt} and \eqref{est-coeff-c} of Lemma~\ref{lem-sup-uxx} with $\beta=\frac{\alpha+2}{\alpha+1}$ imply
\begin{equation*}
|\partial_x^2 v(x)|\lesssim_\alpha x^{-\frac{\alpha+2}{\alpha+1}}
\qquad \text{for }x \geq 1,
\end{equation*}
since $c x^{-\frac{2 \alpha+1}{\alpha}}$ decays faster than $x^{-\frac{\alpha+2}{\alpha+1}}$. Consequently,
\begin{equation*}
|w(x)|=x^{\alpha+2}|\partial_x^2 v(x)|^\alpha
\lesssim_\alpha x^{\frac{\alpha+2}{\alpha+1}},
\qquad
|\partial_x w(x)|\lesssim_\alpha 1+x^{\frac{1}{2}},
\end{equation*}
where the latter follows from \eqref{fundamental-w}. Therefore each boundary term in
\eqref{est-i21} is of order $O(x^\gamma)$ with the exponent $\gamma$ taking values in 
\begin{align*}
\gamma &\in \{\gamma_1,\gamma_2,\gamma_3,\gamma_4,\gamma_5,\gamma_6\}, \\
\gamma_1 &\coloneq -\beta_1-3\beta_2+\frac{2(\alpha+2)}{\alpha+1} \downarrow - \frac{4 (\alpha+2)}{(\alpha+1) (5\alpha+3)} && \text{for} \quad v(\partial_x v)^3 w^2, \\
\gamma_2 &\coloneq -2\beta_1-2\beta_2+\frac{
3\alpha+5}{2(\alpha+1)} \downarrow - \frac{\alpha^2+14\alpha+17}{2(\alpha+1)(5\alpha+3)} && \text{for} \quad v^2(\partial_x v)^2 w\,(\partial_x w), \\
\gamma_3 &\coloneq -2\beta_1-\beta_2+\frac{\alpha+2}{\alpha+1} \downarrow - \frac{2(\alpha+2)}{(\alpha+1)(5\alpha+3)} && \text{for} \quad v^2(\partial_x v)x^{-\frac{\alpha+2}{\alpha}}|w|^{\frac{\alpha+1}{\alpha}}w, \\
\gamma_4 &\coloneq -3\beta_1+\frac1{\alpha+1} \downarrow - \frac{3\alpha^2+4\alpha+3}{(\alpha+1)(5\alpha+3)} && \text{for} \quad v^3 x^{-\frac{2(\alpha+1)}{\alpha}}|w|^{\frac{\alpha+1}{\alpha}}w,\\
\gamma_5 &\coloneq -3\beta_1+\frac12 \downarrow - \frac{\alpha+9}{2(5\alpha+3)} && \text{for} \quad v^3 x^{-\frac{\alpha+2}{\alpha}}|w|^{\frac{\alpha+1}{\alpha}} (\partial_x w), \\
\gamma_6 &\coloneq-3\beta_1-\beta_2+1 \downarrow - \frac{\alpha+9}{5\alpha+3} && \text{for} \quad v^3 (\partial_x v)(\partial_x w)^2.
\end{align*}
As indicated, near the right endpoint values of $\beta_1$ and $\beta_2$ all of these exponents are negative. Thus, for $\beta_1$ and $\beta_2$ chosen sufficiently
close to their respective upper bounds, we infer that every boundary term also vanishes as $x\to\infty$.

\bigskip

We estimate the seven remaining integrals $I_{2,1,j}, j=1, \ldots, 7$, separately. For $I_{2,1,1}$ we obtain with Lemma~\ref{lem:weight xu} (with $\beta=0$) and Lemma~\ref{lem:con-lem-sup-uxx},
\begin{subequations}\label{i21j}
\begin{align}\nonumber
    I_{2,1,1} \quad &\le_{\phantom{\alpha}} \quad \frac 3 2 \big(\sup_{x \ge 0} |v|\big)^2 \int_0^\infty (x^{\alpha+2} |\partial_x^2 v|^{\alpha+1})^2 \, \d x \\
    &\stackrel{\mathclap{\eqref{eq:weight xu}, \eqref{eq:con-lem-sup-uxx}}}{\lesssim}_{\alpha} \quad \Big(\int_0^\infty v^2 \, \d x\Big)^{\frac{2(\alpha-1)}{3\alpha-1}} \Big( \int_0^\infty x^{\alpha +2} \, |\partial_x^2 v|^{\alpha +1} \, \d x\Big)^{\frac{4\alpha}{3\alpha-1}} \Big( \int_0^\infty (\partial_x^2 w)^2 \, \d x\Big)^{\frac{\alpha+1}{3\alpha-1}}. \label{i211}
\end{align}
Furthermore, by Lemma~\ref{lem:weight xu}, H\"older's and Hardy's inequality, \eqref{asymptotic-vw} of Lemma~\ref{lem-valpha}, and because $\alpha > 2$, it holds
\begin{align} \nonumber
    I_{2,1,2} \quad &\lesssim_{\alpha} \quad \big(\sup_{x > 0} |v|\big)^2 \Big(\int_0^\infty (x^{\alpha+2} |x^{-2} v|^{\alpha+1})^2 \, \d x\Big)^{\frac{1}{2\alpha+2}} \Big(\int_0^\infty (x^{\alpha+2} |\partial_x^2 v|^{\alpha+1})^2 \, \d x\Big)^{\frac{2\alpha+1}{2\alpha+2}} \\
    &\lesssim_{\alpha} \quad \big(\sup_{x > 0} |v|\big)^2 \int_0^\infty (x^{\alpha+2} |\partial_x^2 v|^{\alpha+1})^2 \, \d x \nonumber \\
    &\stackrel{\mathclap{\eqref{eq:weight xu}, \eqref{eq:con-lem-sup-uxx}}}{\lesssim}_{\alpha} \quad \Big(\int_0^\infty v^2 \, \d x\Big)^{\frac{2(\alpha-1)}{3\alpha-1}} \Big(  \int_0^\infty x^{\alpha +2} \, |\partial_x^2 v|^{\alpha +1} \, \d x\Big)^{\frac{4\alpha}{3\alpha-1}} \Big( \int_0^\infty (\partial_x^2 w)^2 \, \d x\Big)^{\frac{\alpha+1}{3\alpha-1}}. \label{i212}
\end{align}
Now, let $\eta \in C^\infty([0,\infty))$ be a cut off with $0 \le \eta \le 1$, $\eta|_{[0,1]} = 1$, and $\eta|_{[2,\infty)} = 0$. Then, with the notation of Lemma~\ref{lem-valpha} we obtain for $I_{2,1,3}$
\begin{align*}
    I_{2,1,3} \qquad &\lesssim_{\alpha} \qquad \int_0^\infty |v|^3 (x^{\alpha+2}|\partial_x^2 v|^{\alpha+1}) |x^{-1} (\partial_x w - \eta a)| \, \d x + \int_0^\infty |v|^3 x^{\alpha+1}|\partial_x^2 v|^{\alpha+1} \eta |a| \, \d x \\
    &\lesssim_\alpha \qquad \big(\sup_{x > 0}|v|\big)^3 \Big(\int_0^\infty (x^{\alpha+2}|\partial_x^2 v|^{\alpha+1})^2 \, \d x\Big)^{\frac 1 2} \Big(\int_0^\infty (\partial_x^2 w)^2 \, \d x + |a|^2\Big)^{\frac 1 2} \\
    &\phantom{\lesssim_\alpha} \qquad + \big(\sup_{x > 0} x^{-\frac{1}{3}} |v|\big)^3 |a| \int_0^\infty x^{\alpha+2}|\partial_x^2 v|^{\alpha+1} \, \d x \\
    &\stackrel{\mathclap{\eqref{eq:weight xu}, \eqref{est-coeff-a}, \eqref{eq:con-lem-sup-uxx}}}{\lesssim}_{\alpha} \qquad \Big(\int_0^\infty v^2 \, \d x\Big)^{\frac{3 (\alpha-1)}{3\alpha -1}} 
    \Big(\int_0^\infty x^{\alpha+2} \, |\partial_x^2 v|^{\alpha+1} \, \d x\Big)^{\frac{2(\alpha+1)}{3\alpha -1}} 
    \Big(\int_0^\infty (\partial_x^2 w)^2 \, \d x\Big)^{\frac{2\alpha}{3\alpha-1}} \\
    &\phantom{\lesssim_\alpha} \qquad + \Big(\int_0^\infty v^2 \, \d x\Big)^{\frac{3 (\alpha-1)}{3\alpha -1}} 
    \Big(\int_0^\infty x^{\alpha+2} \, |\partial_x^2 v|^{\alpha+1} \, \d x\Big)^{\frac{3\alpha+2}{3\alpha -1}} 
    \Big(\int_0^\infty (\partial_x^2 w)^2 \, \d x\Big)^{\frac 1 2} \\
    &\phantom{\lesssim_\alpha} \qquad + \Big(\int_0^\infty v^2 \, \d x\Big)^{\frac{2(\alpha-2)}{3\alpha -1}} 
    \Big(\int_0^\infty x^{\alpha+2} \, |\partial_x^2 v|^{\alpha+1} \, \d x\Big)^{\frac{4(\alpha+1)}{3\alpha -1}} 
    \Big(\int_0^\infty (\partial_x^2 w)^2 \, \d x\Big)^{\frac{\alpha-1}{3\alpha-1}},
\end{align*}
where Hardy's inequality and the Cauchy-Schwarz inequality were used in the second estimate, and Lemma~\ref{lem:weight xu} and Lemma~\ref{lem-sup-wx} in conjunction with $\alpha>2$ in the last estimate. Scaling $x \mapsto \lambda x$, dividing by $\lambda^{1-2\alpha}$, and taking 
$\lambda \downarrow 0$ eliminates the second line $\sim \lambda^{\frac 1 2}$ of the estimate. The last line is absorbed by the first one using estimate~\eqref{eq:B2-AC} of
Lemma~\ref{lem:B2-AC} in form of
\begin{align*}
\Big(\int_0^\infty x^{\alpha+2} |\partial_x^2 v|^{\alpha+1} \, \d x
\Big)^{\frac{2(\alpha+1)}{3\alpha-1}}
\leq
\Big(\int_0^\infty v^2 \, \d x\Big)^{\frac{\alpha+1}{3\alpha-1}}
\Big(\int_0^\infty (\partial_x^2 w)^2 \, \d x\Big)^{\frac{\alpha+1}{3\alpha-1}},
\end{align*}
so that
\begin{align}
    I_{2,1,3} \lesssim_{\alpha} \Big(\int_0^\infty v^2 \, \d x\Big)^{\frac{3(\alpha-1)}{3\alpha-1}} \Big(\int_0^\infty x^{\alpha +2} \, |\partial_x^2 v|^{\alpha +1} \, \d x\Big)^{\frac{2(\alpha+1)}{3\alpha-1}} \Big( \int_0^\infty (\partial_x^2 w)^2 \, \d x\Big)^{\frac{2\alpha}{3\alpha-1}}. \label{i213}
\end{align}
We next estimate the integral $I_{2,1,4}$. We write
\begin{equation*}
    I_{2,1,4} \quad 
    \lesssim_\alpha
    \int_0^\infty |v|^3\, |\partial_x^2 v| |\partial_x w - a|^2\, \d x
    +
    |a|^2 \int_0^\infty |v|^3\, |\partial_x^2 v|\, \d x
    \eqqcolon
    J_{2,1,4} + K_{2,1,4}
\end{equation*}
and estimate the two integrals $J_{2,1,4}$ and $K_{2,1,4}$ separately. 
For $J_{2,1,4}$, we apply H\"older's inequality and then Lemma~\ref{lem:weight xu} (with $\beta=-\frac 1{\alpha+1}$ and $\beta=0$, respectively) and Lemma \ref{lem-sup-wx} to obtain
\begin{align*}
    J_{2,1,4} 
    &\leq \big( \sup_{x>0} x^{-\frac{1}{2}} |\partial_x w - a|\big)^2
    \int_0^\infty x |v|^3 |\partial_x^2 v| \, dx \nonumber\\
    &\leq \big( \sup_{x>0} x^{-\frac{1}{2}} |\partial_x w - a|\big)^2
    \big(\sup_{x>0} x^{-\frac{1}{\alpha+1}}|v|\big)\big(\sup_{x>0} |v|\big)^\frac{2}{\alpha+1}
    \Big(\int_0^\infty v^2\, \d x\Big)^\frac{\alpha}{\alpha+1}
    \Big(\int_0^\infty x^{\alpha+2} |\partial_x^2 v|^{\alpha+1}\, \d x\Big)^\frac{1}{\alpha+1} \nonumber
    \\
    &\stackrel{\mathclap{\eqref{eq:weight xu},\eqref{sup-wx}}}{\lesssim}_{\alpha}\quad
    \Big(\int_0^\infty v^2\, \d x\Big)^\frac{4(\alpha-1)}{3\alpha-1}
    \Big(\int_0^\infty x^{\alpha+2} |\partial_x^2 v|^{\alpha+1}\, \d x\Big)^\frac{4}{3\alpha-1} \int_0^\infty (\partial_x^2 w)^2\, \d x
\end{align*}
For the integral $K_{2,1,4}$ we apply H\"older's inequality and subsequently Lemma~\ref{lem:xu-int} (with $p=\frac{3(\alpha+1)}{\alpha}$ and $\nu = - \frac{\alpha+2}{3(\alpha+1)}$) and Lemma~\ref{lem-sup-wx} to obtain 
\begin{align*}
    K_{2,1,4} &\leq
    |a|^2 \Big(\int_0^\infty x^{\alpha+2}|\partial_x^2 v|^{\alpha+1}\,  \d x\Big)^{\frac{1}{\alpha+1}} \Big(\int_0^\infty x^{-\frac{\alpha+2}{\alpha}}|v|^{\frac{3(\alpha+1)}{\alpha}}\,  \d x\Big)^{\frac{\alpha}{\alpha+1}}
    \nonumber
    \\
    &\stackrel{\mathclap{\eqref{eq:weight xu lp},\eqref{est-coeff-a}}}{\lesssim}_{\alpha} \quad
    \Big(\int_0^\infty v^2\, \d x\Big)^\frac{3\alpha-5}{3\alpha-1}
    \Big(\int_0^\infty x^{\alpha+2} |\partial_x^2 v|^{\alpha+1}\, \d x\Big)^\frac{2(\alpha+3)}{3\alpha-1} \Big(\int_0^\infty (\partial_x^2 w)^2\, \d x\Big)^{\frac{2(\alpha-1)}{3\alpha-1}}
\end{align*}
We use estimate~\eqref{eq:B2-AC} of
Lemma~\ref{lem:B2-AC} in form of
\begin{align*}
\Big(\int_0^\infty x^{\alpha+2} |\partial_x^2 v|^{\alpha+1} \, \d x
\Big)^{\frac{2(\alpha+1)}{3\alpha-1}}
\leq
\Big(\int_0^\infty v^2 \, \d x\Big)^{\frac{\alpha+1}{3\alpha-1}}
\Big(\int_0^\infty (\partial_x^2 w)^2 \, \d x\Big)^{\frac{\alpha+1}{3\alpha-1}},
\end{align*}
so that
\begin{align*}
    K_{2,1,4}
\lesssim_\alpha
    \Big(\int_0^\infty v^2\, \d x
    \Big)^{\frac{4(\alpha-1)}{3\alpha-1}}
    \Big(\int_0^\infty x^{\alpha+2} |\partial_x^2 v|^{\alpha+1}\, \d x
    \Big)^{\frac{4}{3\alpha-1}}
    \int_0^\infty (\partial_x^2 w)^2\, \d x.
\end{align*}
Combining the estimates for $J_{2,1,4}$ and $K_{2,1,4}$, we conclude
\begin{align}\label{eq:I214}
    I_{2,1,4}
    \lesssim_\alpha \Big(\int_0^\infty v^2\, \d x\Big)^\frac{4(\alpha-1)}{3\alpha-1}
    \Big(\int_0^\infty x^{\alpha+2} |\partial_x^2 v|^{\alpha+1}\, \d x\Big)^\frac{4}{3\alpha-1} \int_0^\infty (\partial_x^2 w)^2\, \d x.
\end{align}
Like \eqref{i213} we can estimate with help of Lemma~\ref{lem:con-lem-sup-uxx},
\begin{align} \nonumber
I_{2,1,5} \quad &\lesssim_{\alpha} \quad \big(\sup_{x > 0} |v|\big)^3 \Big(\int_0^\infty (x^{\alpha+2} |\partial_x^2 v|^{\alpha+1})^2 \, \d x\Big)^{\frac 1 2} \Big(\int_0^\infty (\partial_x^2 w)^2 \, \d x\Big)^{\frac 1 2} \\
&\stackrel{\mathclap{\eqref{eq:weight xu}, \eqref{eq:con-lem-sup-uxx}}}{\lesssim}_{\alpha} \quad \Big(\int_0^\infty v^2 \, \d x\Big)^{\frac{3(\alpha-1)}{3\alpha-1}} \Big(\int_0^\infty x^{\alpha +2} \, |\partial_x^2 v|^{\alpha +1} \, \d x\Big)^{\frac{2(\alpha+1)}{3\alpha-1}} \Big( \int_0^\infty (\partial_x^2 w)^2 \, \d x\Big)^{\frac{2\alpha}{3\alpha-1}}. \label{i215}
\end{align}
Next, we obtain with H\"older's and Hardy's inequality together with Lemma~\ref{lem:weight xu} and Lemma~\ref{lem:weight xu'}
\begin{align} \nonumber
I_{2,1,6} \qquad &\le_{\phantom{\alpha}} \qquad \big(\sup_{x > 0} x^{\frac{\alpha}{5\alpha+3}} |v|\big)^2 \big(\sup_{x > 0} x^{\frac{3(\alpha+1)}{5\alpha+3}} |\partial_x v|\big) \Big(\int_0^\infty (x^{\alpha+2} |x^{-1} \partial_x v|^{\alpha+1})^2 \, \d x\Big)^{\frac{1}{2 (\alpha + 1)}} \\
&\phantom{\le_\alpha} \qquad \times \Big(\int_0^\infty (x^{\alpha+2} |\partial_x^2 v|^{\alpha+1})^2 \, \d x\Big)^{\frac{\alpha}{2 (\alpha + 1)}} \Big( \int_0^\infty (\partial_x^2 w)^2 \, \d x\Big)^{\frac 1 2} \nonumber \\
&\lesssim_{\alpha} \qquad \big(\sup_{x > 0} x^{\frac{\alpha}{5\alpha+3}} |v|\big)^2 \big(\sup_{x > 0} x^{\frac{3(\alpha+1)}{5\alpha+3}} |\partial_x v|\Big) \Big(\int_0^\infty (x^{\alpha+2} |\partial_x^2 v|^{\alpha+1})^2 \, \d x\Big)^{\frac{1}{2}} \Big( \int_0^\infty (\partial_x^2 w)^2 \, \d x\Big)^{\frac 1 2} \nonumber \\
&\stackrel{\mathclap{\eqref{eq:weight xu},\eqref{eq:weight xu'},\eqref{eq:con-lem-sup-uxx}}}{\lesssim}_{\alpha} \qquad \Big(\int_0^\infty v^2 \, \d x\Big)^{\frac{3(\alpha-1)}{3\alpha-1}} \Big(\int_0^\infty x^{\alpha+2} |\partial_x^2 v|^{\alpha+1} \, \d x\Big)^{\frac{2 (\alpha+1)}{3\alpha-1}} \Big( \int_0^\infty (\partial_x^2 w)^2 \, \d x\Big)^{\frac{2\alpha}{3\alpha-1}}. \label{i216}
\end{align}
Likewise, for the integral $I_{2,1,7}$ we obtain 
\begin{align*} 
   I_{2,1,7} 
    &\leq 
    \int_0^\infty |v|^3\, |\partial_x v|\, |\partial_x w - \eta a|\, |\partial_x^2 w|\, \d x
    +
    |a| \int_0^2 |v|^3\, |\partial_x v|\, |\partial_x^2 w|\, \d x \nonumber
    \\
    &
    \eqqcolon
    J_{2,1,7} + K_{2,1,7},
\end{align*}
and we  estimate $J_{2,1,7}$ and $K_{2,1,7}$ separately. For $J_{2,1,7}$ we apply the Cauchy--Schwarz inequality and then Lemma~\ref{lem:weight xu}, Lemma~\ref{lem:weight xu'}, and Lemma~\ref{lem-sup-wx} to obtain
\begin{align*}
    J_{2,1,7} \quad
    &\leq \quad
    \big(\sup_{x > 0} |v|\big) \big(\sup_{x > 0} x^{\frac{\alpha}{5\alpha+3}} |v|\big)^2 \big(\sup_{x > 0} x^{\frac{3(\alpha+1)}{5\alpha+3}} |\partial_x v|\big) 
    \Big(\int_0^\infty x^{-2} (\partial_x w - \eta a)^2\, \d x\Big)^\frac{1}{2}
    \\
    &\phantom{\lesssim_\alpha} \quad
    \times
    \Big(\int_0^\infty (\partial_x^2 w)^2\, \d x\Big)^\frac{1}{2}
    \\
    &\stackrel{\mathclap{\eqref{eq:weight xu},\eqref{eq:weight xu'}}}{\lesssim}_{\alpha} \quad
    \Big(\int_0^\infty v^2\, \d x\Big)^\frac{4(\alpha-1)}{3\alpha-1} 
    \Big(\int_0^\infty x^{\alpha+2} |\partial_x^2 v|^{\alpha+1}\, \d x\Big)^\frac{4}{3\alpha-1}
    \Big(\int_0^\infty x^{-2} (\partial_x w - \eta a)^2\, \d x\Big)^\frac{1}{2}
    \\
    &\phantom{\lesssim_\alpha} \quad
    \times
    \Big(\int_0^\infty (\partial_x^2 w)^2\, \d x\Big)^\frac{1}{2}
\end{align*}
Using Hardy's inequality and then Lemma~\ref{lem-sup-wx}, we end up with the estimate
\begin{align*}
    J_{2,1,7}
    \; &\lesssim_\alpha 
    \Big(\int_0^\infty v^2\, \d x\Big)^\frac{4(\alpha-1)}{3\alpha-1} 
    \Big(\int_0^\infty x^{\alpha+2} |\partial_x^2 v|^{\alpha+1}\, \d x\Big)^\frac{4}{3\alpha-1} 
    \Big(
    \Big(\int_0^\infty (\partial_x^2 w)^2\, \d x\Big)^\frac{1}{2} 
    +
    |a|
    \Big) \\
    &\phantom{\lesssim_\alpha} \times \Big(\int_0^\infty (\partial_x^2 w)^2\, \d x\Big)^\frac{1}{2} \nonumber
    \\
    &\stackrel{\mathclap{\eqref{est-coeff-a}}}{\lesssim}_{\alpha}
    \Big(\int_0^\infty v^2\, \d x\Big)^\frac{4(\alpha-1)}{3\alpha-1} 
    \Big(\int_0^\infty x^{\alpha+2} |\partial_x^2 v|^{\alpha+1}\, \d x\Big)^\frac{4}{3\alpha-1} 
    \int_0^\infty (\partial_x^2 w)^2\, \d x \nonumber
    \\
    &\phantom{\lesssim_\alpha}
    +
    \Big(\int_0^\infty v^2\, \d x\Big)^\frac{4(\alpha-1)}{3\alpha-1} 
    \Big(\int_0^\infty x^{\alpha+2} |\partial_x^2 v|^{\alpha+1}\, \d x\Big)^\frac{\alpha+4}{3\alpha-1} 
    \Big(\int_0^\infty (\partial_x^2 w)^2\, \d x\Big)^\frac{5\alpha-3}{2(3\alpha-1)}
\end{align*}
$K_{2,1,7}$ can be estimated by means of H\"older's inequality and then applying Lemma~\ref{lem:weight xu}, Lemma~\ref{lem:weight xu'}, and Lemma~\ref{lem-sup-wx},  
\begin{align*}
    K_{2,1,7} \qquad
    &\leq \qquad
    |a| \big( \sup_{x>0} x^\frac{1-\alpha}{\alpha+1} |v|\big)^3
    \big(\sup_{x>0} x^\frac{2}{\alpha+1} |\partial_x v|\big)
    \Big(\int_0^2 x^\frac{2(3\alpha-5)}{\alpha+1}\, \d x\Big)^\frac{1}{2} \Big(\int_0^\infty (\partial_x^2 w)^2\, \d x\Big)^\frac{1}{2} \nonumber
    \\ 
    & \stackrel{\mathclap{\eqref{eq:weight xu7},\eqref{eq:weight xu'3},\eqref{est-coeff-a}}}
    {\lesssim}_{\alpha} \qquad
    \Big(\int_0^\infty x^{\alpha+2} |\partial_x^2 v|^{\alpha+1}\, \d x\Big)^\frac{\alpha^2 + 13\alpha - 4}{(3\alpha-1)(\alpha+1)} \Big(\int_0^\infty (\partial_x^2 w)^2\, \d x\Big)^{\frac{5\alpha-3}{2(3\alpha-1)}}.
\end{align*}
%
Altogether, we obtain
\begin{align*}
    I_{2,1,7}
    &\lesssim_\alpha
        \Big(\int_0^\infty v^2\, \d x\Big)^\frac{4(\alpha-1)}{3\alpha-1} 
    \Big(\int_0^\infty x^{\alpha+2} |\partial_x^2 v|^{\alpha+1}\, \d x\Big)^\frac{4}{3\alpha-1} 
    \int_0^\infty (\partial_x^2 w)^2\, \d x
    \\
    &\phantom{\lesssim_\alpha}
    +
    \Big(\int_0^\infty v^2\, \d x\Big)^\frac{4(\alpha-1)}{3\alpha-1} 
    \Big(\int_0^\infty x^{\alpha+2} |\partial_x^2 v|^{\alpha+1}\, \d x\Big)^\frac{\alpha+4}{3\alpha-1} 
    \Big(\int_0^\infty (\partial_x^2 w)^2\, \d x\Big)^\frac{5\alpha-3}{2(3\alpha-1)}
    \\
    &\phantom{\lesssim_\alpha}
    +
    \Big(\int_0^\infty x^{\alpha+2} |\partial_x^2 v|^{\alpha+1}\, \d x\Big)^\frac{\alpha^2 + 13\alpha - 4}{(3\alpha-1)(\alpha+1)} \Big(\int_0^\infty (\partial_x^2 w)^2\, \d x\Big)^{\frac{5\alpha-3}{2(3\alpha-1)}} \\
    &\eqcolon R_1+R_2+R_3.
\end{align*}
Scaling $x \mapsto \lambda x$ for $\lambda > 0$ leads to
\[
I_{2,1,7} \mapsto \lambda^{1-2\alpha} I_{2,1,7}, \qquad R_1 \mapsto \lambda^{1-2\alpha} R_1, \qquad R_2 \mapsto \lambda^{\frac 3 2 - 2 \alpha} R_2, \qquad R_3 \mapsto \lambda^{1 - 2\alpha + \frac{9-7\alpha}{2(\alpha+1)}} R_3.
\]
Division by $\lambda^{1-2\alpha}$ and optimizing in $\lambda$, we infer that the second and third line, now $\lambda^{\frac 1 2} R_2 + \lambda^{\frac{9-7\alpha}{2(\alpha+1)}} R_3$, can be estimated against $R_1^{\frac{7\alpha-9}{8(\alpha-1)}} R_2^{\frac{\alpha+1}{8(\alpha-1)}}$, so that
\begin{align} \label{i217}
    I_{2,1,7} 
    &\lesssim_\alpha
    \Big(\int_0^\infty v^2\, \d x\Big)^\frac{4(\alpha-1)}{3\alpha-1} 
    \Big(\int_0^\infty x^{\alpha+2} |\partial_x^2 v|^{\alpha+1}\, \d x\Big)^\frac{4}{3\alpha-1} 
    \int_0^\infty (\partial_x^2 w)^2\, \d x \nonumber
    \\
    &\phantom{\lesssim_\alpha}
    +
    \Big(\int_0^\infty v^2\, \d x\Big)^{\frac{7\alpha-9}{2(3\alpha-1)}}
    \Big(\int_0^\infty x^{\alpha+2} |\partial_x^2 v|^{\alpha+1}\, \d x\Big)^{\frac{\alpha+5}{3\alpha-1}} 
    \Big(\int_0^\infty (\partial_x^2 w)^2\, \d x\Big)^\frac{5\alpha-3}{2(3\alpha-1)}
\end{align}
\end{subequations}
for the integral $I_{2,1,7}$. 
Gathering \eqref{i21j} in \eqref{est-i21} gives 
\begin{align*}
    I_{2,1}
    &\lesssim_\alpha
    \underbrace{A^\frac{2(\alpha-1)}{3\alpha-1}
    B^\frac{4\alpha}{3\alpha-1}
    C^\frac{\alpha+1}{3\alpha-1}
    }_{I_{2,1,1}, \ I_{2,1,2}}
    +
    \underbrace{
    A^\frac{3(\alpha-1)}{3\alpha-1}
    B^\frac{2(\alpha+1)}{3\alpha-1}
    C^\frac{2\alpha}{3\alpha-1}
    }_{I_{2,1,3},\ I_{2,1,5},\ I_{2,1,6}}
    +
    \underbrace{
    A^\frac{4(\alpha-1)}{3\alpha-1}
    B^\frac{4}{3\alpha-1}
    C
    }_{I_{2,1,4},\ I_{2,1,7}}
    +
    \underbrace{A^{\frac{7\alpha-9}{2(3\alpha-1)}} 
    B^{\frac{\alpha+5}{3\alpha-1}} 
    C^\frac{5\alpha-3}{2(3\alpha-1)}}_{I_{2,1,7}}
\end{align*}
We can estimate the last term on the right-hand side against the first and third terms. Indeed, introducing 
\begin{align*}
    S_1 \coloneqq A^\frac{2(\alpha-1)}{3\alpha-1}
    B^\frac{4\alpha}{3\alpha-1}
    C^\frac{\alpha+1}{3\alpha-1} \qquad 
    \text{and}\qquad S_2 \coloneqq A^\frac{4(\alpha-1)}{3\alpha-1}
    B^\frac{4}{3\alpha-1}
    C,
\end{align*}
we can apply Young's inequality to find (notice $\tfrac{\alpha+1}{4(\alpha-1)} \in (0,1)$ and $1 -\tfrac{\alpha+1}{4(\alpha-1)} = \tfrac{3\alpha -5}{4(\alpha-1)}$)
\begin{align*}
    A^{\frac{7\alpha-9}{2(3\alpha-1)}} 
    B^{\frac{\alpha+5}{3\alpha-1}} 
    C^\frac{5\alpha-3}{2(3\alpha-1)} = S_1^\frac{\alpha+1}{4(\alpha-1)}\, S_2^\frac{3\alpha-5}{4(\alpha-1)} \lesssim_\alpha  S_1 + S_2, 
\end{align*}
and consequently, we obtain the reduced estimate
\begin{align}\label{nonlinear-i21}
    I_{2,1}
    &\lesssim_\alpha
    \underbrace{A^\frac{2(\alpha-1)}{3\alpha-1}
    B^\frac{4\alpha}{3\alpha-1}
    C^\frac{\alpha+1}{3\alpha-1}
    }_{I_{2,1,1}, \ I_{2,1,2}, \ I_{2,1,7}}
    +
    \underbrace{
    A^\frac{3(\alpha-1)}{3\alpha-1}
    B^\frac{2(\alpha+1)}{3\alpha-1}
    C^\frac{2\alpha}{3\alpha-1}
    }_{I_{2,1,3},\ I_{2,1,5},\ I_{2,1,6}}
    +
    \underbrace{
    A^\frac{4(\alpha-1)}{3\alpha-1}
    B^\frac{4}{3\alpha-1}
    C
    }_{I_{2,1,4},\ I_{2,1,7}} \;\, \stackrel{\mathclap{\eqref{eq:B2-AC}}}{\lesssim}_\alpha A^\frac{4(\alpha-1)}{3\alpha-1}
    B^\frac{4}{3\alpha-1}
    C,
\end{align}
where we have used Lemma~\ref{lem:B2-AC} in the last estimate.

\bigskip

\noindent\textbf{Step 2.2: Estimate for $I_{2,2}$. }
For the integral $I_{2,2}$ we apply  Lemma~\ref{lem:con-lem-sup-uxx} which leads to
\begin{align} \nonumber
    I_{2,2}
    \; &\stackrel{\mathclap{\eqref{i2-split}}}{=} \;
    \int_0^\infty \big(x^{\alpha+2} |\partial_x^2 v|^{\alpha+1}\big)^2\, \d x \lesssim_\alpha 
    \Big(\int_0^\infty x^{\alpha +2} \, |\partial_x^2 v|^{\alpha +1} \, \d x\Big)^{\frac{2(2\alpha-1)}{3\alpha-1}}
    \Big( \int_0^\infty (\partial_x^2 w)^2 \, \d x\Big)^{\frac{\alpha+1}{3\alpha-1}} \\
    &= \;
    B^{\frac{2(2\alpha-1)}{3\alpha-1}}
    C^{\frac{\alpha+1}{3\alpha-1}}  \;\, \stackrel{\mathclap{\eqref{eq:B2-AC}}}{\lesssim}_\alpha A^\frac{2(\alpha-1)}{3\alpha-1}
    B^\frac{2}{3\alpha-1}
    C, \label{nonlinear-i22}
\end{align}
where Lemma~\ref{lem:B2-AC} was applied in the last estimate.

\bigskip

\noindent\textbf{Step 2.3: Estimate for $I_{2,3}$. }
Applying integration by parts to $I_{2,3}$ on the right-hand side, we deduce
\begin{equation} \label{eq:est-ux2-wx21}
    I_{2,3}
    = 
    -2  \int_0^\infty v \, ( \partial_x v) (\partial_x w) (\partial_x^2 w) \, \d x 
    -  
    \int_0^\infty v \, ( \partial_x^2 v) (\partial_x w)^2 \, \d x \eqqcolon 
    I_{2,3,1} + I_{2,3,2}. 
\end{equation}
Notice that in the above integration by parts, the boundary term 
$v(\partial_x v)(\partial_x w)^2$ vanishes at both endpoints $x = 0$ and $x = \infty$. Indeed, at $x = 0$, this follows from by Lemma~\ref{lem-valpha}, since $v(\partial_x v)(\partial_x w)^2
\stackrel{\eqref{asymptotic-vw}}{=}O(x^{\frac{\alpha-2}{\alpha}}) \to 0$ as $x \downarrow 0$ for $\alpha > 2$. In order to obtain vanishing $x = \infty$, we multiply the integrand of $I_{2,3}$ by a cut off $\chi_R(x)  \coloneq \chi(x/R)$ with
$\chi \in C^\infty_\mathrm{c}([0,\infty))$,
$\chi|_{[0,1]} = 1$, $\chi|_{[2,\infty)} = 0$, and let $R \to \infty$ after integration by parts. Then we obtain for the remainder term with help of the Cauchy-Schwarz inequality, Lemma~\ref{lem:weight xu'}, and Lemma~\ref{lem-sup-wx}, for $\frac{2}{\alpha+1} \le \beta < \frac{3(\alpha+2)}{5\alpha+3}$,
\begin{align*}
& \int_0^\infty |\chi_R'| |v| |\partial_x v| (\partial_x w)^2 \, \d x \\
& \quad \lesssim \qquad R^{-1} \int_R^{2R} |v| |\partial_x v| (\partial_x w)^2 \, \d x \\
& \quad \le \qquad R^{-\beta} \Big(\int_R^{2R} |v| \, \d x\Big) \big(\sup_{R \le x \le 2R} x^\beta |\partial_x v|\big) \big(\sup_{R \le x \le 2R} x^{-\frac 1 2} |\partial_x w -a| + |a| R^{-\frac 1 2}\big)^2 \\
& \quad \stackrel{\mathclap{\eqref{eq:weight xu'}, \eqref{sup-wx}}}{\lesssim}_{\alpha,\beta} \quad R^{\frac 1 2-\beta} \Big(\int_0^\infty v^2 \, \d x\Big)^{\frac{2(\alpha +1)\beta + 3\alpha-5}{2(3\alpha -1)}} \Big(\int_0^\infty x^{\alpha+2} \, |\partial_x^2 v|^{\alpha+1} \, \d x\Big)^{\frac{3-2\beta}{3\alpha -1}} \Big(\int_0^\infty (\partial_x^2 w)^2 \, \d x + a^2 R^{-1}\Big) \\
& \quad \to \qquad 0 \quad \text{as } R \to \infty,
\end{align*}
on choosing $\beta > \frac 1 2$ (a compatible choice). Notice that finiteness of $|a|$ is ensured by \eqref{est-coeff-a}.

\bigskip

By H\"older's inequality, Lemma~\ref{lem:weight xu}, and finally Young's inequality, we can estimate $I_{2,3,1}$ as 
\begin{align} \label{eq:estimate_I_3}
    I_{2,3,1} 
    &\lesssim 
    \, \sup_{x\geq 0} |v| \Big(\int_0^\infty (\partial_x^2 w)^2 \, \d x\Big)^{\frac{1}{2}} \Big(\int_0^\infty (\partial_x v)^2\,(\partial_x w)^2 \, \d x \Big)^{\frac{1}{2}}
    \nonumber\\&\stackrel{\mathclap{\eqref{lem:weight xu}}}
    {\lesssim}_{\alpha}
    \Big(\int_0^\infty v^2\, \d x\Big)^\frac{\alpha-1}{3\alpha-1} \Big(\int_0^\infty x^{\alpha+2} |\partial_x^2 v|^{\alpha+1}\, \d x\Big)^\frac{1}{3\alpha-1}
    \Big(\int_0^\infty (\partial_x^2 w)^2 \, \d x\Big)^{\frac{1}{2}} \Big(\int_0^\infty (\partial_x v)^2\,(\partial_x w)^2 \, \d x \Big)^{\frac{1}{2}}
    \nonumber \\
    &\lesssim_\alpha \varepsilon^{-1}
    \Big(\int_0^\infty v^2\, \d x\Big)^\frac{2(\alpha-1)}{3\alpha-1} \Big(\int_0^\infty x^{\alpha+2} |\partial_x^2 v|^{\alpha+1}\, \d x\Big)^\frac{2}{3\alpha-1}
    \int_0^\infty (\partial_x^2 w)^2 \, \d x
    + \varepsilon I_{2,3}, 
\end{align}
for any $\varepsilon>0$. The last integral will later on be absorbed in the left-hand side of \eqref{eq:est-ux2-wx21}.
For $I_{2,3,2}$, it follows from integration by parts
\begin{align}
\label{eq:u-uxx-wxx1} 
    I_{2,3,2}
    &=  \int_0^\infty v \, ( \partial_x^2 v) w\,  (\partial_x^2 w) \, \d x  +  \int_0^\infty  ( \partial_x v) ( \partial_x^2 v) w \, (\partial_x w)  \, \d x  + \int_0^\infty v \, ( \partial_x^3 v) w\,  (\partial_x w) \, \d x \nonumber \\
    &\eqcolon J_{2,3,2} + K_{2,3,2}+ L_{2,3,2}. 
\end{align}
Again, in the integration by parts the term $v (\partial_x^2 v) w (\partial_x w)$ vanishes at both boundaries $x = 0$ and $x = \infty$. Indeed, by Lemma~\ref{lem-valpha} it holds $v (\partial_x^2 v) w (\partial_x w) \stackrel{\eqref{asymptotic-vw}}{=} O(x^{\frac{\alpha-2}{\alpha}})$ as $x \downarrow 0$, which ensures vanishing because of $\alpha > 2$. As $x \to \infty$, again take a cut off $\chi_R(x)  \coloneq \chi(x/R)$ with
$\chi \in C^\infty_\mathrm{c}([0,\infty))$,
$\chi|_{[0,1]} = 1$, $\chi|_{[2,\infty)} = 0$, and let $R \to \infty$ after integration by parts. This gives for the remainder term with help of Lemma~\ref{lem:weight xu} and Lemma~\ref{lem-sup-wx}, where $\frac{1-\alpha}{\alpha +1} \le \beta < \frac{\alpha+2}{5\alpha +3}$,
\begin{align*}
&\int_0^\infty |\chi_R'| |v| |\partial_x^2 v| |w| |\partial_x w| \, \d x \\
& \quad \lesssim \qquad R^{-1} \int_R^{2R} |v| |\partial_x^2 v| |w| |\partial_x w| \, \d x \\
& \quad \lesssim \qquad R^{-\frac 1 2 - \beta} \big(\sup_{R \le x \le 2R} x^\beta |v|\big) \Big(\int_0^\infty x^{\alpha+2} |\partial_x^2 v|^{\alpha+1} \, \d x\Big) \big(\sup_{R \le x \le 2R} x^{-\frac 1 2} |\partial_x w - a| + |a| R^{-\frac 1 2}\big) \\
& \quad \stackrel{\mathclap{\eqref{eq:weight xu}, \eqref{sup-wx}}}{\lesssim}_{\alpha,\beta} \quad R^{-\frac 1 2 - \beta} \Big(\int_0^\infty v^2 \, \d x\Big)^{\frac{\alpha-1+(\alpha+1)\beta}{3\alpha-1}} \Big(\int_0^\infty x^{\alpha+2} |\partial_x^2 v|^{\alpha+1} \, \d x\Big)^{\frac{3\alpha-2\beta}{3\alpha-1}} \Big(\Big(\int_0^\infty (\partial_x^2 w)^2 \, \d x\Big)^{\frac 1 2} + |a| R^{-\frac 1 2}\Big) \\
& \quad \to \qquad 0 \quad \text{as } R \to \infty,
\end{align*}
on choosing for instance $\beta = 0$. Notice that finiteness of $|a|$ is ensured by \eqref{est-coeff-a}.

\bigskip

We next estimate the terms $J_{2,3,2}$, $K_{2,3,2}$, $L_{2,3,2}$, one by one. 
For $J_{2,3,2}$, we obtain by H\"older's inequality, Lemma~\ref{lem:weight xu}, and Lemma~\ref{lem:con-lem-sup-uxx},
\begin{align}
   &J_{2,3,2} \nonumber \\
   & \quad \stackrel{\mathclap{\eqref{eq:u-uxx-wxx1}}}{\leq} \; 
    \sup_{x\geq 0} |v| \Big( \int_0^\infty (x^{\alpha+2} |\partial_x^2 v|^{\alpha +1})^2 \, \d x\Big)^{\frac{1}{2}} \Big( \int_0^\infty (\partial_x^2 w)^2 \, \d x\Big)^{\frac{1}{2}} \nonumber
    \\
    & \quad \stackrel{\mathclap{\eqref{eq:weight xu}}}{\lesssim}_{\alpha}
    \Big(\int_0^\infty v^2\, \d x\Big)^\frac{\alpha-1}{3\alpha-1}
    \Big(\int_0^\infty x^{\alpha+2} |\partial_x^2 v|^{\alpha+1}\, \d x\Big)^\frac{1}{3\alpha-1} \Big( \int_0^\infty (x^{\alpha+2} |\partial_x^2 v|^{\alpha +1})^2 \, \d x\Big)^{\frac{1}{2}} \Big( \int_0^\infty (\partial_x^2 w)^2 \, \d x\Big)^{\frac{1}{2}} \nonumber \\
    & \quad \stackrel{\mathclap{\eqref{eq:con-lem-sup-uxx}}}{\lesssim}_{\alpha} \Big(\int_0^\infty v^2\, \d x\Big)^\frac{\alpha-1}{3\alpha-1}
    \Big(\int_0^\infty x^{\alpha +2} \, |\partial_x^2 v|^{\alpha +1} \, \d x\Big)^{\frac{2\alpha}{3\alpha-1}}
    \Big( \int_0^\infty (\partial_x^2 w)^2 \, \d x\Big)^{\frac{2\alpha}{3\alpha-1}}. \label{eq:est-u-uxx-w-wxx}
\end{align}
\bigskip

For $K_{2,3,2}$, we use H\"older's inequality, Lemma~\ref{lem:con-lem-sup-uxx}, and Young's inequality to find that 
\begin{align}
   K_{2,3,2}
    \; &\stackrel{\mathclap{\eqref{eq:u-uxx-wxx1}}}{\leq} \;  \Big(\int_0^\infty (\partial_x v)^2\, (\partial_x w)^2 \d x\Big)^\frac{1}{2}
    \Big(\int_0^\infty \big(x^{\alpha+2}\, |\partial_x^2 v|^{\alpha+1}\big)^2 \d x\Big)^\frac{1}{2} \nonumber 
    \\
    &\stackrel{\mathclap{\eqref{eq:con-lem-sup-uxx}}}{\lesssim}_{\alpha} \Big(\int_0^\infty (\partial_x v)^2\,(\partial_x w)^2 \, \d x \Big)^{\frac{1}{2}}
    \Big(  \int_0^\infty x^{\alpha +2} \, |\partial_x^2 v|^{\alpha +1} \, \d x\Big)^{\frac{2\alpha-1}{3\alpha-1}}
    \Big( \int_0^\infty (\partial_x^2 w)^2 \, \d x\Big)^{\frac{\alpha+1}{2(3\alpha-1)}} \nonumber 
    \\
    &\stackrel{\mathclap{\eqref{i2-split}}}{\lesssim}_\alpha
    \eps^{-1}\Big(\int_0^\infty x^{\alpha +2} \, |\partial_x^2 v|^{\alpha +1} \, \d x\Big)^{\frac{2(2\alpha-1)}{3\alpha-1}}
    \Big( \int_0^\infty (\partial_x^2 w)^2 \, \d x\Big)^{\frac{\alpha+1}{3\alpha-1}}+  \eps I_{2,3}, \label{eq:est-ux-uxx-w-wx}
\end{align}
for any $\eps > 0$. The last integral will later on be absorbed in the left-hand side of \eqref{eq:est-ux2-wx21} provided $\eps > 0$ is chosen sufficiently small.

\bigskip

For $L_{2,3,2}$, recalling the definition of $w$ and integrating by parts, we derive 
\begin{align}\label{eq:u-uxxx-w-wx1}
    L_{2,3,2}
    \; &\stackrel{\mathclap{\eqref{eq:u-uxx-wxx1}}}{=} \; \int_0^\infty x^{\alpha+2} \, v\,  \partial_x\Big(\frac{|\partial_x^2 v|^{\alpha +1}}{\alpha+1}\Big) \, \partial_x w \, \d x \nonumber 
    \\
    &\lesssim_{\alpha} \int_0^\infty x^{\alpha+2}  \, |\partial_x v|  |\partial_x^2 v|^{\alpha+1}|\partial_x w| \, \d x 
    + 
    \int_0^\infty x^{\alpha+1}  \,|v|   |\partial_x^2 v|^{\alpha+1}|\partial_x w| \, \d x \nonumber \\
    & \qquad + \int_0^\infty x^{\alpha+2}  \,|v|   |\partial_x^2 v|^{\alpha+1}|\partial_x^2 w| \, \d x 
    \nonumber \\
    &\eqcolon L_{2,3,2}^{(1)} + L_{2,3,2}^{(2)}+L_{2,3,2}^{(3)}, 
\end{align}
where, again, the boundary term $x^{\alpha+2} v |\partial_x^2 v|^{\alpha+1} (\partial_x w)$ arising from the integration by parts vanishes at both boundaries $x=0$ and $x=\infty$. Indeed, by Lemma~\ref{lem-valpha} it follows that 
$x^{\alpha+2} v |\partial_x^2 v|^{\alpha+1} (\partial_x w) \stackrel{\eqref{asymptotic-vw}}{=} O(x^\frac{\alpha-2}{\alpha})$ as $x \downarrow 0$, which vanishes since $\alpha>2$. To handle the limit as $x \to \infty$, we again take a cut off $\chi_R(x) \coloneq \chi(x/R)$ with $\chi \in C^\infty_\mathrm{c}([0,\infty))$, $\chi|_{[0,1]} = 1$, and $\chi|_{[2,\infty)} = 0$, perform integration by parts, and let $R \to \infty$. Applying H\"older's inequality, Lemma~\ref{lem:weight xu}, and Lemma~\ref{lem-sup-wx}, to the resulting remainder term for 
$\frac{1-\alpha}{\alpha +1} \leq \beta < \frac{\alpha+2}{5\alpha +3}$ then yields
\begin{align*}
&\int_0^\infty |\chi_R'| \, x^{\alpha+2} \, |v| |\partial_x^2 v|^{\alpha+1}|\partial_x w| \, \d x \\
& \quad \lesssim_{\phantom{\alpha,\beta}} \; R^{-1} \int_R^{2R} x^{\alpha+2} |v| |\partial_x^2 v|^{\alpha+1}|\partial_x w|  \, \d x \\
& \quad \lesssim_{\phantom{\alpha,\beta}} \; R^{-\frac 1 2 - \beta} \big(\sup_{R \le x \le 2R} x^\beta |v|\big) \Big(\int_0^\infty x^{\alpha+2} |\partial_x^2 v|^{\alpha+1} \, \d x\Big) \big(\sup_{R \le x \le 2R} x^{-\frac 1 2} |\partial_x w -a| + |a| R^{-\frac 1 2}\big) \\
& \quad \stackrel{\mathclap{\eqref{eq:weight xu}, \eqref{sup-wx}}}{\lesssim}_{\alpha,\beta} \; R^{-\frac 1 2 - \beta} \Big(\int_0^\infty v^2 \, \d x\Big)^{\frac{\alpha-1+(\alpha+1)\beta}{3\alpha-1}} \Big(\int_0^\infty x^{\alpha+2} |\partial_x^2 v|^{\alpha+1} \, \d x\Big)^{\frac{3\alpha-2\beta}{3\alpha-1}} \Big(\Big(\int_0^\infty (\partial_x^2 w)^2 \, \d x\Big)^{\frac 1 2} + |a| R^{-\frac 1 2}\Big) \\
& \quad \to_{\phantom{\alpha,\beta}} \; 0 \quad \text{as } R \to \infty,
\end{align*}
upon choosing, for instance, $\beta = 0$, noting again that \eqref{est-coeff-a} guarantees the finiteness of $|a|$. Hence, the integration by parts is justified. We now turn to the integrals $L_{2,3,2}^{(j)}$ for $j=1,2,3$, estimating each of them separately.

\bigskip

For $L_{2,3,2}^{(1)}$, we apply H\"older's inequality, Lemma~\ref{lem:con-lem-sup-uxx}, and then Young's inequality, leading to 
\begin{align} \label{eq:est-xux-uxx-wx}
    L_{2,3,2}^{(1)} 
    \; &\stackrel{\mathclap{\eqref{eq:u-uxxx-w-wx1}}}{\lesssim}_{\alpha} \Big(\int_0^\infty (\partial_x v)^2\,(\partial_x w)^2 \, \d x \Big)^{\frac{1}{2}} 
    \Big(\int_0^\infty \big(x^{\alpha+2}\, |\partial_x^2 v|^{\alpha+1}\big)^2 \d x\Big)^\frac{1}{2} \nonumber 
    \\
    &\stackrel{\mathclap{\eqref{eq:con-lem-sup-uxx}}}{\le}_{\phantom{\alpha}} \Big(\int_0^\infty (\partial_x v)^2\,(\partial_x w)^2 \, \d x \Big)^{\frac{1}{2}} 
    \Big(\int_0^\infty x^{\alpha +2} \, |\partial_x^2 v|^{\alpha +1} \, \d x\Big)^{\frac{2\alpha-1}{3\alpha-1}}
    \Big( \int_0^\infty (\partial_x^2 w)^2 \, \d x\Big)^{\frac{\alpha+1}{2(3\alpha-1)}} \nonumber \\
    &\stackrel{\mathclap{\eqref{i2-split}}}{\leq}_{\phantom{\alpha}}
    \eps^{-1} \Big(\int_0^\infty x^{\alpha +2} \, |\partial_x^2 v|^{\alpha +1} \, \d x\Big)^{\frac{2(2\alpha-1)}{3\alpha-1}}
    \Big(\int_0^\infty (\partial_x^2 w)^2 \, \d x\Big)^{\frac{\alpha+1}{3\alpha-1}} + \eps I_{2,3}, 
\end{align}
for $\eps > 0$ sufficiently small, such that the last integral on the right-hand side can later on be absorbed in the left-hand side of \eqref{eq:est-ux2-wx21}.

\bigskip

In order to estimate $L_{2,3,2}^{(2)}$,  we split the integral
\begin{equation*}
   L_{2,3,2}^{(2)}
    \; \stackrel{\mathclap{\eqref{eq:u-uxxx-w-wx1}}}{\leq} \;
    \int_0^\infty x^{\alpha+1} |v|\, |\partial_x^2 v|^{\alpha+1} |\partial_x w - \eta a|\, \d x
    +
    |a|\, \int_0^2 x^{\alpha+1} |v|\, |\partial_x^2 v|^{\alpha+1}\, \d x
    \eqqcolon 
    L_{2,3,2}^{(2,1)}+ L_{2,3,2}^{(2,2)}, 
\end{equation*}
where $\eta \in C^\infty([0,\infty))$ is again a cut off function satisfying $0 \le \eta \le 1$, $\eta|_{[0,1]} = 1$, and $\eta|_{[2,\infty)} = 0$, and we treat the integrals $ L_{2,3,2}^{(2,1)}$ and $ L_{2,3,2}^{(2,2)}$ separately. 

\bigskip

For $ L_{2,3,2}^{(2,1)}$, applying H\"older's inequality, followed by Lemma \ref{lem:weight xu} and Lemma~\ref{lem:con-lem-sup-uxx} alongside Hardy's inequality, and finally invoking Lemma~\ref{lem-sup-wx}, we deduce
\begin{align} \label{eq:L232_21}
 L_{2,3,2}^{(2,1)} \quad
    &\leq_{\phantom{\alpha}} \quad
    \big(\sup_{x\geq 0} |v|\big)
    \Big( \int_0^\infty \big(x^{\alpha +2} \, |\partial_x^2 v|^{\alpha +1}\big)^2 \, \d x\Big)^\frac{1}{2}
    \Big(\int_0^\infty x^{-2}\,|\partial_x w - \eta a|^2 \, \d x \Big)^{\frac{1}{2}} \nonumber
    \\
    &\stackrel{\mathclap{\eqref{eq:weight xu}, \eqref{eq:con-lem-sup-uxx}}}{\lesssim}_{\alpha}  \quad 
    \Big(\int_0^\infty v^2\, \d x\Big)^\frac{\alpha-1}{3\alpha-1}
    \Big(\int_0^\infty x^{\alpha+2} |\partial_x^2 v|^{\alpha+1}\, \d x\Big)^\frac{2\alpha}{3\alpha-1}
    \Big(  \int_0^\infty (\partial_x^2 w)^2 \, \d x\Big)^\frac{\alpha+1}{2(3\alpha-1)} \nonumber
    \\
    &\phantom{\stackrel{\mathclap{\eqref{eq:weight xu}, \eqref{eq:con-lem-sup-uxx}}}{\lesssim}_{\alpha}} \quad 
    \times
    \bigg(\Big(\int_0^\infty (\partial_x^2 w)^2\, \d x\Big)^\frac{1}{2} + |a| \bigg) \nonumber 
    \\
    &\stackrel{\mathclap{\eqref{est-coeff-a}}}{\lesssim}_{\alpha} \quad 
    \Big(\int_0^\infty v^2\, \d x\Big)^\frac{\alpha-1}{3\alpha-1} \Big(\int_0^\infty x^{\alpha+2} |\partial_x^2 v|^{\alpha+1}\, \d x\Big)^\frac{2\alpha}{3\alpha-1} 
    \Big(\int_0^\infty  (\partial_x^2 w)^2\, \d x\Big)^\frac{2\alpha}{3\alpha-1}\nonumber
    \\
    &\phantom{\lesssim_\alpha} \quad 
    + 
    \Big(\int_0^\infty v^2\, \d x\Big)^\frac{\alpha-1}{3\alpha-1} \Big(\int_0^\infty x^{\alpha+2} |\partial_x^2 v|^{\alpha+1}\, \d x\Big)^\frac{3\alpha}{3\alpha-1} 
    \Big(\int_0^\infty  (\partial_x^2 w)^2\, \d x\Big)^\frac{1}{2}.
\end{align}
\bigskip

To estimate $ L_{2,3,2}^{(2,2)}$, we rewrite $v = (v - cx^{\frac{\alpha-1}{\alpha}}) + cx^{\frac{\alpha-1}{\alpha}}$, where $c$ is from Lemma \ref{lem-valpha}, and then apply H\"older's inequality. Treating only the second resulting integral on the right-hand side via Hardy's inequality and Lemma~\ref{lem-sup-uxx} with $\beta = \tfrac{2\alpha+1}{2\alpha}$, we find
\begin{align*}
   L_{2,3,2}^{(2,2)}
    \; &\leq_{\phantom{\alpha}}
    |a| \int_0^2 x^{\alpha+1} |\partial_x^2 v|^{\alpha+1} |v- cx^{\frac{\alpha-1}{\alpha}}|\, \d x +    |a||c| \int_0^2 x^{-\frac{1}{\alpha}+ (\alpha+2)}  |\partial_x^2 v|^{\alpha+1} \, \d x
    \\
    &\leq_{\phantom{\alpha}}
    |a| \Big( \int_0^\infty  \left( x^{\alpha+2} |\partial_x^2 v|^{\alpha+1} \right)^2 \d x \Big)^\frac{1}{2} \bigg( \Big( \int_0^2 x^{-2} |v-cx^\frac{\alpha-1}{\alpha}|^2 \, \d x \Big)^\frac{1}{2} +
    |c| \Big( \int_0^2 x^{-\frac{2}{\alpha}} \, \d x \Big)^\frac{1}{2} \bigg)\\
    &\stackrel{\mathclap{\eqref{sup-uxx-alt}}}{\lesssim}_{\alpha} |a| \Big( \int_0^\infty  \left( x^{\alpha+2} |\partial_x^2 v|^{\alpha+1} \right)^2 \d x \Big)^\frac{1}{2} \bigg(  \Big(\int_0^\infty  (\partial_x^2 w)^2\, \d x\Big)^\frac{1}{2\alpha} +
    |c| \bigg), 
\end{align*}
where we have used $\alpha > 2$. We then invoke Lemma~\ref{lem-sup-wx} and Lemma~\ref{lem-sup-uxx} to control the coefficients $|a|$ and $|c|$, respectively, as well as Lemma~\ref{lem:con-lem-sup-uxx} for the first integral on the right-hand side, which yields the bound
\begin{align}\label{eq:L232_22}
    L_{2,3,2}^{(2,2)} \quad\qquad
    &\stackrel{\mathclap{\eqref{est-coeff-a}, \eqref{est-coeff-c}, \eqref{eq:con-lem-sup-uxx}}}{\lesssim}_{\alpha} \;\qquad \Big( \int_0^\infty  x^{\alpha+2} |\partial_x^2 v|^{\alpha+1} \d x \Big) \Big(\int_0^\infty  (\partial_x^2 w)^2\, \d x\Big)^\frac{\alpha+1}{2\alpha} \nonumber
    \\
    &\phantom{\stackrel{\mathclap{\eqref{est-coeff-a}, \eqref{est-coeff-c}, \eqref{eq:con-lem-sup-uxx}}}{\lesssim}_{\alpha}} \qquad\;
    + \Big( \int_0^\infty  x^{\alpha+2} |\partial_x^2 v|^{\alpha+1} \d x \Big)^\frac{3\alpha}{3\alpha-1} \Big(\int_0^\infty  (\partial_x^2 w)^2\, \d x\Big)^\frac{3\alpha^2+\alpha -2}{2\alpha(3\alpha-1)}. 
\end{align}
Combining \eqref{eq:L232_21} and \eqref{eq:L232_22}, we arrive at
\begin{align*}
L_{2,3,2}^{(2)}
\, \stackrel{\mathclap{\eqref{eq:def_ABC_2}}}{\lesssim}_{\alpha} A^\frac{\alpha-1}{3\alpha-1} B^\frac{2\alpha}{3\alpha-1} C^\frac{2\alpha}{3\alpha-1} + A^\frac{\alpha-1}{3\alpha-1} B^\frac{3\alpha}{3\alpha-1} 
    C^\frac{1}{2} + B C^\frac{\alpha+1}{2\alpha} + B^\frac{3\alpha}{3\alpha-1} C^\frac{3\alpha^2+\alpha -2}{2\alpha(3\alpha-1)} \eqcolon R_1 + R_2 + R_3 + R_4.
\end{align*}
By scaling $x \mapsto \lambda x$, where $\lambda > 0$, we obtain
\[
L^{(2)}_{2,3,2} \mapsto \lambda^{1-2\alpha} L^{(2)}_{2,3,2}, \quad R_1 \mapsto \lambda^{1-2\alpha} R_1, \quad R_2 \mapsto \lambda^{\frac 3 2 - 2 \alpha} R_2, \quad R_3 \mapsto \lambda^{\frac{\alpha+1}{2\alpha} - 2 \alpha} R_3, \quad R_4 \mapsto \lambda^{\frac{\alpha+2}{2\alpha} - 2\alpha} R_4,
\]
so that
\[
L^{(2)}_{2,3,2} \lesssim_\alpha R_1 + \lambda^{\frac 1 2} R_2 + \lambda^{\frac{1-\alpha}{2\alpha}} R_3 + \lambda^{\frac{2-\alpha}{2\alpha}} R_4.
\]
Without loss of generality suppose $B \ne 0$. Take $\lambda \coloneq B^{-\frac{2\alpha}{3\alpha-1}} C^{\frac{\alpha+1}{3\alpha-1}}$, then
\[
\lambda^{\frac 1 2} R_2 = R_1, \quad \lambda^{\frac{1-\alpha}{2\alpha}} R_3 = B^{\frac{2(2\alpha-1)}{3\alpha-1}} C^{\frac{\alpha+1}{3\alpha-1}}, \quad \lambda^{\frac{2-\alpha}{2\alpha}} R_4 = B^{\frac{2(2\alpha-1)}{3\alpha-1}} C^\frac{\alpha+1}{3\alpha-1},
\]
so that with help of Lemma~\ref{lem:B2-AC},
\begin{equation}\label{eq:estimate_K_2}
L^{(2)}_{2,3,2} \lesssim_\alpha A^\frac{\alpha-1}{3\alpha-1} B^\frac{2\alpha}{3\alpha-1} C^\frac{2\alpha}{3\alpha-1} + B^{\frac{2(2\alpha-1)}{3\alpha-1}} C^\frac{\alpha+1}{3\alpha-1} \; \stackrel{\mathclap{\eqref{eq:B2-AC}}}{\lesssim}_\alpha A^\frac{\alpha-1}{3\alpha-1} B^\frac{2\alpha}{3\alpha-1} C^\frac{2\alpha}{3\alpha-1}.
\end{equation}

\medskip

Finally, $ L_{2,3,2}^{(3)}$ can be estimated by the Cauchy-Schwarz inequality, Lemma~\ref{lem:weight xu}, and Lemma~\ref{lem:con-lem-sup-uxx},
\begin{align} \label{eq:estimate_K_3}
  L_{2,3,2}^{(3)} \;
    \quad\;&\stackrel{\mathclap{\eqref{eq:u-uxxx-w-wx1}}}{\lesssim}_{\alpha}\quad \big(\sup_{x \ge 0} |v|\big) \Big(\int_0^\infty (x^{\alpha+2} |\partial_x^2 v|^{\alpha+1})^2 \, \d x\Big)^{\frac 1 2} \Big(\int_0^\infty (\partial_x^2 w)^2 \, \d x\Big)^{\frac 1 2} \nonumber \\
    &\stackrel{\mathclap{\eqref{eq:weight xu},\eqref{eq:con-lem-sup-uxx}}}{\lesssim}_\alpha \quad 
    \Big(\int_0^\infty v^2\, \d x\Big)^\frac{\alpha-1}{3\alpha-1}
    \Big(  \int_0^\infty x^{\alpha +2} \, |\partial_x^2 v|^{\alpha +1} \, \d x\Big)^{\frac{2\alpha}{3\alpha-1}}
    \Big( \int_0^\infty (\partial_x^2 w)^2 \, \d x\Big)^{\frac{2\alpha}{3\alpha-1}}.
\end{align}
Inserting \eqref{eq:estimate_I_3}, \eqref{eq:u-uxx-wxx1}, \eqref{eq:est-u-uxx-w-wxx}, \eqref{eq:est-ux-uxx-w-wx}, \eqref{eq:u-uxxx-w-wx1}, \eqref{eq:est-xux-uxx-wx}, \eqref{eq:estimate_K_2}, and \eqref{eq:estimate_K_3} into \eqref{eq:est-ux2-wx21} leads to the estimate
\begin{equation}\label{nonlinear-i23}
    I_{2,3}
    \, \stackrel{\mathclap{\eqref{eq:def_ABC_2}}}{\lesssim}_{\alpha}
    A^\frac{2(\alpha-1)}{3\alpha-1} B^\frac{2}{3\alpha-1}
    C
    +
    A^\frac{\alpha-1}{3\alpha-1}
    B^{\frac{2\alpha}{3\alpha-1}}
    C^{\frac{2\alpha}{3\alpha-1}} 
    +
    B^{\frac{2(2\alpha-1)}{3\alpha-1}}
    C^{\frac{\alpha+1}{3\alpha-1}} \;\, \stackrel{\mathclap{\eqref{eq:B2-AC}}}{\lesssim}_\alpha A^\frac{2(\alpha-1)}{3\alpha-1} B^\frac{2}{3\alpha-1}
    C,
\end{equation}
where we have used Lemma~\ref{lem:B2-AC} in the last estimate.

\bigskip

\noindent\textbf{Step 2.4: Estimate for $I_{2,4}$. }
We estimate the integral $I_{2,4}$ using Lemma~\ref{lem:weight xu} to obtain 
\begin{align} \label{eq:estimate_I_24}
    I_{2,4}
    &\stackrel{\mathclap{\eqref{i2-split}}}{=}
    \big(\sup_{x\geq 0} |v|\big)^2 \int_0^\infty (\partial_x^2 w)^2\, \d x
    \stackrel{\mathclap{\eqref{eq:weight xu}}}{\lesssim_{\alpha}}
    \Big(\int_0^\infty v^2\, \d x\Big)^\frac{2(\alpha-1)}{3\alpha-1}
    \Big(\int_0^\infty x^{\alpha+2}\, |\partial_x^2 v|^{\alpha+1} \d x\Big)^\frac{2}{3\alpha-1}
    \int_0^\infty (\partial_x^2 w)^2\, \d x
    \nonumber\\
    &\stackrel{\mathclap{\eqref{eq:def_ABC_2}}}{=} \, 
    A^\frac{2(\alpha-1)}{3\alpha-1}
    B^\frac{2}{3\alpha-1}
    C.
\end{align}
\bigskip

\noindent\textbf{Step 3. }
The combination of \eqref{nonlinear-i1}, \eqref{i2-split}, \eqref{nonlinear-i21}, \eqref{nonlinear-i22}, \eqref{nonlinear-i23}, and \eqref{eq:estimate_I_24} finishes the proof of \eqref{est-nv}.
\end{proof}
\begin{proposition} \label{prop:N(v)u}
Suppose $\alpha > 2$, let $u \in U_\alpha$ and $v \in V_\alpha$ with
$\|u\|_{U_\alpha} \ll_\alpha 1$ and $\|v\|_{U_\alpha} \ll_\alpha 1$. Then we have the inequality
\begin{align}\nonumber
    \int_0^\infty N(v)\, u\, \d x
    & \lesssim_\alpha
    B C^\frac{2(\alpha-1)}{3\alpha-1} D^\frac{3\alpha^2+1}{3\alpha-1} 
    +
    A^\frac{\alpha-1}{3\alpha-1}
    B^\frac{2\alpha}{3\alpha-1}
    C^\frac{\alpha-1}{3\alpha-1} 
    D^\frac{\alpha(3\alpha+1)}{3\alpha-1}
    \\
    &\qquad
    + 
    A^\frac{2(\alpha-1)}{3\alpha-1} B^\frac{\alpha+1}{3\alpha-1}
    C^\frac{2(\alpha-1)}{3\alpha-1} 
    D^\frac{3\alpha(\alpha+1)}{3\alpha-1}
    + 
    A^\frac{2(\alpha-1)}{3\alpha-1} B^\frac{\alpha+1}{3\alpha-1} 
    D^{\alpha+1},\label{est-N(v)u}
\end{align}
where
\begin{align*}
    A &\coloneq \Big(\int_0^\infty u^2\, \d x\Big)^\frac{1}{2},
    \quad
    B \coloneq \Big(\int_0^\infty x^{\alpha+2} |\partial_x^2 u|^{\alpha+1}\, \d x\Big)^\frac{1}{\alpha+1},
    \\
    C & \coloneq \Big(\int_0^\infty v^2\, \d x\Big)^\frac{1}{2},
    \quad
    D \coloneq \Big(\int_0^\infty x^{\alpha+2} |\partial_x^2 v|^{\alpha+1}\, \d x\Big)^\frac{1}{\alpha+1}.
\end{align*}
\end{proposition}

Note that in estimate~\eqref{est-N(v)u} of Proposition~\ref{prop:N(v)u} in each term of the sum on the right-hand side the exponents of $A$ and $B$ sum up to one, while the exponents of $C$ and $D$ sum up to at least $\alpha+1$. This will become important in the time-discretization scheme in \S\ref{sec:time-discretization}.

\begin{proof}
By approximation according to Lemma~\ref{lem-approx-ualpha}, we may assume $u \in C^\infty_\mathrm{c}((0,\infty))$, thus justifying all integrations by parts within this proof. By \eqref{est-nv} of Proposition~\ref{prop-est-nv} we have $N(v) \in L^2(0,\infty)$, so that the left-hand side of \eqref{est-N(v)u} is well-defined. Using the definition \eqref{def-nu} of $N(\cdot)$ and applying integration by parts twice on each term of $N(v)u$, we obtain
\begin{align}\nonumber
    \int_0^{\infty} N(v) u\, \d x 
    &= 
    \int_0^{\infty} u\, \bigl(1-(1+v)^{\frac{3}{2}}\bigr) \, \partial^2_x\bigl(x^{\alpha+2}\, g_{\alpha}(\partial_x^2 v)\bigr) \, \d x 
    \\
    & \quad 
    + \int_0^{\infty} u\,  (1+v)^{\frac{3}{2}} \, \partial^2_x\bigl(
    \bigl(1-(1+v)^{\frac{\alpha}{2}}\bigr) x^{\alpha+2}\, g_{\alpha}(\partial_x^2 v)\bigr) \, \d x 
    \nonumber\\
    &= 
    \int_0^{\infty} x^{\alpha+2}\, g_{\alpha}(\partial_x^2 v)\, 
    \partial^2_x\bigl(u\, \bigl(1-(1+v)^{\frac{3}{2}}\bigr)\bigr) \, \d x 
    \nonumber\\
    &\quad 
    + \int_0^{\infty} \bigl(1-(1+v)^{\frac{\alpha}{2}}\bigr) \, x^{\alpha+2}\, g_{\alpha}(\partial_x^2 v)\, \partial^2_x\bigl(u\,  (1+v)^{\frac{3}{2}}\bigr) \,\d x
    \nonumber\\
    &\eqcolon
    I_1 + I_2. \label{eq:estimate1RHS}
\end{align}
We start by estimating $I_1$. Due to estimate~\eqref{eq:weight xu} of Lemma~\ref{lem:weight xu} and $\|v\|_{U_\alpha} \ll_\alpha 1$, we can assume $\sup_{x\geq 0} |v| \le \frac 1 2$. Thus, computing $\partial^2_x\bigl(u\, \bigl(1-(1+v)^{\frac{3}{2}}\bigr)\bigr)$ and using the definition of $g_\alpha$, we may estimate 
\begin{align}\nonumber
    I_1
    &=
    \int_0^\infty x^{\alpha+2} g_\alpha(\partial_x^2 v)  \bigl(1 - (1 + v)^\frac{3}{2}\bigr)\, \partial_x^2u\, \d x 
    - 
    3 \int_0^\infty x^{\alpha+2} g_\alpha(\partial_x^2 v) (1 + v)^\frac{1}{2} (\partial_x v)\, (\partial_x u)\, \d x
    \\
    & \phantom{=}
    - \frac{3}{4} \int_0^\infty x^{\alpha+2} g_\alpha(\partial_x^2 v) (\partial_x v)^2\, (1 + v)^{-\frac{1}{2}}\, u\,  \d x
    -
    \frac{3}{2} \int_0^\infty x^{\alpha+2} |\partial_x^2 v|^{\alpha+1} (1 + v)^\frac{1}{2} u\, \d x
    \nonumber\\
    & \lesssim
    \sup_{x\geq 0} |v|
    \int_0^\infty x^{\alpha+2} 
    |\partial_x^2 v|^\alpha\, |\partial_x^2 u|\, \d x 
    + 
    \int_0^\infty x^{\alpha+2} 
    |\partial_x^2 v|^\alpha\, |\partial_x v|\, |\partial_x u|\, \d x 
    \nonumber\\
    & \phantom{=} 
    +
    \int_0^\infty x^{\alpha+2} |\partial_x^2 v|^\alpha\, (\partial_x v)^2\, |u|\, \d x
    + 
    \int_0^\infty x^{\alpha+2} |\partial_x^2 v|^{\alpha+1}\, |u|\, \d x
    \nonumber\\
    &=
    J_1 + J_2 + J_3 + J_4. \label{eq:N(v)u_J_i}
\end{align}
Applying H\"older's inequality to $J_1$ and estimating $\sup_{x\geq 0} |v|$ by means of Lemma~\ref{lem:weight xu} with $\beta=0$, we find
\begin{align}\nonumber
    J_1 
    &\leq_{\phantom{\alpha}} 
    \Big(\int_0^\infty x^{\alpha+2} |\partial_x^2 u|^{\alpha+1}\, \d x\Big)^\frac{1}{\alpha+1} \sup_{x \geq 0} |v|\, \Big(\int_0^\infty x^{\alpha+2} |\partial_x^2 v|^{\alpha+1}\, \d x\Big)^\frac{\alpha}{\alpha+1}
    \\
    &\stackrel{\mathclap{\eqref{eq:weight xu}}}{\lesssim}_\alpha 
    \Big(\int_0^\infty x^{\alpha+2} |\partial_x^2 u|^{\alpha+1}\, \d x\Big)^\frac{1}{\alpha+1}
    \Big(\int_0^\infty v^2\, \d x\Big)^\frac{\alpha-1}{3\alpha-1}
    \Big(\int_0^\infty x^{\alpha+2} |\partial_x^2 v|^{\alpha+1}\, \d x\Big)^\frac{3\alpha^2+1}{(3\alpha-1)(\alpha+1)}. \label{eq:N(v)u_J1}
\end{align}
In order to estimate the integral $J_2$, we first apply H\"older's inequality which gives
\begin{equation*}
    J_2
    \leq
    \Big(\int_0^\infty x^{\alpha+2} |\partial_x u|^{2(\alpha+1)}\, \d x\Big)^\frac{1}{2(\alpha+1)}
    \Big(\int_0^\infty x^{\alpha+2} |\partial_x^2 v|^{\alpha+1}\, \d x\Big)^\frac{\alpha}{\alpha+1} 
    \Big(\int_0^\infty x^{\alpha+2} |\partial_x v|^{2(\alpha+1)}\, \d x\Big)^\frac{1}{2(\alpha+1)}.
\end{equation*}
The second factor is benign. To estimate the other two factors, we apply Lemma~\ref{lem:weight xu lp'} with $p = 2(\alpha+1) > \frac{4(\alpha+1)}{\alpha+3}$ and $\nu = \frac{\alpha+2}{2(\alpha+1)}$, where we note that
\begin{align*}
\nu &\ge \frac{2}{\alpha+1} + \frac{2(\alpha-2)}{(\alpha+3)p} = \frac{3\alpha+4}{(\alpha+1)(\alpha+3)} \quad \Longleftarrow \quad \alpha \ge 2, \\
\nu &< \frac{\alpha+2}{5\alpha+3} \big(3-\tfrac 2 p\big) = \frac{(\alpha+2)(3\alpha+2)}{(\alpha+1)(5\alpha+3)} \quad \Longleftarrow \quad \alpha > -1.
\end{align*}
This gives
\begin{equation*}
    \int_0^\infty x^{\alpha+2} |\partial_x u|^{2(\alpha+1)}\, \d x
    \; \stackrel{\mathclap{\eqref{eq:weight xu lp'}}}{\lesssim}_\alpha \;
    \Big(\int_0^\infty u^2\, \d x\Big)^\frac{(\alpha-1)(\alpha+1)}{3\alpha-1}
    \Big(
    \int_0^\infty x^{\alpha+2} |\partial_x^2 u|^{\alpha+1}\, \d x
    \Big)^\frac{4\alpha}{3\alpha-1}
\end{equation*}
for the first factor and the same, with $u$ replaced by $v$, for the third factor. Thus, for $J_2$ we end up with
\begin{align}\nonumber
    J_2
    & \lesssim_\alpha
    \Big(\int_0^\infty
    u^2\, \d x
    \Big)^\frac{\alpha-1}{2(3\alpha-1)}
    \Big(\int_0^\infty
    x^{\alpha+2} |\partial_x^2 u|^{\alpha+1}\, \d x
    \Big)^\frac{2\alpha}{(\alpha+1)(3\alpha-1)} \\
    & \phantom{\lesssim_\alpha} \times \Big(\int_0^\infty
    v^2\, \d x
    \Big)^\frac{\alpha-1}{2(3\alpha-1)} 
    \Big(\int_0^\infty
    x^{\alpha+2} |\partial_x^2 v|^{\alpha+1}\, \d x
    \Big)^\frac{\alpha(3\alpha+1)}{(\alpha+1)(3\alpha-1)}. \label{eq:N(v)u_J2}
\end{align}
Next, we estimate $J_3$. We have
\begin{equation*}
\begin{split}
    J_3
    \leq
    \sup_{x \geq 0} |u| 
    \int_0^\infty x^{\alpha+2} |\partial_x^2 v|^\alpha\, (\partial_x v)^2\, \d x
\end{split}
\end{equation*}
and observe that $\sup_{x \geq 0} |u|$ can again be estimated by estimate~\eqref{eq:weight xu} of Lemma~\ref{lem:weight xu} with $\beta=0$, while the integral can be treated like $J_2$ with $u=v$. This leads to the estimate
\begin{equation}\label{eq:N(v)u_J3}
    J_3
    \lesssim_\alpha
    \Big(\int_0^\infty u^2\, \d x\Big)^\frac{\alpha-1}{3\alpha-1}
    \Big(\int_0^\infty x^{\alpha+2} |\partial_x^2 u|^{\alpha+1} \, \d x\Big)^\frac{1}{3\alpha-1}
    \Big(\int_0^\infty
    v^2\, \d x
    \Big)^\frac{\alpha-1}{3\alpha-1} 
    \Big(\int_0^\infty
    x^{\alpha+2} |\partial_x^2 v|^{\alpha+1}\, \d x
    \Big)^\frac{3\alpha}{3\alpha-1}.
\end{equation}
By applying Lemma~\ref{lem:weight xu} once more, we obtain for $J_4$,
\begin{equation} \label{eq:N(v)u_J4}
    J_4 
    \, \stackrel{\mathclap{\eqref{eq:weight xu}}}{\lesssim}_\alpha 
    \Big(\int_0^\infty u^2\, \d x\Big)^\frac{\alpha-1}{3\alpha-1}
    \Big(\int_0^\infty x^{\alpha+2} |\partial_x^2 u|^{\alpha+1} \, \d x\Big)^\frac{1}{3\alpha-1}
    \int_0^\infty x^{\alpha+2} |\partial_x^2 v|^{\alpha+1}\, \d x.
\end{equation}
Finally, using \eqref{eq:N(v)u_J1}--\eqref{eq:N(v)u_J4} in \eqref{eq:N(v)u_J_i} we end up with
\begin{align}\nonumber
    I_1
    &\lesssim_\alpha
    \Big(\int_0^\infty x^{\alpha+2} |\partial_x^2 u|^{\alpha+1}\, \d x\Big)^\frac{1}{\alpha+1}
    \Big(\int_0^\infty v^2\, \d x\Big)^\frac{\alpha-1}{3\alpha-1}
    \Big(\int_0^\infty x^{\alpha+2} |\partial_x^2 v|^{\alpha+1}\, \d x\Big)^\frac{3\alpha^2+1}{(3\alpha-1)(\alpha+1)}
    \\
    &\phantom{\lesssim_\alpha}
    +
    \Big(\int_0^\infty
    u^2\, \d x
    \Big)^\frac{\alpha-1}{2(3\alpha-1)}
    \Big(\int_0^\infty
    x^{\alpha+2} |\partial_x^2 u|^{\alpha+1}\, \d x
    \Big)^\frac{2\alpha}{(\alpha+1)(3\alpha-1)} \nonumber \\
    &\phantom{\lesssim_\alpha +} \times \Big(\int_0^\infty
    v^2\, \d x
    \Big)^\frac{\alpha-1}{2(3\alpha-1)} 
    \Big(\int_0^\infty
    x^{\alpha+2} |\partial_x^2 v|^{\alpha+1}\, \d x
    \Big)^\frac{\alpha(3\alpha+1)}{(\alpha+1)(3\alpha-1)}
    \nonumber \\
    &\phantom{\lesssim_\alpha}
    +
    \Big(\int_0^\infty u^2\, \d x\Big)^\frac{\alpha-1}{3\alpha-1}
    \Big(\int_0^\infty x^{\alpha+2} |\partial_x^2 u|^{\alpha+1} \, \d x\Big)^\frac{1}{3\alpha-1}
    \Big(\int_0^\infty
    v^2\, \d x
    \Big)^\frac{\alpha-1}{3\alpha-1} 
    \Big(\int_0^\infty
    x^{\alpha+2} |\partial_x^2 v|^{\alpha+1}\, \d x
    \Big)^\frac{3\alpha}{3\alpha-1}
    \nonumber \\
    &\phantom{\lesssim_\alpha}
    +
    \Big(\int_0^\infty u^2\, \d x\Big)^\frac{\alpha-1}{3\alpha-1}
    \Big(\int_0^\infty x^{\alpha+2} |\partial_x^2 u|^{\alpha+1} \, \d x\Big)^\frac{1}{3\alpha-1}
    \int_0^\infty x^{\alpha+2} |\partial_x^2 v|^{\alpha+1}\, \d x. \label{eq:N(v)u_I1}
\end{align} 
We obtain a similar estimate for $I_2$. Indeed, recalling the definition of $g_\alpha$ and the assumption $\sup_{x\geq 0} |v| \leq \frac{1}{2} < 1$ because of \eqref{eq:weight xu} of Lemma~\ref{lem:weight xu} and $\|v\|_{U_\alpha} \ll_\alpha 1$, we  apply the mean-value theorem to the first factor and calculate $\partial^2_x\bigl(u\,  (1+v)^{\frac{3}{2}}\bigr)$. This yields
\begin{equation*}
\begin{split}
    I_2
    &=
    \int_0^{\infty} \bigl(1-(1+v)^{\frac{\alpha}{2}}\bigr) \, x^{\alpha+2}\, g_{\alpha}(\partial_x^2 v)\, \partial^2_x\bigl(u\,  (1+v)^{\frac{3}{2}}\bigr) \,\d x
    \\
    &\lesssim_\alpha
    \sup_{x \geq 0} |v|
    \Big(
    \int_0^\infty x^{\alpha+2} |\partial_x^2 v|^\alpha\, |\partial_x^2u|\, \d x 
    + 
    \int_0^\infty x^{\alpha+2} |\partial_x^2 v|^\alpha |\partial_x v|\, |\partial_x u|\, \d x
    \\
    &\phantom{\lesssim_\alpha
    \sup_{x \geq 0} |v| \Big(}
    + \int_0^\infty x^{\alpha+2} |\partial_x^2 v|^\alpha (\partial_x v)^2\, |u|\, \d x
    +
    \int_0^\infty x^{\alpha+2} |\partial_x^2 v|^{\alpha+1}\, |u|\, \d x
    \Big).
\end{split}
\end{equation*}
The second and third line contain exactly the same integrals $J_1$, $J_2$, $J_3$, and $J_4$ as in estimate~\eqref{eq:N(v)u_J_i} for $I_1$, except the additional pre-factor $\sup_{x \geq 0} |v|$ that appears now in front of all integrals and not only in the first one. However, since $\sup_{x \ge 0} |v| \le \frac 1 2$, estimate~\eqref{eq:N(v)u_I1} is also valid for $I_2$ instead of $I_1$, thus finishing the proof of \eqref{est-N(v)u}.
\end{proof}
\section{Existence of strong solutions using a time-discretization scheme} \label{sec:time-discretization}
Throughout the section we assume $\alpha > 2$. We address problem~\eqref{problem-u}, that is,
\begin{equation*}
u_t + \partial_x^2 \big(x^{\alpha+2} \, g_\alpha(\partial_x^2 u)\big) = f \quad \text{for } t,x > 0,
\end{equation*}
with the nonlinear right-hand side
\begin{align*}
    f\coloneqq
    N(u) &\coloneqq (1-  (1+u)^{\frac{3}{2}}) \,  \partial_x^2 (x^{\alpha+2} \, g_{\alpha}(\partial_x^2 u))
    + (1+u)^{\frac{3}{2}}\, \partial_x^2\big( (1- (1+u)^{\frac{\alpha}{2}}) x^{\alpha+2} \, g_{\alpha}(\partial_x^2 u) \big),
\end{align*}
by means of the following time-discretization scheme:
\begin{subequations}\label{full-discrete}
\begin{equation}\label{eq-full-discrete}
\frac{u_j - u_{j-1}}{h} + \partial_x^2 \big(x^{\alpha+2} g_\alpha(\partial_x^2 u_j)\big) = f_{j-1}, \qquad j \in \N,
\end{equation}
equipped with the boundary condition 
\begin{equation}\label{bc-full-discrete}
u_j = 0 \quad \text{at } x = 0,
\end{equation}
\end{subequations}
where $0 < h < \infty$, $t_j\coloneq jh$ (hence $t_0 =0$), $f_{j-1}$ is given by
\begin{equation}\label{def-fj}
f_{j-1} \coloneq N(u_{j-1})
\end{equation}
(cf.~\eqref{def-nu}), and the prescribed initial data 
$u_0(\cdot)\coloneq u(0,\cdot)$. Thanks to the approximation property of Lemma~\ref{lem-approx-ualpha}, we may first assume $u_0 \in C_\mathrm{c}^\infty((0,\infty))$.
On the linear space
\[
U_\alpha \stackrel{\eqref{def-ualpha}}{=} \big\{u \in L^2(0,\infty) \colon x^{\frac{\alpha+2}{\alpha+1}} \partial_x^2 u \in L^{\alpha+1}(0,\infty), \; u = 0 \text{ at } x = 0\big\},
\]
we introduce the functional 
\begin{equation}\label{def-jalpha}
U_\alpha \owns u \longmapsto J_{\alpha,j}(u) \coloneq \int_0^\infty \big(\tfrac 1 2 u^2 + \tfrac{h}{\alpha+1} x^{\alpha+2} |\partial_x^2 u|^{\alpha+1} - \tilde f_{j-1} u\big) \, \d x,
\end{equation}
where
\begin{equation}\label{def-gj}
\tilde f_j \coloneq h f_j + u_j. 
\end{equation}
Provided $\tilde f_{j-1}\in L^2(0,\infty)$ (in fact, we have the stronger regularity $\tilde f_0\in C^\infty_\mathrm{c}((0,\infty))$ for $j = 0$), the functional $J_{\alpha,j}$ is well-defined and finite on $U_\alpha$. The existence of a solution $u_j$ for \eqref{full-discrete} at step $j$ can be obtained by minimizing the functional $J_{\alpha,j}$ via the direct method in the calculus of variations by using the fact that the space $U_\alpha$ is weakly closed. Observe that the boundary condition \eqref{bc-full-discrete} is closed under weak convergence by estimate~\eqref{eq:weight xu} of Lemma \ref{lem:weight xu} and letting $x \downarrow 0$.

\begin{lemma}[Existence of minimizers for $J_{\alpha,j}$]\label{lem-ex-min}
Provided $\tilde f_{j-1} \in L^2(0,\infty)$, there exists a unique minimizer $u_j \in U_\alpha$ of the functional $J_{\alpha,j}$(cf.~\eqref{def-ualpha}, \eqref{def-jalpha}). 
\end{lemma}
\begin{proof}
\textbf{Step 1: $U_\alpha$ is non-empty. } We have the inclusion $C^\infty_\mathrm{c}((0,\infty)) \subset U_\alpha$ by Lemma~\ref{lem-approx-ualpha}.

\bigskip

\textbf{Step 2: Boundedness from below and coercivity.} We have by Young's inequality
\begin{equation*}
    - \frac 1 2 \|\tilde f_{j-1}\|_{L^2(0,\infty)}^2 
    \leq
    \inf_{u \in U_\alpha} J_{\alpha,j}(u)
    =
    \inf_{u \in U_\alpha}  \int_0^\infty \bigl(\tfrac{1}{2}u^2 + \tfrac{h}{\alpha+1} x^{\alpha+2} |\partial_x^2 u|^{\alpha+1} - \tilde f_{j-1} u\bigr) \d x 
\end{equation*}
by the definition \eqref{def-ualpha} of $U_\alpha$, so the functional is bounded from below.
Moreover, the functional is coercive in the sense that if $(u^{(k)})_{k \in \N} \subset U_\alpha$ is such that $\sup_{k \in \N} J_{\alpha,j}(u^{(k)}) \le C$, then by Young's inequality
\begin{equation*}
    \frac 1 4 \|u^{(k)}\|_{L^2(0,\infty)}^2
    +
    \frac{h}{\alpha+1}\|x^\frac{\alpha+2}{\alpha+1} \partial_x^2 u^{(k)}\|_{L^{\alpha+1}(0,\infty)}^{\alpha+1}
    - \|\tilde{f}_{j-1}\|_{L^2(0,\infty)}^2
    \leq 
    J_{\alpha,j}(u^{(k)}) 
    \le C.
\end{equation*}
Thus, for any minimizing sequence $(u^{(k)})_{k \in \N} \subset U_\alpha$ there exists a function $u_j \in U_\alpha$ such that for a subsequence (not relabeled)
\begin{equation*}
u^{(k)} \rightharpoonup u_j \quad \text{in } L^2(0,\infty)
    \qquad \text{and} \qquad
    x^{\frac{\alpha+2}{\alpha+1}} \partial_x^2 u^{(k)} \rightharpoonup x^{\frac{\alpha+2}{\alpha+1}} \partial_x^2 u_j
    \quad \text{in } L^{\alpha+1}(0,\infty),
\end{equation*}
where $u_j\in U_\alpha$ because $U_\alpha$ is weakly closed.

\bigskip

\textbf{Step 3: Weak lower semi-continuity. } 
This is ensured by the (strict) convexity of the functions $s \mapsto |s|^2$ and $\xi \mapsto |\xi|^{\alpha+1}$. Indeed, the linear space $U_\alpha$ is convex and the second Gateaux derivative is
\begin{equation}\label{second_ja}
    D^2 J_{\alpha,j}(u)(v,v) = \int_0^\infty \big( v^2 + \alpha h x^{\alpha+2} |\partial_x^2 u|^{\alpha-1} (\partial_x^2 v)^2\big) \, \d x \ge 0 \quad \text{for} \quad u,v \in U_\alpha.
\end{equation}
Thus, we have
\begin{equation*}
    J_{\alpha,j}(u_j) \leq \liminf_{k \to \infty} J_{\alpha,j}(u^{(k)})
\end{equation*}
and hence, if $u^{(k)}$ is a minimizing sequence, then $u_j$ is a minimizer.

\bigskip

\textbf{Step 4: Uniqueness. } 
For two minimizers $u_j, u_j' \in U_\alpha$ it holds
\begin{align*}
0 &= D J_{\alpha,j}(u_j)(u_j - u_j') - D J_{\alpha,j}(u_j')(u_j - u_j') \; = \; \int_0^1 D^2 J_{\alpha,j}(\tau u_j + (1-\tau) u_j') (u_j-u_j',u_j-u_j') \, \d \tau \\
&\stackrel{\mathclap{\eqref{second_ja}}}{\ge} \; \int_0^1 \int_0^\infty (u_j - u_j')^2 \, \d x \, \d \tau = \|u_j - u_j'\|_{L^2(0,\infty)}^2,
\end{align*}
so that $u_j = u_j'$ must hold almost everywhere.
\end{proof}
\begin{lemma}[Euler--Lagrange equation corresponding to $J_{\alpha,j}$] \label{lem:Euler-Lagrange_J}
Suppose $\tilde f_{j-1} \in L^2(0,\infty)$ and let $u_j \in U_\alpha$ be a minimizer of $J_{\alpha,j}$. Then we have the strong formulation of the Euler-Lagrange equation
\begin{equation} \label{eq:Euler-Lagrange_I}
     h\partial_x^2 (x^{\alpha+2} g_\alpha(\partial_x^2 u_j)) = \tilde f_{j-1} - u_j \quad \text{almost everywhere.}
\end{equation}
In particular, $\partial_x^2 (x^{\alpha+2} g_\alpha(\partial_x^2 u_j)) \in L^2(0,\infty)$, that is,
\[
u_j \in V_\alpha \stackrel{\eqref{def-valpha}}{=} \Big\{v \in U_\alpha \colon \int_0^\infty \big(\partial_x^2 (x^{\alpha+2} g_\alpha(\partial_x^2 v))\big)^2 \, \d x < \infty\Big\}.
\]
\end{lemma}
\begin{proof}
Since $D J_{\alpha,j}(u_j) = 0$ for the Gateaux derivative $D J_{\alpha,j}$ of $J_{\alpha,j}$, we have for $v \in C^\infty_\mathrm{c}((0,\infty))$,
\[
\int_0^\infty h x^{\alpha+2} g_\alpha(\partial_x^2 u_j) (\partial_x^2 v) \, \d x = \int_0^\infty (\tilde f_{j-1} - u_j) v \, \d x.
\]
This entails the identity of weak derivatives $h\partial_x^2 (x^{\alpha+2} g_\alpha(\partial_x^2 u_j)) = \tilde f_{j-1} - u_j \in L^2(0,\infty)$  and that \eqref{full-discrete} is satisfied almost everywhere.
\end{proof}

By testing the equation with $u_j$ we obtain the following lemma, which is the discretized analogue of \eqref{est-lead-weak-1}.
\begin{lemma}\label{lem-lead-weak-discrete}
Suppose that $u_1,\ldots,u_j$, constructed in Lemma~\ref{lem-ex-min} and Lemma~\ref{lem:Euler-Lagrange_J} provided that $\tilde f_0,\ldots, \tilde f_{j-1} \in L^2(0,\infty)$ exist. Then it holds
\begin{equation}\label{est-lead-weak-discrete-ineq}
    \frac{1}{2} \int_0^\infty (u_j^2-u_{j-1}^2)\, \d x 
    + h \int_0^\infty x^{\alpha+2} |\partial_x^2 u_j|^{\alpha+1}\, \d x
    \leq
    h \int_0^\infty f_{j-1}\, u_j\, \d x.
\end{equation}
\end{lemma}
\begin{proof}
Testing the discretized equation
\begin{equation*}
    \frac{u_j - u_{j-1}}{h} 
    +
    \partial_x^2 \left(
    x^{\alpha+2} |\partial_x^2 u_j|^{\alpha-1} \partial_x^2 u_j
    \right)
    =
    f_{j-1}
\end{equation*}
with $u_j \in U_\alpha$ and integrating by parts (cf.~\eqref{eq:B2-AC-1} of Lemma~\ref{lem:B2-AC}), where we have used that $\partial_x^2\bigl(x^{\alpha+2} |\partial_x^2 u_j|^{\alpha-1} \partial_x^2 u_j\bigr) \in L^2(0,\infty)$ due to Lemma~\ref{lem:Euler-Lagrange_J}, we obtain
\begin{equation*}
    \int_0^\infty \frac{u_j - u_{j-1}}{h} u_j \d x
    +
    \int_0^\infty x^{\alpha+2} |\partial_x^2 u_j|^{\alpha+1} \d x
    =
    \int_0^\infty f_{j-1} u_j\, \d x.
\end{equation*}
Using the elementary identity
\begin{equation*}
   (u_j-u_{j-1}) u_j = \frac 1 2 (u_j^2 - u_{j-1}^2) + \frac 1 2 (u_j-u_{j-1})^2
\end{equation*}
and multiplying by $h$, we arrive at \eqref{est-lead-weak-discrete-ineq}. 
\end{proof}

By testing the equation with $\partial_x^2\bigl(x^{\alpha+2} |\partial_x^2 u_j|^{\alpha-1} \partial_x^2 u_j\bigr)$ we obtain the following time-discretized analogue of estimate~\eqref{est-lead-max-1}.

\begin{lemma}\label{lem-lead-max-discrete}
Suppose that $u_1,\ldots,u_j$, constructed in Lemma~\ref{lem-ex-min} and Lemma~\ref{lem:Euler-Lagrange_J} provided that $\tilde f_0,\ldots, \tilde f_{j-1} \in L^2(0,\infty)$ exist. Then it holds
\begin{equation}\label{est-lead-max-discrete}
    \frac{2}{\alpha+1} \int_0^\infty x^{\alpha+2} \big(
    |\partial_x^2 u_j|^{\alpha+1} - |\partial_x^2 u_{j-1}|^{\alpha+1} 
    \big) \d x
    +
    h
    \int_0^\infty \big(\partial_x^2(x^{\alpha+2} g_\alpha(\partial_x^2 u_j))\big)^2 \d x
    \le
    h
    \int_0^\infty f_{j-1}^2\, \d x.
\end{equation}
\end{lemma}

\begin{proof}
Since $\partial_x^2\bigl(x^{\alpha+2} g_\alpha(\partial_x^2 u_j)\bigr) \in L^2(0,\infty)$ by Lemma \ref{lem:Euler-Lagrange_J}, we may test the discretized equation
\begin{equation*}
    \frac{u_j - u_{j-1}}{h} 
    +
    \partial_x^2 \left(
    x^{\alpha+2} g_\alpha(\partial_x^2 u_j)
    \right)
    =
    f_{j-1}
\end{equation*}
with $\partial_x^2\bigl(x^{\alpha+2} g_\alpha(\partial_x^2 u_j)\bigr)$ and integrate by parts (cf.~\eqref{eq:B2-AC-1} of Lemma~\ref{lem:B2-AC}) to obtain
\begin{equation*}
\begin{split}
    &\int_0^\infty x^{\alpha+2} |\partial_x^2 u_{j}|^{\alpha+1}\, \d x
    +
    h \int_0^\infty \big(\partial_x^2(x^{\alpha+2} g_\alpha(\partial_x^2 u_j))\big)^2 \d x
    -
    \int_0^\infty u_{j-1}\, \partial_x^2(x^{\alpha+2} g_\alpha\big(\partial_x^2 u_j)\big)\, \d x
    \\
    & \quad =
    h\int_0^\infty f_{j-1}\, \partial_x^2\big(x^{\alpha+2} g_\alpha(\partial_x^2 u_j)\big)\, \d x.
\end{split}
\end{equation*}
By Young's inequality, we have for the right-hand side
\begin{equation*}
h\int_0^\infty f_{j-1}\, \partial_x^2\big(x^{\alpha+2} g_\alpha(\partial_x^2 u_j)\big)\, \d x
\le \frac h 2 \int_0^\infty f_{j-1}^2 \, \d x
+ \frac h 2 \int_0^\infty \big(\partial_x^2\big(x^{\alpha+2}
g_\alpha(\partial_x^2 u_j)\big)\big)^2 \, \d x.  
\end{equation*}
Furthermore, integration by parts (justified as in the proof of Lemma~\ref{lem:B2-AC} by employing density of $C^\infty_\mathrm{c}((0,\infty))$ in $U_\alpha$ by Lemma~\ref{lem-approx-ualpha},
$\partial_x^2(x^{\alpha+2} g_\alpha(\partial_x^2 u_j)) \in L^2(0,\infty)$, and $u_{j-1} \in U_\alpha$)
and Young's inequality with exponents $\frac{\alpha+1}{\alpha}$ and $\alpha+1$ entail
\begin{equation*}
\begin{split}
    \int_0^\infty u_{j-1}\, \partial_x^2\big(x^{\alpha+2} g_\alpha(\partial_x^2 u_j)\big)\, \d x
    &=
    \int_0^\infty (\partial_x^2 u_{j-1}) \, x^{\alpha+2} g_\alpha(\partial_x^2 u_j) \, \d x
    \\
    &\leq
    \frac{\alpha}{\alpha+1} \int_0^\infty x^{\alpha+2} |\partial_x^2 u_j|^{\alpha+1}\, \d x
    +
    \frac{1}{\alpha+1} \int_0^\infty x^{\alpha+2} |\partial_x^2 u_{j-1}|^{\alpha+1}\, \d x,
\end{split} 
\end{equation*}
so that we obtain \eqref{est-lead-max-discrete}.
\end{proof}
\begin{proposition}[Time-discretized problem]\label{prop-discrete-est}
Let $u_0 \in C^\infty_\mathrm{c}((0,\infty))$, $\int_0^\infty (u_0^2 + x^{\alpha+2} |\partial_x^2 u_0|^{\alpha+1}) \, \d x < \infty$ be sufficiently small, and $h > 0$ be sufficiently small. Then the iteration scheme \eqref{full-discrete} is well-defined. Defining
\begin{align*}
f^h(t,x) &\coloneq \sum_{j \ge 1} f_{j-1}(x) \chi_{[(j-1)h,jh)}(t), \quad \text{where} \quad f_{j-1} = N(u_{j-1}), \\
u^h(t,x) &\coloneq \sum_{j \ge 1} u_j(x) \chi_{[(j-1)h,jh)}(t), \\
\tilde u^h(t,x) &\coloneq \sum_{j \ge 1} \big(\tfrac{t-(j-1)h}{h} u_j(x) + \tfrac{jh-t}{h} u_{j-1}(x)\big) \chi_{[(j-1)h,jh)}(t),
\end{align*}
we have the approximated equation
\begin{align}\label{eq-uh}
\partial_t \tilde u^h + \partial_x^2 (x^{\alpha+2} g_\alpha(\partial_x^2 u^h)) = f^h \quad \text{for almost every } x > 0,
\end{align}
and the a-priori estimate
\begin{align}\nonumber
& \sup_{t \in [0,T]} \int_0^\infty \big((u^h)^2 + x^{\alpha+2} |\partial_x^2 u^h|^{\alpha+1}\big) \, \d x \\
&+ \int_0^T \int_0^\infty \big((\partial_t \tilde u^h)^2 + x^{\alpha+2} |\partial_x^2 u^h|^{\alpha+1} + (\partial_x^2 (x^{\alpha+2} g_\alpha(\partial_x^2 u^h)))^2\big) \, \d x \, \d t \nonumber\\
& \quad \le C_\alpha \Big(\int_0^\infty \big(u_0^2 + x^{\alpha+2} |\partial_x^2 u_0|^{\alpha+1}\big) \, \d x + h \int_0^\infty (\partial_x^2 (x^{\alpha+2} g_\alpha(\partial_x^2 u_0)))^2 \, \d x\Big) \label{apriori-uh-nonlinear}
\end{align}
for a constant $C_\alpha < \infty$ sufficiently large and depending only on $\alpha$, and all $T \in (0,\infty]$. 
\end{proposition}
\begin{proof}
Since $\tilde f_0 \in C^\infty_\mathrm{c}((0,\infty))$ and applying Lemma~\ref{lem-ex-min} and Lemma~\ref{lem:Euler-Lagrange_J}, a unique minimizer $u_1 \in V_\alpha$ exists. Suppose that the iteration scheme $u_1, \ldots, u_N$ is well-defined up to $T = N h$ for some $N \in \N$. We sum estimate~\eqref{est-lead-weak-discrete-ineq} of Lemma~\ref{lem-lead-weak-discrete} and estimate~\eqref{est-lead-max-discrete} of Lemma~\ref{lem-lead-max-discrete} over $j = 1,\ldots,N$ and obtain
\begin{align*}
& \int_0^\infty \big(\tfrac 1 2 u_N^2 + \tfrac{2}{\alpha+1} x^{\alpha+2} |\partial_x^2 u_N|^{\alpha+1}\big) \, \d x \\
& + \sum_{j = 1}^N h \int_0^\infty \big(x^{\alpha+2} |\partial_x^2 u_j|^{\alpha+1} + (\partial_x^2 (x^{\alpha+2} g_\alpha(\partial_x^2 u_j)))^2\big) \, \d x \\
& \quad \le \int_0^\infty \big(\tfrac 1 2 u_0^2 + \tfrac{2}{\alpha+1} x^{\alpha+2} |\partial_x^2 u_0|^{\alpha+1}\big) \, \d x + \sum_{j = 1}^N h \int_0^\infty f_{j-1} u_j \, \d x + \sum_{j = 1}^N h \int_0^\infty f_{j-1}^2 \, \d x.
\end{align*}
On noting that \eqref{eq-uh} is valid by \eqref{eq:Euler-Lagrange_I} of Lemma~\ref{lem:Euler-Lagrange_J}, this entails for $T \in (0,\infty]$,
\begin{align}\nonumber
& \sup_{t \in [0,T]} \int_0^\infty \big(\tfrac 1 2 (u^h)^2 + \tfrac{2}{\alpha+1} x^{\alpha+2} |\partial_x^2 u^h|^{\alpha+1}\big) \, \d x \\
& + \int_0^T \int_0^\infty \big((\partial_t \tilde u^h)^2 + x^{\alpha+2} |\partial_x^2 u^h|^{\alpha+1} + (\partial_x^2 (x^{\alpha+2} g_\alpha(\partial_x^2 u^h)))^2\big) \, \d x \, \d t \nonumber\\
& \quad \le 3 \int_0^\infty \big(\tfrac 1 2 u_0^2 + \tfrac{2}{\alpha+1} x^{\alpha+2} |\partial_x^2 u_0|^{\alpha+1}\big) \, \d x + 3 \int_0^T \int_0^\infty f^h \, u^h \, \d x \, \d t + 5 \int_0^T \int_0^\infty (f^h)^2 \, \d x \, \d t. \label{apriori-uh}
\end{align}
Now we can estimate
\[
\int_0^T \int_0^\infty f^h \, u^h \, \d x \, \d t \qquad \text{and} \qquad \int_0^T \int_0^\infty (f^h)^2 \, \d x \, \d t
\]
by estimate~\eqref{est-N(v)u} of Proposition~\ref{prop:N(v)u} and estimate~\eqref{est-nv} of Proposition~\ref{prop-est-nv}, respectively. This follows iteratively for $T = j h$ with $j = 1,\ldots,N$ on applying Young's inequality provided
\[
\int_0^\infty \big(\tfrac 1 2 u_0^2 + \tfrac{2}{\alpha+1} x^{\alpha+2} |\partial_x^2 u_0|^{\alpha+1}\big) \, \d x
\]
is sufficiently small and $h > 0$ is small enough, as this ensures uniform smallness of
\[
\int_0^\infty \big(\tfrac 1 2 u_j^2 + \tfrac{2}{\alpha+1} x^{\alpha+2} |\partial_x^2 u_j|^{\alpha+1}\big) \, \d x
\]
for each $j = 1,\ldots,N$. This implies the a-priori estimate \eqref{apriori-uh-nonlinear} with a constant $C_\alpha < \infty$ independent of $T$. By estimate~\eqref{est-nv} of Proposition~\ref{prop-est-nv} we then have $f_N \stackrel{\eqref{def-fj}}{=} N(u_N) \in L^2(0,\infty)$, so that by \eqref{def-gj} we have $\tilde f_N \in L^2(0,\infty)$, and applying Lemma~\ref{lem-ex-min} and Lemma~\ref{lem:Euler-Lagrange_J}, $u_{N+1} \in V_\alpha$ exists. The existence of $u_j$ for all $j \in \N$ then follows by complete induction.
\end{proof}

Next we collect convergence results for the approximations $\tilde u^h$ and $u^h$:
\begin{proposition}[Uniform and weak convergence]\label{lem-uniform}
Suppose $u_0 \in C^\infty_\mathrm{c}((0,\infty))$ is such that $\int_0^\infty (u_0^2 + x^{\alpha+2} |\partial_x^2 u_0|^{\alpha+1}) \, \d x < \infty$ is sufficiently small and $h > 0$ is sufficiently small. Define $f^h$, $\tilde u^h$, and $u^h$ as in Proposition~\ref{prop-discrete-est}. Then we have
\begin{enumerate}[(i)]
\item\label{it:weak*} a subsequence of $u^h$, again denoted by $u^h$, converges weak-$*$ in $L^\infty(0,\infty;L^2(0,\infty))$;
\item\label{it:weak}
a subsequence of $u^h$, again denoted by $u^h$, is such that $x^{-\frac{\alpha}{\alpha+1}+j} \partial_x^j u^h$ for $j = 0,1,2$ converges weakly in $L^{\alpha+1}((0,\infty)^2)$ to $x^{-\frac{\alpha}{\alpha+1}+j} \partial_x^j u$ as $h \downarrow 0$, respectively;
\item\label{it:uniform} a subsequence of $\tilde u^h$ converges uniformly on compact subsets of $[0,\infty)^2$ to $u$ as $h \downarrow 0$; consequently, the same holds for $u^h$.
\item\label{it:time} a subsequence of $\partial_t \tilde u^h$ converges weakly to $\partial_t u$ in $L^2((0,\infty)^2)$ as $h \downarrow 0$.
\item\label{it:rhs} a subsequence of $f^h$ converges weakly to some $f$ in $L^2((0,\infty)^2)$ as $h \downarrow 0$;
\item\label{it:minty} a subsequence of $\partial_x^2(x^{\alpha+2} g_\alpha(\partial_x^2 u^h))$ converges weakly to $\partial_x^2 (x^{\alpha+2} g_\alpha(\partial_x^2 u))$ in $L^2((0,\infty)^2)$ as $h \downarrow 0$;
\item\label{it:weak2} 
a subsequence of $x^{\alpha+2} |\partial_x^2 u^h|^{\alpha+1}$, again 
denoted the same, converges weakly in 
$L^{2}((0,\infty)^2)$ to $x^{\alpha+2} |\partial_x^2 u|^{\alpha+1}$ 
as $h \downarrow 0$;
\item\label{it:rhs2} $f = N(u)$.
\end{enumerate}
In particular, we have the weak formulation
\begin{align}\nonumber
& \int_0^\infty v(t,x) u(t,x) \, \d x + \int_0^t \int_0^\infty v(t',x) \partial_x^2 (x^{\alpha+2} g_\alpha(\partial_x^2 u(t',x))) \, \d x \, \d t' \\
& - \int_0^t \int_0^\infty (\partial_t v)(t',x) u(t',x) \, \d x \, \d t' - \int_0^\infty v(0,x) u_0(x) \, \d x - \int_0^t \int_0^\infty v(t',x) N(u(t',x)) \, \d x \, \d t' = 0  \label{eq-u-tested}
\end{align}
for all $v \in C^\infty_\mathrm{c}([0,\infty) \times (0,\infty))$.
\end{proposition}
\begin{proof} We prove the statements separately.

\bigskip
\textbf{Proof of \eqref{it:weak*}. }
This follows from the bound \eqref{apriori-uh-nonlinear} of Proposition~\ref{prop-discrete-est} and weak-$*$-compactness by the Banach--Alaoglu theorem.

\bigskip
\textbf{Proof of \eqref{it:weak}. }
This follows from the bound \eqref{apriori-uh-nonlinear} of Proposition~\ref{prop-discrete-est}, weak-$*$-compactness by the Banach--Alaoglu theorem, and Hardy's inequality \eqref{eq:Hardy} of Lemma~\ref{lem:Hardy}.

\bigskip
\textbf{Proof of \eqref{it:uniform}. }
From \eqref{apriori-uh-nonlinear} of Proposition~\ref{prop-discrete-est}, it follows that $\partial_t \tilde u^h$ is bounded in $L^2((0,T)\times(0,\infty))$. Furthermore, using $\alpha > 2$, from estimate~\eqref{eq:weight xu} of Lemma~\ref{lem:weight xu} and estimate~\eqref{eq:weight xu'} of Lemma~\ref{lem:weight xu'} in combination with \eqref{apriori-uh-nonlinear} of Proposition~\ref{prop-discrete-est}, we infer that $\tilde u^h \in L^\infty(0,T;C^\beta([0,M]))$, where $\beta = \frac{\alpha-1}{\alpha+1}$, is bounded for every $M > 0$. The classical Aubin--Lions--Simon lemma \cite[Corollary~4]{simon1986compact} entails that $\tilde u^h$ has a uniformly convergent subsequence in $[0,T] \times [0,M]$. Exhausting $[0,\infty)^2$ with a countable number of compact intervals $[0,T] \times [0,M]$, an iterated subsequence argument implies that a subsequence of $\tilde u^h$, again denoted by $\tilde u^h$, converges uniformly to some $\tilde u$ on compact subsets of $[0,\infty)^2$. 
It remains to identify $\tilde{u} = u$. With the identification $u^h(t) = \tilde u^h(j_{t,h}h)$ with $j_{t,h} \in \N_0$ minimal such that $j_{t,h} h \ge t$ (in particular $j_{t,h}h-t < h$), we have for $K \Subset [0,\infty)^2$:
\[
\sup_{(t,x) \in K} |u^h(t,x) - \tilde u(t,x)| \le \sup_{(t,x) \in K} |\tilde u^h(j_{t,h} h,x) - \tilde u(j_{t,h} h,x)| + \sup_{(t,x) \in K} |\tilde u(j_{t,h} h,x) - \tilde u(t,x)|.
\]
By uniform convergence of $\tilde u^h$ on compact subsets of $[0,\infty)^2$ established above, we have
\[
\sup_{(t,x) \in K} |\tilde u^h(j_{t,h} h,x) - \tilde u(j_{t,h} h,x)| \to 0 \quad \text{as} \quad h \downarrow 0.
\]
Since $\tilde u$ is uniformly continuous on $K$, we have
\[
\sup_{(t,x) \in K} |\tilde u(j_{t,h} h,x) - \tilde u(t,x)| \to 0 \quad \text{as} \quad h \downarrow 0.
\]
Hence, also $u^h$ converges uniformly on compact subsets of $[0,\infty)^2$ to $\tilde u$. Thus, it holds
\[
\int_0^t \int_0^\infty x^{-\frac{\alpha}{\alpha+1}} u^h v \, \d x \, \d t' \longrightarrow \int_0^t \int_0^\infty x^{-\frac{\alpha}{\alpha+1}} \tilde u \, v \, \d x \, \d t' \quad \text{as} \quad h \downarrow 0,
\]
for any $v \in C^\infty_\mathrm{c}((0,\infty)^2)$.  Since by part~\eqref{it:weak} we have
\[
\int_0^t \int_0^\infty x^{-\frac{\alpha}{\alpha+1}} u^h v \, \d x \, \d t' \longrightarrow \int_0^t \int_0^\infty x^{-\frac{\alpha}{\alpha+1}} u \, v \, \d x \, \d t' \quad \text{as} \quad h \downarrow 0,
\]
it follows $\tilde u = u$ by the fundamental lemma of calculus of variations.

\bigskip
\textbf{Proof of \eqref{it:time}. }
Since $\partial_t \tilde u^h$ is bounded in $L^2((0,\infty)^2)$ by \eqref{apriori-uh-nonlinear} of Proposition \ref{prop-discrete-est}, the Banach--Alaoglu theorem allows us to extract a subsequence, again denoted by $\partial_t \tilde u^h$, that converges weakly to some $\check u \in L^2((0,\infty)^2)$ as $h \downarrow 0$. By part~\eqref{it:uniform} it holds for $v \in C^\infty_\mathrm{c}((0,\infty)^2)$,
\[
\int_0^\infty \int_0^\infty (\partial_t \tilde u^h) v \, \d x \, \d t = - \int_0^\infty \int_0^\infty \tilde u^h (\partial_t v) \, \d x \, \d t \longrightarrow - \int_0^\infty \int_0^\infty u \, (\partial_t v) \, \d x \, \d t  \quad \text{as} \quad h \downarrow 0,
\]
which implies $\partial_t u = \check u$.

\bigskip
\textbf{Proof of \eqref{it:rhs}. }
By estimate~\eqref{est-nv} of Proposition~\ref{prop-est-nv} and \eqref{apriori-uh-nonlinear} of Proposition~\ref{prop-discrete-est} the sequence $f^h$ is bounded in $L^2((0,\infty)^2)$, so that the claim follows again by means of the Banach--Alaoglu theorem.

\bigskip
\textbf{Proof of \eqref{it:minty}. }
For $v \in C^\infty_\mathrm{c}([0,\infty) \times (0,\infty))$ it follows from \eqref{eq-uh} of Proposition~\ref{prop-discrete-est} that for $t > 0$,
\begin{align}\nonumber
& \int_0^\infty v(t,x) \tilde u^h(t,x) \, \d x - \int_0^\infty v(0,x) u_0(x) \, \d x - \int_0^t \int_0^\infty (\partial_t v)(t',x) \tilde u^h(t',x) \, \d x \, \d t'
\\
& + \int_0^t \int_0^\infty x^{\alpha+2} (\partial_x^2 v)(t',x) g_\alpha(\partial_x^2 u^h(t',x)) \, \d x \, \d t' - \int_0^t \int_0^\infty v(t',x) f^h(t',x) \, \d x \, \d t' = 0. \label{eq-uh-tested}
\end{align}
Furthermore, since $\int_0^\infty \int_0^\infty |\partial_x^2(x^{\alpha+2} g_\alpha(\partial_x^2 u^h))|^2 \, \d x \, \d t'$ is uniformly bounded by estimate~\eqref{apriori-uh-nonlinear} of Proposition~\ref{prop-discrete-est}, we infer by the Banach--Alaoglu theorem the existence of some $\psi \in L^2((0,\infty)^2)$ such that $\partial_x^2(x^{\alpha+2} g_\alpha(\partial_x^2 u^h)) \rightharpoonup \psi$ in $L^2((0,t) \times (0,\infty))$ for every $t > 0$ as $h \downarrow 0$, upon passing to a subsequence. Then it holds for $v \in C_\mathrm{c}^\infty((0,t) \times (0,\infty))$,
\[
\int_0^t \int_0^\infty v \partial_x^2\big(x^{\alpha+2} g_\alpha(\partial_x^2 u^h)\big) \, \d x \, \d t' = \int_0^t \int_0^\infty x^{\alpha+2} (\partial_x^2 v) g_\alpha(\partial_x^2 u^h) \, \d x \, \d t'.
\]
We can now use Minty's argument to identify the limit $\psi$. To this end, observe that for $a,b \in \R$ and some $\theta \in (0,1)$, we have
\[
(g_\alpha(a)-g_\alpha(b)) (a-b) = \alpha |\theta a + (1-\theta) b|^{\alpha-1} (a-b)^2 \ge 0.
\]
Due to estimate~\eqref{apriori-uh-nonlinear} of Proposition~\ref{prop-discrete-est},
the sequence $x^{\frac{\alpha(\alpha+2)}{\alpha+1}} g_\alpha(\partial_x^2 u^h)$
is bounded in $L^{\frac{\alpha+1}{\alpha}}((0,\infty)^2)$. By the Banach--Alaoglu theorem there exists $\phi$ with $x^{\frac{\alpha(\alpha+2)}{\alpha+1}} \phi\in L^{\frac{\alpha+1}{\alpha}}((0,\infty)^2)$ such that, upon passing to a subsequence,
\[
x^{\frac{\alpha(\alpha+2)}{\alpha+1}} g_\alpha(\partial_x^2 u^h)
\rightharpoonup x^{\frac{\alpha(\alpha+2)}{\alpha+1}} \phi
\quad \text{in } L^{\frac{\alpha+1}{\alpha}}((0,t) \times (0,\infty))
\text{ for every } t > 0 \text{ as } h \downarrow 0.
\]
We then have for $\omega \colon (0,t) \times (0,\infty) \to \R$ locally integrable with $\int_0^t \int_0^\infty x^{\alpha+2} |\omega|^{\alpha+1} \, \d x \, \d t' < \infty$, $\eta \in C^\infty([0,\infty))$ with $\eta|_{[0,1]} = 1$, $\eta|_{[2,\infty)} = 0$, and $\eta_k(x) \coloneq \eta(x/k)$, $k \in \N$, 
\begin{align*}
R_{h,k} &\le \int_0^t \int_0^\infty x^{\alpha+2} \big(g_\alpha(\omega) - g_\alpha(\partial_x^2 u^h)\big) \big(\eta_k \omega - \partial_x^2 (\eta_k u^h)\big) \, \d x \d t' \\
&= \int_0^t \int_0^\infty x^{\alpha+2} |\omega|^{\alpha+1} \eta_k \, \d x \, \d t' - \int_0^t \int_0^\infty x^{\alpha+2} g_\alpha(\partial_x^2 u^h) \omega \eta_k \, \d x \, \d t' \\
& \phantom{=} - \int_0^t \int_0^\infty x^{\alpha+2} g_\alpha(\omega) \partial_x^2 (\eta_k u^h) \, \d x \, \d t' + \int_0^t \int_0^\infty x^{\alpha+2} g_\alpha(\partial_x^2 u^h) \partial_x^2 (\eta_k u^h) \, \d x \, \d t',
\end{align*}
where we can estimate the remainder
\[
R_{h,k} \coloneq \int_0^t \int_0^\infty x^{\alpha+2} \big(g_\alpha(\omega) - g_\alpha(\partial_x^2 u^h)\big) [\partial_x^2,\eta_k]_- u^h \, \d x \d t'
\]
with help of H\"older's inequality according to
\begin{align*}
|R_{h,k}| &\lesssim_\alpha \Big(\int_0^t \int_k^{2k} x^{\alpha+2} |\omega|^{\alpha+1} \, \d x \, \d t' + \int_0^t \int_k^{2k} x^{\alpha+2} |\partial_x^2 u^h|^{\alpha+1} \, \d x \, \d t'\Big)^{\frac{\alpha}{\alpha+1}} \\
&\phantom{\lesssim_\alpha} \times \Big(\int_0^t \int_k^{2k} x^{\alpha+2} \big(k^{-\alpha-1} |\partial_x u^h|^{\alpha+1} + k^{-2\alpha-2} |u^h|^{\alpha+1}\big) \, \d x \, \d t'\Big)^{\frac{1}{\alpha+1}}.
\end{align*}
We then note that
\begin{align*}
\int_0^t \int_k^{2k} x^{\alpha+2} |\omega|^{\alpha+1} \, \d x \, \d t' \;\, &\le_{\phantom{\alpha}} \int_0^t \int_0^\infty x^{\alpha+2} |\omega|^{\alpha+1} \, \d x \, \d t' < \infty, \\
\int_0^t \int_k^{2k} x^{\alpha+2} |\partial_x^2 u^h|^{\alpha+1} \, \d x \, \d t' \;\, &\stackrel{\mathclap{\eqref{apriori-uh-nonlinear}}}{\lesssim}_\alpha \int_0^\infty \big(u_0^2 + x^{\alpha+2} |\partial_x^2 u_0|^{\alpha+1}\big) \, \d x + h \int_0^\infty \big(\partial_x^2 (x^{\alpha+2} g_\alpha(\partial_x^2 u_0))\big)^2 \, \d x
\end{align*}
by Proposition~\ref{prop-discrete-est}. Furthermore, by Lemma~\ref{lem:weight xu}, Lemma~\ref{lem:weight xu'}, and H\"older's inequality,
\begin{align*}
& \Big(\int_0^t \int_k^{2k} x^{\alpha+2} k^{-2\alpha-2} |u^h|^{\alpha+1} \, \d x \, \d t'\Big)^{\frac{1}{\alpha+1}} \\
& \quad \stackrel{\mathclap{\eqref{eq:weight xu}}}{\lesssim}_{\alpha,\beta_1} k^{\frac{1-\alpha - (\alpha+1)\beta_1}{\alpha+1}} \, t^{\frac{2(\alpha -1 + (\alpha +1)\beta_1)}{(\alpha+1)(3\alpha -1)}}\, \Big( \sup_{0 \le t' \le t} \int_0^\infty (u^h)^2 \, \d x\Big)^{\frac{\alpha -1 + (\alpha +1)\beta_1}{3\alpha -1}} \\
& \quad \phantom{\lesssim_{\alpha,\beta_1}} \times \Big(\int_0^t \int_0^\infty x^{\alpha+2} \, |\partial_x^2 u^h|^{\alpha+1} \, \d x \, \d t'\Big)^{\frac{1-2\beta_1}{3\alpha -1}}, \\[0.2cm]
& \Big(\int_0^t \int_k^{2k} x^{\alpha+2} k^{-\alpha-1} |\partial_x u^h|^{\alpha+1} \, \d x \, \d t'\Big)^{\frac{1}{\alpha+1}} \\
& \quad \stackrel{\mathclap{\eqref{eq:weight xu'}}}{\lesssim}_{\alpha,\beta_2} k^{\frac{2-(\alpha+1)\beta_2}{\alpha+1}} \, t^{\frac{2((\alpha+1)\beta_2-2)}{(\alpha+1)(3\alpha-1)}}\, \Big( \sup_{0 \le t' \le t} \int_0^\infty (u^h)^2 \, \d x\Big)^{\frac{(\alpha +1)\beta_2 -2}{3\alpha -1}} \Big(\int_0^t \int_0^\infty x^{\alpha+2} \, |\partial_x^2 u^h|^{\alpha+1} \, \d x \, \d t'\Big)^{\frac{3-2\beta_2}{3\alpha -1}},
\end{align*}
where \(\beta_1 \in \big(\frac{1-\alpha}{\alpha +1}, \frac{\alpha+2}{5\alpha +3}\big)\) and \(\beta_2 \in \big(\frac{2}{\alpha +1}, \frac{3(\alpha+2)}{5\alpha +3}\big)\). By the a-priori estimate \eqref{apriori-uh-nonlinear} of Proposition~\ref{prop-discrete-est} and for $k \ge 1$, there exists a constant $C(\alpha,u_0,t) < \infty$ only depending on $\alpha$, $u_0$ and $t$, as well as a $\gamma > 0$ depending only on $\alpha$ such that
\begin{equation}\label{est-remainder-minty}
|R_{h,k}| \le C(\alpha,u_0,t) k^{-\gamma}.
\end{equation}
Next, we identify the weak limits
\begin{align*}
\int_0^t \int_0^\infty x^{\alpha+2} g_\alpha(\partial_x^2 u^h) \omega \eta_k \, \d x \, \d t' &\to \int_0^t \int_0^\infty x^{\alpha+2} \phi \, \omega \eta_k \, \d x \, \d t' && \text{as } h \downarrow 0, \\
\int_0^t \int_0^\infty x^{\alpha+2} g_\alpha(\omega) \partial_x^2 (\eta_k u^h) \, \d x \, \d t' &\to \int_0^t \int_0^\infty x^{\alpha+2} g_\alpha(\omega) \partial_x^2 (\eta_k u) \, \d x \, \d t' && \text{as } h \downarrow 0,
\end{align*}
where we have used part~\eqref{it:weak}. Now, taking the limit $h \downarrow 0$ in \eqref{eq-uh-tested} while using parts~\eqref{it:uniform} and \eqref{it:rhs}, we obtain the weak formulation
\begin{align}\nonumber
& \int_0^\infty v(t,x) u(t,x) \, \d x - \int_0^\infty v(0,x) u_0(x) \, \d x - \int_0^t \int_0^\infty (\partial_t v)(t',x) u(t',x) \, \d x \, \d t'\\
& + \int_0^t \int_0^\infty x^{\alpha+2} (\partial_x^2 v)(t',x) \phi(t',x) \, \d x \, \d t' - \int_0^t \int_0^\infty v(t',x) f(t',x) \, \d x \, \d t' = 0. \label{eq-u-weak}
\end{align}
This already implies $\partial_t u + \partial_x^2 (x^{\alpha+2} \phi) = f$ almost everywhere on $(0,\infty)^2$ and that $u = u_0$ at $t = 0$ weakly in $L^2(0,\infty)$. By Proposition~\ref{prop-discrete-est} and parts~\eqref{it:uniform}, \eqref{it:time} and \eqref{it:rhs}, we have for every $t>0$,
\begin{align*}
\lim_{h \downarrow 0} \int_0^t \int_0^\infty x^{\alpha+2} g_\alpha(\partial_x^2 u^h) \partial_x^2 (\eta_k u^h) \, \d x \, \d t' \quad &= \quad \lim_{h \downarrow 0} \int_0^t \int_0^\infty \eta_k \, (\partial_x^2 (x^{\alpha+2} g_\alpha(\partial_x^2 u^h))) \, u^h \, \d x \, \d t' \\
&\stackrel{\mathclap{\eqref{eq-uh}}}{=} \quad \lim_{h \downarrow 0} \int_0^t \int_0^\infty \eta_k \, (f^h - \partial_t \tilde u^h) \, u^h \, \d x \, \d t' \\
&= \quad \int_0^t \int_0^\infty \eta_k \, (f - \partial_t u) \, u \, \d x \, \d t' \\
&= \quad \int_0^t \int_0^\infty \eta_k \, (\partial_x^2 (x^{\alpha+2} \phi)) \, u \, \d x \, \d t' \\
&= \quad \int_0^t \int_0^\infty x^{\alpha+2} \, \phi \, \partial_x^2 (\eta_k u) \, \d x \, \d t'.
\end{align*}
With \eqref{est-remainder-minty}, for every $t > 0$ and every $k \ge 1$, this amounts to
\begin{align*}
- C(\alpha,u_0,t) k^{-\gamma} &\le \int_0^t \int_0^\infty x^{\alpha+2} |\omega|^{\alpha+1} \eta_k \, \d x \, \d t' - \int_0^t \int_0^\infty x^{\alpha+2} \phi \,  \omega \eta_k \, \d x \, \d t' \\
& \phantom{=} - \int_0^t \int_0^\infty x^{\alpha+2} g_\alpha(\omega) \partial_x^2 (\eta_k u) \, \d x \, \d t' + \int_0^t \int_0^\infty x^{\alpha+2} \phi \, \partial_x^2 (\eta_k u) \, \d x \, \d t'.
\end{align*}
Taking the limit $k \to \infty$ implies by dominated convergence that for every $t > 0$,
\begin{align*}
0 &\le \int_0^t \int_0^\infty x^{\alpha+2} |\omega|^{\alpha+1} \, \d x \, \d t' - \int_0^t \int_0^\infty x^{\alpha+2} \phi \, \omega \, \d x \, \d t' \\
& \phantom{=} - \int_0^t \int_0^\infty x^{\alpha+2} g_\alpha(\omega) \, (\partial_x^2 u) \, \d x \, \d t' + \int_0^t \int_0^\infty x^{\alpha+2} \phi \, (\partial_x^2 u) \, \d x \, \d t' \\
&= \int_0^t \int_0^\infty x^{\alpha+2} (g_\alpha(\omega) - \phi) (\omega - \partial_x^2 u) \, \d x \, \d t'.
\end{align*}
Replacing $\omega$ by $\partial_x^2 u \pm \tau \omega$ for $\tau > 0$, dividing by $\tau$ and taking the limit $\tau \to 0$, we obtain, for every $t>0$,
\[
\int_0^t \int_0^\infty x^{\alpha+2} (g_\alpha(\partial_x^2 u) - \phi) \omega \, \d x \, \d t' = 0,
\]
which implies $\phi = g_\alpha(\partial_x^2 u)$ almost everywhere on $(0,\infty)^2$. Hence, for every $t > 0$ and all $v \in C^\infty_\mathrm{c}((0,t) \times (0,\infty))$,
\begin{align*}
\int_0^t \int_0^\infty v \psi \, \d x \, \d t' &= \lim_{h \downarrow 0} \int_0^t \int_0^\infty v \partial_x^2(x^{\alpha+2} g_\alpha(\partial_x^2 u^h)) \, \d x \, \d t' \\
&= \lim_{h \downarrow 0} \int_0^t \int_0^\infty x^{\alpha+2} (\partial_x^2 v) g_\alpha(\partial_x^2 u^h) \, \d x \, \d t' \\
&= \int_0^t \int_0^\infty x^{\alpha+2} (\partial_x^2 v) g_\alpha(\partial_x^2 u) \, \d x \, \d t',
\end{align*}
so that the claim follows.

\bigskip
\textbf{Proof of \eqref{it:weak2}. }
By~\eqref{eq:con-lem-sup-uxx} of Lemma~\ref{lem:con-lem-sup-uxx} squared, \eqref{apriori-uh-nonlinear} of Proposition~\ref{prop-discrete-est}, and H\"older's inequality in $t$,
\[
\sup_{h > 0} \int_0^\infty \int_0^\infty \bigl(x^{\alpha+2} |\partial_x^2 u^h|^{\alpha+1}\bigr)^2 \, \d x \, \d t < \infty,
\]
so that the Banach-Alaoglu theorem allows us to extract a subsequence, again denoted by $u^h$, and some $\bar u \in L^2((0,\infty)^2)$ such that $x^{\alpha+2} |\partial_x^2 u^h|^{\alpha+1} \rightharpoonup \bar u$ in $L^2((0,\infty)^2)$ as $h \downarrow 0$. For a test function $v \in C^\infty_\mathrm{c}((0,\infty)^2)$ two integrations by parts give
\begin{align*}
& \int_0^\infty \int_0^\infty x^{\alpha+2} |\partial_x^2 u^h|^{\alpha+1} v \, \d x \, \d t \\
& \quad = \int_0^\infty \int_0^\infty u^h (\partial_x^2(x^{\alpha+2} g_\alpha(\partial_x^2 u^h) v)) \, \d x \, \d t \\
& \quad = \int_0^\infty \int_0^\infty u^h (\partial_x^2(x^{\alpha+2} g_\alpha(\partial_x^2 u^h))) v \, \d x \, \d t + 2 \int_0^\infty \int_0^\infty u^h (\partial_x(x^{\alpha+2} g_\alpha(\partial_x^2 u^h))) (\partial_x v) \, \d x \, \d t \\
& \quad \phantom{=} + \int_0^\infty \int_0^\infty u^h (x^{\alpha+2} g_\alpha(\partial_x^2 u^h)) (\partial_x^2 v) \, \d x \, \d t \eqcolon I_{1,h} + I_{2,h} + I_{3,h}.
\end{align*}
By parts~\eqref{it:uniform} and~\eqref{it:minty}, we have
\[
\lim_{h \downarrow 0} I_{1,h} = \int_0^\infty \int_0^\infty u (\partial_x^2(x^{\alpha+2} g_\alpha(\partial_x^2 u))) v \, \d x \, \d t.
\]
For the limits of $I_{2,h}$ and $I_{3,h}$, we exhaust $(0,\infty)^2$ by a countable collection of compact subsets $K \times L$. By~\eqref{fundamental-w} of Lemma~\ref{lem-valpha}, \eqref{est-coeff-a} of Lemma~\ref{lem-sup-wx}, \eqref{apriori-uh-nonlinear} of Proposition~\ref{prop-discrete-est}, and H\"older's inequality in $t$, the sequence $\partial_x(x^{\alpha+2} g_\alpha(\partial_x^2 u^h))$ is bounded in $L^2(K \times L)$ uniformly in $h$. Using $\lim_{x \downarrow 0} (x^{\alpha+2} g_\alpha(\partial_x^2 u^h))(t,x) = 0$ from~\eqref{asymptotic-vw} of Lemma~\ref{lem-valpha}, Cauchy-Schwarz in $x$ further gives
\[
\bigl|x^{\alpha+2} g_\alpha(\partial_x^2 u^h)(t,x)\bigr|^2 \le x \int_0^x \bigl|\partial_x(x^{\alpha+2} g_\alpha(\partial_x^2 u^h))(t,y)\bigr|^2 \, \d y,
\]
so that $x^{\alpha+2} g_\alpha(\partial_x^2 u^h)$ is bounded in $L^2(K \times L)$ uniformly in $h$ as well. Furthermore, by part~\eqref{it:minty}, integration by parts gives for $j = 0, 1$,
\begin{align*}
\lim_{h \downarrow 0} \int_0^\infty \int_0^\infty (\partial_x^j(x^{\alpha+2} g_\alpha(\partial_x^2 u^h))) (\partial_x^{2-j} v) \, \d x \, \d t &= (-1)^j \lim_{h \downarrow 0} \int_0^\infty \int_0^\infty (\partial_x^2(x^{\alpha+2} g_\alpha(\partial_x^2 u^h))) v \, \d x \, \d t \\
&= (-1)^j \int_0^\infty \int_0^\infty (\partial_x^2(x^{\alpha+2} g_\alpha(\partial_x^2 u))) v \, \d x \, \d t \\
&= \int_0^\infty \int_0^\infty (\partial_x^j(x^{\alpha+2} g_\alpha(\partial_x^2 u))) (\partial_x^{2-j} v) \, \d x \, \d t. 
\end{align*}
The uniform $L^2(K \times L)$-bounds established above combined with this distributional identification against $C^\infty_\mathrm{c}((0,\infty)^2)$-test functions and the Banach-Alaoglu theorem yield
\[
\partial_x^j \bigl(x^{\alpha+2} g_\alpha(\partial_x^2 u^h)\bigr) \;\rightharpoonup\; \partial_x^j \bigl(x^{\alpha+2} g_\alpha(\partial_x^2 u)\bigr) \quad \text{in } L^2(K \times L), \quad j = 0, 1,
\]
for every compact subset $K \times L$ of $(0,\infty)^2$. Combined with the uniform convergence of $u^h$ on compact subsets from part~\eqref{it:uniform}, this entails
\begin{align*}
\lim_{h \downarrow 0} I_{2,h} &= 2 \int_0^\infty \int_0^\infty u (\partial_x(x^{\alpha+2} g_\alpha(\partial_x^2 u))) (\partial_x v) \, \d x \, \d t, \\
\lim_{h \downarrow 0} I_{3,h} &= \int_0^\infty \int_0^\infty u (x^{\alpha+2} g_\alpha(\partial_x^2 u)) (\partial_x^2 v) \, \d x \, \d t.
\end{align*}
Undoing the integrations by parts implies
\[
\lim_{h \downarrow 0} \int_0^\infty \int_0^\infty x^{\alpha+2} |\partial_x^2 u^h|^{\alpha+1} v \, \d x \, \d t = \int_0^\infty \int_0^\infty x^{\alpha+2} |\partial_x^2 u|^{\alpha+1} v \, \d x \, \d t
\]
for every $v \in C^\infty_\mathrm{c}((0,\infty)^2)$. Combined with $\bar u \in L^2((0,\infty)^2)$ and density of $C^\infty_\mathrm{c}((0,\infty)^2)$ in $L^2((0,\infty)^2)$, this identifies $\bar u = x^{\alpha+2} |\partial_x^2 u|^{\alpha+1}$ and proves~\eqref{it:weak2}.

\bigskip
\textbf{Proof of \eqref{it:rhs2}. }
For identifying the limit $f$ of the nonlinear term $f^h$, note that
\[
f^h \stackrel{\eqref{def-nu}}{=} (1-(1+\check u^h)^{\frac 3 2}) \partial_x^2(x^{\alpha+2} g_\alpha(\partial_x^2 \check u^h)) + (1+\check u^h)^{\frac 3 2} \partial_x^2((1-(1+\check u^h)^{\frac \alpha 2}) x^{\alpha+2} g_\alpha(\partial_x^2 \check u^h)),
\]
where $\check u^h(t,x) = u^h(t-h,x)$. On passing to a subsequence, again denoted by $\check u^h$, by part~\eqref{it:uniform} we have $\check u^h \to u$ uniformly on compact subsets of $[0,\infty)^2$, part~\eqref{it:weak} entails $x^{\frac{\alpha+2-j(\alpha+1)}{\alpha+1}} \partial_x^{2-j} \check u^h \rightharpoonup x^{\frac{\alpha+2-j(\alpha+1)}{\alpha+1}} \partial_x^{2-j} u$ in $L^{\alpha+1}((0,\infty)^2)$ as $h \downarrow 0$ for $j = 0,1,2$, part~\eqref{it:weak2} implies that $x^{\alpha+2} |\partial_x^2 \check u^h|^{\alpha+1} \rightharpoonup x^{\alpha+2} |\partial_x^2 u|^{\alpha+1}$ in $L^2((0,\infty)^2)$ as $h \downarrow 0$, and by part~\eqref{it:minty} we have $\partial_x^2(x^{\alpha+2} g_\alpha(\partial_x^2 \check u^h)) \rightharpoonup \partial_x^2(x^{\alpha+2} g_\alpha(\partial_x^2 u))$ in $L^2((0,\infty)^2)$ as $h \downarrow 0$. This entails
\begin{align*}
(1-(1+\check u^h)^{\frac 3 2}) \partial_x^2(x^{\alpha+2} g_\alpha(\partial_x^2 \check u^h)) &\rightharpoonup (1-(1+u)^{\frac 3 2}) \partial_x^2(x^{\alpha+2} g_\alpha(\partial_x^2 u)), \\
(1+\check u^h)^{\frac 3 2} (1-(1+\check u^h)^{\frac \alpha 2}) \partial_x^2(x^{\alpha+2} g_\alpha(\partial_x^2 \check u^h)) &\rightharpoonup (1+u)^{\frac 3 2} (1-(1+u)^{\frac \alpha 2}) \partial_x^2(x^{\alpha+2} g_\alpha(\partial_x^2 u)), \\
(1+\check u^h)^{\frac{\alpha+1}{2}} x^{\alpha+2} |\partial_x^2 \check u^h|^{\alpha+1} &\rightharpoonup (1+u)^{\frac{\alpha+1}{2}} x^{\alpha+2} |\partial_x^2 u|^{\alpha+1}
\end{align*}
in $L^2((0,\infty)^2)$ as $h \downarrow 0$.  Then, in order to establish $f^h \rightharpoonup N(u)$ in $L^2((0,\infty)^2)$ as $h \downarrow 0$, it remains to verify
\begin{align*}
(1+\check u^h)^{\frac{\alpha-1}{2}}
(\partial_x \check u^h)^2 x^{\alpha+2}
g_\alpha(\partial_x^2 \check u^h)
&\rightharpoonup (1+u)^{\frac{\alpha-1}{2}}
(\partial_x u)^2 x^{\alpha+2}
g_\alpha(\partial_x^2 u), \\
(1+\check u^h)^{\frac{\alpha+1}{2}}
(\partial_x \check u^h)
\partial_x(x^{\alpha+2}
g_\alpha(\partial_x^2 \check u^h))
&\rightharpoonup (1+u)^{\frac{\alpha+1}{2}}
(\partial_x u)
\partial_x(x^{\alpha+2}
g_\alpha(\partial_x^2 u))
\end{align*}
in $L^2((0,\infty)^2)$ as $h \downarrow 0$. 
Exhausting
$(0,\infty)^2$ by a countable number of
compact subsets $K \times L$ and using that
$\check u^h$ is bounded in
$L^\infty(K;H^2(L))$ by
estimate~\eqref{apriori-uh-nonlinear}
and $\partial_t \check u^h$ is bounded
in $L^2((0,\infty)^2)$
by~\eqref{apriori-uh-nonlinear},
we infer with the Aubin--Lions--Simon lemma
\cite[Corollary~4]{simon1986compact} that,
upon taking a subsequence, $\check u^h$
converges strongly in $C^0(K;C^1(L))$ to $u$. In particular, for each compact
$K \times L$,
the factors
$(1+\check u^h)^{\frac{\alpha-1}{2}}
(\partial_x \check u^h)^2$
and
$(1+\check u^h)^{\frac{\alpha+1}{2}}
(\partial_x \check u^h)$
converge uniformly on $K \times L$
to
$(1+u)^{\frac{\alpha-1}{2}}
(\partial_x u)^2$
and
$(1+u)^{\frac{\alpha+1}{2}}
(\partial_x u)$
respectively, as $h \downarrow 0$,
by part~\eqref{it:uniform} and the
Aubin--Lions--Simon lemma.
These strongly convergent factors
can be absorbed into the test function,
while the remaining factors
$x^{\frac{\alpha(\alpha+2)}{\alpha+1}}
g_\alpha(\partial_x^2 \check u^h)
\rightharpoonup
x^{\frac{\alpha(\alpha+2)}{\alpha+1}}
g_\alpha(\partial_x^2 u)$
converge weakly in
$L^{\frac{\alpha+1}{\alpha}}(K \times L)$
by part ~\eqref{it:minty}, and 
$\partial_x(x^{\alpha+2}
g_\alpha(\partial_x^2 \check u^h))
\rightharpoonup
\partial_x(x^{\alpha+2}
g_\alpha(\partial_x^2 u))$
converges weakly in $L^2(K \times L)$
by part~\eqref{it:minty}
and integration by parts (see for instance
arguments in the proof of part~\eqref{it:weak2}).
Hence $f^h \rightharpoonup N(u)$
in $L^2((0,\infty)^2)$ as $h \downarrow 0$,
by density of
$C^\infty_\mathrm{c}((0,\infty)^2)$
in $L^2((0,\infty)^2)$ and the uniform
bound on $f^h$ in $L^2((0,\infty)^2)$
from~\eqref{est-nv}
and~\eqref{apriori-uh-nonlinear}.
By part~\eqref{it:rhs}, it follows
$f = N(u)$  almost everywhere
on $(0,\infty)^2$. 

\bigskip
\textbf{Proof of \eqref{eq-u-tested}. }
Parts~\eqref{it:minty} and \eqref{it:rhs2} entail that \eqref{eq-u-weak} of part~\eqref{it:minty} upgrades to \eqref{eq-u-tested}.
\end{proof}

Before proving Theorem~\ref{th:ex}, we
establish a differential inequality that
will be applied in Step~4 of the proof of
Theorem~\ref{th:ex} below to deduce the
asymptotic decay estimate~\eqref{decay-main}.
\begin{lemma}\label{lem-der-uxx-int}
Let $\alpha > 2$ and let $u_0 \in U_\alpha$ with $\|u_0\|_{U_\alpha}$ sufficiently small. Suppose $u$ satisfies \eqref{apriori-u} and \eqref{eq-strong} of Theorem~\ref{th:ex}. Then 
\[
t \mapsto \frac{1}{\alpha+1}
\int_0^\infty x^{\alpha+2}
|\partial_x^2 u|^{\alpha+1}\,\d x
\]
is locally absolutely continuous on $[0,\infty)$ and there exists $C = C(\alpha) < \infty$ such that for
\[
\eps \coloneq C \Big(\int_0^\infty(u_0^2 + x^{\alpha+2}
|\partial_x^2 u_0|^{\alpha+1})\,\d x\Big)^{\frac{\alpha}{3\alpha-1}}
\]
it holds
%
\begin{equation}\label{eq-der-uxx-int}
\frac{\d}{\d t} \frac{1}{\alpha+1} \int_0^\infty x^{\alpha+2} |\partial_x^2 u|^{\alpha+1} \, \d x \le - (1-\eps) \int_0^\infty \big(\partial_x^2 (x^{\alpha+2} g_\alpha(\partial_x^2 u))\big)^2 \, \d x
\end{equation}
for almost every $t \in (0,\infty)$.
\end{lemma}
\begin{proof}
Define
$\psi\colon L^2(0,\infty) \to [0,\infty]$
by
\[
\psi(v) \coloneq \frac{1}{\alpha+1}
\int_0^\infty x^{\alpha+2}
|\partial_x^2 v|^{\alpha+1}\,\d x
\]
for $v \in U_\alpha$ and
$\psi(v) \coloneq +\infty$ otherwise.
Then $\psi$ is proper, convex, and
lower semi-continuous, and
$\partial_x^2(x^{\alpha+2}
g_\alpha(\partial_x^2 v))
\in \partial \psi(v)$
for $v \in V_\alpha$
by the convexity inequality
\[
|b|^{\alpha+1} - |a|^{\alpha+1}
\ge (\alpha+1)|a|^{\alpha-1}a(b-a) \quad \text{ for all } a, b \in \R
\]
and integrations by parts via the
density argument from the proof of
Lemma~\ref{lem:B2-AC}.
By \eqref{apriori-u} of
Theorem~\ref{th:ex},
\begin{equation*}
    u \in W^{1,2}(0,\infty;L^2(0,\infty)), 
    \quad 
    \psi \circ u \in L^\infty(0,\infty), \quad \text{and} \quad \partial_x^2(x^{\alpha+2}
    g_\alpha(\partial_x^2 u))
    \in L^2((0,\infty)^2).
\end{equation*}
Applying \cite[Lemma~3.3]{brezis1973} on each finite interval $[0,T]$ with $T \in (0,\infty)$ yields that $\psi \circ u$ is locally absolutely continuous on $[0,\infty)$  and the chain rule 
\[
\frac{\d}{\d t} \frac{1}{\alpha+1}
\int_0^\infty x^{\alpha+2}
|\partial_x^2 u|^{\alpha+1}\,\d x
= \int_0^\infty
\partial_x^2(x^{\alpha+2}
g_\alpha(\partial_x^2 u))\,
\partial_t u\,\d x
\quad \text{for almost every } t \in (0,\infty).
\]
Substituting \eqref{eq-strong},
by the Cauchy--Schwarz inequality together with
estimate~\eqref{est-nv} of
Proposition~\ref{prop-est-nv}
and~\eqref{apriori-u},
\[
\int_0^\infty
\partial_x^2(x^{\alpha+2}
g_\alpha(\partial_x^2 u))\,
N(u)\,\d x
\le \eps \int_0^\infty
\big(\partial_x^2(x^{\alpha+2}
g_\alpha(\partial_x^2 u))\big)^2
\,\d x,
\]
where $\eps \coloneq \sqrt{2 C_2(\alpha)}
K_0^{\frac{\alpha}{3\alpha-1}}$
with
$K_0 \coloneq C_1(\alpha)
\int_0^\infty(u_0^2 + x^{\alpha+2}
|\partial_x^2 u_0|^{\alpha+1})\,\d x
\ll_\alpha 1$,
and $C_1(\alpha), C_2(\alpha) \in (0,\infty)$ are the constants in the
a-priori estimate~\eqref{apriori-u} of Theorem~\ref{th:ex} and in
estimate~\eqref{est-nv} of Proposition~\ref{prop-est-nv}, respectively.
\end{proof}

We are now in position to prove our main theorem:
\begin{proof}[Proof of Theorem~\ref{th:ex}] We split the proof into several steps.

\bigskip
\textbf{Step~1. A-priori estimate and Cauchy problem. }
Passage to the limit $h \downarrow 0$ in \eqref{apriori-uh-nonlinear} entails on taking a subsequence and using Proposition~\ref{lem-uniform}
~\eqref{it:weak*}, \eqref{it:weak}, \eqref{it:time}, and \eqref{it:minty} that the a-priori estimate \eqref{apriori-u} is satisfied by weak lower semi-continuity of the norm. From \eqref{eq-u-tested} of Proposition~\ref{lem-uniform} on noting that $\partial_t u$, $\partial_x^2(x^{\alpha+2} g_\alpha(\partial_x^2 u)), N(u) \in L^2((0,\infty)^2)$, integration by parts entails together with density of test functions $C^\infty_\mathrm{c}([0,\infty) \times (0,\infty))$ in $L^2((0,\infty)^2)$ that \eqref{eq-strong} is satisfied, where the initial value $u(0,\cdot) = u_0$ is attained in the sense of continuous functions uniformly converging on compact subsets of $[0,\infty)$ (cf.~Proposition~\ref{lem-uniform}~\eqref{it:uniform}).

\bigskip
\textbf{Step~2: Initial data in $U_\alpha$. }
For $u_0 \in U_\alpha$, by Lemma~\ref{lem-approx-ualpha} we choose a sequence $u_0^{(k)} \in C^\infty_\mathrm{c}((0,\infty))$ such that
\[
\int_0^\infty \big((u_0 - u_0^{(k)})^2 + x^{\alpha+2} |\partial_x^2 u_0 - \partial_x^2 u_0^{(k)}|^{\alpha+1}\big) \, \d x \to 0 \quad \text{as} \quad k \to \infty.
\]
For each $k$, we find a solution $u^{(k)}$ satisfying the a-priori estimate \eqref{apriori-u} and the strong formulation \eqref{eq-strong}. Following the lines of the just proved compactness argument, Proposition~\ref{lem-uniform}, a subsequence of $u^{(k)}$ converges to some $u$ satisfying \eqref{apriori-u} and the tested and strong formulations \eqref{eq-u-tested} (cf.~Lemma~\ref{lem-uniform}) and \eqref{eq-strong}, respectively.

\bigskip
\textbf{Step~3: (H\"older) continuity in time. }
Using \eqref{eq-strong}, we can further deduce for $t > s \ge 0$
\begin{align*}
\|u(t) - u(s)\|_{L^2(0,\infty)} &\le \int_s^t \|\partial_t u(t')\|_{L^2(0,\infty)} \, \d t'
\\ &\le |t-s|^{\frac 1 2} \Big(2 \int_s^t \int_0^\infty \big(\big(\partial_x^2\big(x^{\alpha+2} g_\alpha(\partial_x^2 u)\big)\big)^2 + (N(u))^2\big) \, \d x \, \d t'\Big)^{\frac 1 2} \\
&\le \Big(2 C \int_0^\infty \big(u_0^2 + x^{\alpha+2} |\partial_x^2 u_0|^{\alpha+1}\big) \, \d x\Big)^{\frac 1 2} |t-s|^{\frac 1 2},
\end{align*}
where \eqref{apriori-u} and \eqref{est-nv} of Proposition~\ref{prop-est-nv} were used in the last inequality, thus proving H\"older continuity $u \in C^{\frac 1 2}([0,\infty);L^2(0,\infty))$. Continuity in $U_\alpha$ then follows by applying Lemma~\ref{lem-der-uxx-int} in conjunction with weak continuity in $U_\alpha$, which comes from the weak formulation of \eqref{eq-strong} together with the uniform $U_\alpha$-bound
from~\eqref{apriori-u}, and the Radon-Riesz property of $L^{\alpha+1}(0,\infty)$ \cite[Proposition~3.32]{brezis2011}.

\bigskip
\textbf{Step~4: Embeddings. }
We obtain \eqref{embed-main} from estimate~\eqref{eq:weight xu} of Lemma~\ref{lem:weight xu},
estimate~\eqref{eq:weight xu'} of Lemma~\ref{lem:weight xu'},
estimate~\eqref{sup-uxx-alt} of Lemma~\ref{lem-sup-uxx},
estimate~\eqref{sup-u-alt} of Lemma~\ref{lem-sup-u-alt},
estimate~\eqref{sup-ux-alt} of Lemma~\ref{lem-sup-ux-alt}, 
and the a-priori estimate \eqref{apriori-u}.

\bigskip
\textbf{Step~5: Asymptotic stability. }
Observe that integrability in time according to
\[
\int_0^\infty \int_0^\infty x^{\alpha+2} |\partial_x^2 u|^{\alpha+1} \, \d x \, \d t < \infty
\]
and the monotonicity
\[
\frac{\d}{\d t} \int_0^\infty x^{\alpha+2} |\partial_x^2 u|^{\alpha+1} \, \d x \le 0 \quad \text{for almost every } t \in (0,\infty)
\]
from inequality \eqref{eq-der-uxx-int} of Lemma~\ref{lem-der-uxx-int} entail $\int_0^\infty x^{\alpha+2} |\partial_x^2 u|^{\alpha+1} \, \d x = o(t^{-1})$
as $t \to \infty$, which is \eqref{decay-main}. Since $\sup_{t \ge 0} \int_0^\infty u^2 \, \d x$ is bounded by \eqref{apriori-u}, we infer from Lemma~\ref{lem:weight xu} that 
\[
\sup_{x \ge 0} x^{\beta_1} |u|
\stackrel{\eqref{eq:weight xu}}{=}
o\big(t^{\frac{2\beta_1-1}{3\alpha-1}}\big)
\quad \text{ as } t \to \infty,
\]
where
$\beta_1 \in \big[\frac{1-\alpha}{\alpha+1},
\frac{\alpha+2}{5\alpha+3}\big)$, which implies \eqref{decay-xbu}.
With the same argumentation, we infer from Lemma~\ref{lem:weight xu'}
that 
\[
\sup_{x \ge 0} x^{\beta_2}
|\partial_x u|
\stackrel{\eqref{eq:weight xu'}}{=}
o\big(t^{\frac{2\beta_2-3}{3\alpha-1}}\big)
\quad \text{ as } t \to \infty,
\]
where
$\beta_2 \in \big[\frac{2}{\alpha+1},
\frac{3(\alpha+2)}{5\alpha+3}\big)$,
which is~\eqref{decay-xbu'}.
\end{proof}
\begin{proof}[Proof of Theorem~\ref{th-velocity}]
From estimates~\eqref{sup-uxx-alt} and \eqref{est-coeff-c} of Lemma~\ref{lem-sup-uxx}, because $\alpha > 2$, and in conjunction with the a-priori estimate \eqref{apriori-u} of Theorem~\ref{th:ex} we infer that $x^{\frac{\alpha+1}{\alpha}} \partial_x^2 u$ is controlled in $L^{2\alpha}(0,\infty;BC^0([0,\infty))$ with
\[
\|x^{\frac{\alpha+1}{\alpha}} \partial_x^2 u\|_{L^{2\alpha}(0,\infty;BC^0([0,\infty)))}^{2\alpha} \lesssim_\alpha \int_0^\infty \big(u_0^2 + x^{\alpha+2} |\partial_x^2 u_0|^{\alpha+1}\big) \, \d x.
\]
By estimate~\eqref{eq:weight xu} of Lemma~\ref{lem:weight xu} and the a-priori estimate \eqref{apriori-u} of Theorem~\ref{th:ex}, we can further infer that $u$ is controlled in $BC^0([0,\infty)^2)$ with $\|u\|_{BC^0([0,\infty)^2)} \lesssim_\alpha 1$. This implies by \eqref{velocity} that $V$ is controlled in $L^2(0,\infty;BC^0([0,\infty)))$, in the sense that \eqref{a-priori-contact-vel} holds true.
\end{proof}

\bigskip

\bibliography{Gnann_Lienstromberg_Nik} 
\bibliographystyle{abbrv} 

\end{document}